\documentclass[11pt]{amsart}

\usepackage[utf8]{inputenc}

\usepackage{tikz}
\usetikzlibrary{matrix,arrows}

\usepackage{tikz-cd}

\title{Iterated Spans and Classical Topological Field Theories}
\usepackage{amsmath,amssymb,amsthm}
\usepackage[utf8]{inputenc}
\usepackage{url}
\usepackage[shortlabels]{enumitem}
  \setitemize[1]{leftmargin=2em}
  \setenumerate[1]{leftmargin=*}
\usepackage[colorlinks=true,linktocpage=true,citecolor=blue,unicode,hyperindex,breaklinks]{hyperref}

\theoremstyle{theorem}
\newtheorem{thm}{Theorem}[section]
\newtheorem{lemma}[thm]{Lemma}
\newtheorem{propn}[thm]{Proposition}
\newtheorem{cor}[thm]{Corollary}

\newtheorem*{thm*}{Theorem}
\newtheorem*{conjecture*}{Conjecture}
\newtheorem*{goal*}{Goal}
\newtheorem*{question*}{Question}
\newtheorem*{prethm*}{Pretheorem}
\theoremstyle{definition}
\newtheorem{defn}[thm]{Definition}
\newtheorem{ex}[thm]{Example}
\newtheorem{exs}[thm]{Examples}

\newtheorem{remark}[thm]{Remark}

\theoremstyle{remark}

\usepackage[alphabetic]{amsrefs}

\author{Rune Haugseng}
\address{University of Copenhagen, Copenhagen, Denmark}
\email{haugseng@math.ku.dk}
\urladdr{http://sites.google.com/site/runehaugseng/}

\date{\today}

\usepackage{palatino}
\usepackage{mathpazo}
\usepackage{eucal}

\makeatletter
\def\amsbb{\use@mathgroup \M@U \symAMSb}
\makeatother
\usepackage{mbboard}
\renewcommand{\mathbb}[1]{\amsbb{#1}}

\newcommand{\blank}{\text{\textendash}}

\newcommand{\defterm}[1]{\emph{#1}}
\newcommand{\isoto}{\xrightarrow{\sim}}
\newcommand{\isofrom}{\xleftarrow{\sim}}

\newcommand{\IFF}{if and only if}

\newcommand{\catname}[1]{\ensuremath{\text{\textup{#1}}}}
\newcommand{\txt}[1]{\ensuremath{\text{\textup{#1}}}}
\newcommand{\Set}{\catname{Set}}
\newcommand{\sSet}{\Set_{\Delta}}
\newcommand{\Cat}{\catname{Cat}}
\newcommand{\CatI}{\catname{Cat}_\infty}
\newcommand{\LCatI}{\widehat{\catname{Cat}}_\infty}

\newcommand{\Fun}{\txt{Fun}}

\newcommand{\Map}{\txt{Map}}
\newcommand{\Hom}{\txt{Hom}}

\newcommand{\op}{\txt{op}}

\newcommand{\icat}{$\infty$-category}
\newcommand{\icats}{$\infty$-categories}
\newcommand{\icatl}{$\infty$-categorical}

\newcommand{\xto}[1]{\xrightarrow{#1}}

\newcommand{\from}{\leftarrow}
\newcommand{\xfrom}[1]{\xleftarrow{#1}}

\newcommand{\csquare}[8]{ %
\[ %
\begin{tikzpicture} %
\matrix (m) [matrix of math nodes,row sep=3em,column sep=2.5em,text height=1.5ex,text depth=0.25ex] %
{ #1 \pgfmatrixnextcell #2 \\ %
  #3 \pgfmatrixnextcell #4 \\ }; %
\path[->,font=\footnotesize] %
(m-1-1) edge node[auto] {$#5$} (m-1-2)%
(m-1-1) edge node[left] {$#6$} (m-2-1)%
(m-1-2) edge node[auto] {$#7$} (m-2-2)%
(m-2-1) edge node[below] {$#8$} (m-2-2);%
\end{tikzpicture}%
\]%
}

\newcommand{\smallcsquare}[8]{ %
\[ %
\begin{tikzpicture} %
\matrix (m) [matrix of math nodes,row sep=1.5em,column sep=1.25em,text height=1.5ex,text depth=0.25ex] %
{ #1 \pgfmatrixnextcell #2 \\ %
  #3 \pgfmatrixnextcell #4 \\ }; %
\path[->,font=\footnotesize] %
(m-1-1) edge node[auto] {$#5$} (m-1-2)%
(m-1-1) edge node[left] {$#6$} (m-2-1)%
(m-1-2) edge node[auto] {$#7$} (m-2-2)%
(m-2-1) edge node[below] {$#8$} (m-2-2);%
\end{tikzpicture}%
\]%
}

\newcommand{\nolabelcsquare}[4]{\csquare{#1}{#2}{#3}{#4}{}{}{}{}}
\newcommand{\nolabelsmallcsquare}[4]{\smallcsquare{#1}{#2}{#3}{#4}{}{}{}{}}

\newcommand{\opctriangle}[6]{ %
\[ %
\begin{tikzpicture} %
\matrix (m) [matrix of math nodes,row sep=3em,column sep=1.2em,text height=1.5ex,text depth=0.25ex] %
{  #1 \pgfmatrixnextcell \pgfmatrixnextcell #2 \\ %
  \pgfmatrixnextcell #3 \pgfmatrixnextcell \\ %
}; %
\path[->,font=\footnotesize] %
(m-1-1) edge node[above] {$#4$} (m-1-3)%
(m-1-1) edge node[below left] {$#5$} (m-2-2)%
(m-1-3) edge node[below right] {$#6$} (m-2-2);%
\end{tikzpicture}%
\]%
}

\newcommand{\nolabelopctriangle}[3]{\opctriangle{#1}{#2}{#3}{}{}{}}

\newcommand{\id}{\txt{id}}

\DeclareMathOperator{\colimP}{colim}
\newcommand{\colim}{\mathop{\colimP}}

\newcommand{\simp}{\bbDelta}

\newcommand{\Seg}{\txt{Seg}}

\newcommand{\Alg}{\catname{Alg}}
\newcommand{\Mon}{\catname{Mon}}

\makeatletter
\def\@tocline#1#2#3#4#5#6#7{\relax
  \ifnum #1>\c@tocdepth 
  \else
    \par \addpenalty\@secpenalty\addvspace{#2}%
    \begingroup \hyphenpenalty\@M
    \@ifempty{#4}{%
      \@tempdima\csname r@tocindent\number#1\endcsname\relax
    }{%
      \@tempdima#4\relax
    }%
    \parindent\z@ \leftskip#3\relax \advance\leftskip\@tempdima\relax
    \rightskip\@pnumwidth plus4em \parfillskip-\@pnumwidth
    #5\leavevmode\hskip-\@tempdima
      \ifcase #1
       \or \hskip -1em \or \hskip 1em \or \hskip 3em \else \hskip 5em \fi%
      #6\nobreak\relax
    \hfill\hbox to\@pnumwidth{\@tocpagenum{#7}}
      \par
    \nobreak
    \endgroup
  \fi}
\makeatother

\newcommand{\Span}{\txt{Span}}
\newcommand{\SPAN}{\txt{SPAN}}
\newcommand{\oSPAN}{\overline{\txt{SPAN}}}

\newcommand{\Gpd}{\txt{Gpd}}
\newcommand{\Cart}{\txt{Cart}}

\newcommand{\dAff}{\txt{dAff}_{k}}
\newcommand{\dSt}{\txt{dSt}_{k}}
\newcommand{\dStA}{\dSt^{\txt{Art}}}
\newcommand{\Lag}{\txt{Lag}}

\newcommand{\bbS}{\bbSigma}
\newcommand{\bbSn}{\bbS^{n}}
\newcommand{\bbL}{\bbLambda}
\newcommand{\bbLn}{\bbL^{n}}
\newcommand{\Sympl}{\txt{Sympl}}
\newcommand{\Adj}{\txt{Adj}}
\newcommand{\Mor}{\txt{Mor}}
\newcommand{\MAP}{\txt{MAP}}
\newcommand{\Ob}{\txt{Ob}}

\newcommand{\Cell}{\txt{Cell}}

\newcommand{\itopos}{$\infty$-topos}
\newcommand{\itopoi}{$\infty$-topoi}
\newcommand{\CSS}{\txt{CSS}}

\newcommand{\spandiag}[5]{
  \[ %
\begin{tikzpicture} %
\matrix (m) [matrix of math nodes,row sep=1em,column sep=1em,text height=1.5ex,text depth=0.25ex] %
{      \pgfmatrixnextcell #1 \pgfmatrixnextcell       \\
  #2 \pgfmatrixnextcell   \pgfmatrixnextcell #3   \\};
\path[->,font=\footnotesize] %
(m-1-2) edge node[above left] {$#4$} (m-2-1)
(m-1-2) edge node[above right] {$#5$} (m-2-3);
\end{tikzpicture}%
\]%
}

\newcommand{\nolabelspandiag}[3]{\spandiag{#1}{#2}{#3}{}{}}

\renewcommand{\simp}{\bbDelta}

\begin{document}

\begin{abstract}
  We construct higher categories of iterated spans, possibly equipped with
  extra structure in the form of higher-categorical local systems, and classify
  their fully dualizable objects. By the Cobordism Hypothesis, these
  give rise to framed topological quantum field theories, which are
  the framed versions of the classical TQFTs considered in the
  quantization programme of Freed-Hopkins-Lurie-Teleman.

  Using this machinery, we also construct an $(\infty,1)$-category of
  symplectic derived algebraic stacks and Lagrangian correspondences
  and show that all its objects are dualizable.
\end{abstract}
\maketitle

\tableofcontents

\section{Introduction}
In this paper our main goal is to construct a fundamental family of
higher categories, namely the symmetric monoidal
$(\infty,n)$-categories $\Span_{n}(\mathcal{C})$ of iterated spans in
an $(\infty,1)$-category (or \icat{}) $\mathcal{C}$ with finite limits, and to classify the fully
dualizable objects in these $(\infty,n)$-categories. Via the Cobordism
Hypothesis, these objects correspond to the framed extended
topological quantum field theories (TQFTs) valued in the
$(\infty,n)$-categories $\Span_{n}(\mathcal{C})$, which can be
interpreted as a simple model for classical topological field
theories.

We also construct variants of these $(\infty,n)$-categories where the
spans are equipped with extra structure in the form of local
systems, and classify their fully dualizable objects. The
corresponding TQFTs are those proposed as classical field theories
in the quantization programme outlined by Freed, Hopkins, Lurie, and
Teleman~\cite{FreedHopkinsLurieTeleman}.

Finally, we apply our machinery in the context of
shifted symplectic structures on derived algebraic stacks, as
developed by Pantev, To\"{e}n,
Vaquié, and Vezzosi~\cite{PTVVShiftedSympl}, to construct $(\infty,1)$-categories
of $n$-shifted symplectic derived stacks with the morphisms given by
Lagrangian correspondences; we also show that here all objects are
dualizable (and thus determine framed 1-dimensional
TQFTs).

Before we describe our results in more detail and discuss how they
relate to extended TQFTs, we will first look at the analogous
classical field theories in the simpler setting of non-extended
TQFTs.

\subsection{Topological Quantum Field Theories and Spans}
\emph{Topological quantum field theories} (or TQFTs) originated in
physics as a particularly simple class of quantum field theories. They
were first formalized mathematically by Atiyah~\cite{AtiyahTQFT} in
the late 1980s, and have been the subject of much research over the
past two decades. A sketch of Atiyah's definition goes as
follows:
\begin{defn}
  Let $\txt{Cob}_{n}$ be the category with objects closed
  $(n-1)$-dimensional manifolds and morphisms diffeomorphism classes
  of $n$-dimensional cobordisms, i.e.\ a morphism from $M$ to $N$ is an
  $n$-manifold $X$ with an identification of $\partial X$ with $M
  \amalg N$; composition of morphisms is given by gluing along the
  boundary, and the disjoint union of manifolds makes this a symmetric
  monoidal category. If $\mathbf{C}$ is another symmetric monoidal
  category, a $\mathbf{C}$-valued $n$-dimensional \emph{topological
    quantum field theory} is a symmetric monoidal functor
  $\txt{Cob}_{n} \to \mathbf{C}$.
\end{defn}
Reflecting the linearity of quantum mechanics, in examples
$\mathbf{C}$ is typically a category of ``linear'' objects, for
example complex vector spaces or chain complexes of these.

Ideas from physics suggest that one should be able to produce
interesting examples of TQFTs as quantizations of classical
topological field theories --- this proposal goes back at least to
\cite{FreedHigherAlg}. In a very simple picture of a classical field
theory we assign to a manifold $M$ a collection $\mathcal{F}(M)$ of
\emph{fields} on $M$, which will typically be some form of stack. If $K$
is a cobordism from $M$ to $M'$, we can restrict the fields on $K$ to
the boundary, giving a \emph{span}
\spandiag{\mathcal{F}(K)}{\mathcal{F}(M)}{\mathcal{F}(M').}{s}{t}
Moreover, these fields should be \emph{local}: if $K$ is obtained from
cobordisms $K_{1}$ and $K_{2}$ by gluing them along a common boundary $N$, giving
a field on $K$ should be the same as giving fields on $K_{1}$ and
$K_{2}$ that agree on $N$. In other words, the stack $\mathcal{F}(K)$
should be the pullback $\mathcal{F}(K_{1}) \times_{\mathcal{F}(N)}
\mathcal{F}(K_{2})$. In this case the cobordism $K$ is the composite
of $K_{1}$ and $K_{2}$ in $\txt{Cob}_{n}$, so we want $\mathcal{F}$ to
be a functor from $\txt{Cob}_{n}$ to a category where the morphisms
are spans and the composition is given by taking pullbacks.

More precisely, if $\mathbf{C}$ is any category with finite limits, 
we would like to define a category $\Span(\mathbf{C})$ where the objects are the objects
of $\mathbf{C}$ and the morphisms from $X$ to $Y$ are spans
  \[ %
\begin{tikzpicture} %
\matrix (m) [matrix of math nodes,row sep=1em,column sep=0.5em,text height=1.5ex,text depth=0.25ex] %
{       & A &   \\
   X &       & Y   \\ };
\path[->,font=\footnotesize] %
(m-1-2) edge (m-2-1)
(m-1-2) edge (m-2-3);
\end{tikzpicture}%
\]
in $\mathbf{C}$. If $X \from A \to Y$ and $Y \from B
\to Z$ are two spans, their composite should be given by the pullback
  \[ %
\begin{tikzpicture} %
\matrix (m) [matrix of math nodes,row sep=1em,column sep=0.5em,text height=1.5ex,text depth=0.25ex] %
{  &       & A \times_{Y} B &       &  \\
   & A &       & B &   \\
X  &   & Y &       & Z.   \\ };
\path[->,font=\footnotesize] %
(m-1-3) edge (m-2-2)
(m-1-3) edge (m-2-4)
(m-2-2) edge (m-3-1)
(m-2-2) edge (m-3-3)
(m-2-4) edge (m-3-3)
(m-2-4) edge (m-3-5);
\end{tikzpicture}%
\]
However, taking pullbacks in this way is not strictly associative,
only associative up to canonical isomorphisms, so to get a category we
are forced to take morphisms in $\Span(\mathbf{C})$ to be isomorphism
classes of spans. The Cartesian product in $\mathbf{C}$ then gives
$\Span(\mathbf{C})$ a natural symmetric monoidal structure, and we can
think of a classical field theory as a TQFT valued in
$\Span(\mathbf{C})$ for some $\mathbf{C}$.

A quantization of such a classical field theory, which is supposed
to be analogous to the path integral of quantum field theory, would
then assign some algebraic object to the stack $\mathcal{F}(M)$, for
example the value of some cohomology theory $E^{*}\mathcal{F}(M)$, and
to the span \[\mathcal{F}(M) \xfrom{s} \mathcal{F}(K) \xto{t}
\mathcal{F}(M')\] a push-pull composite $t_{*}s^{*} \colon
E^{*}\mathcal{F}(M) \to E^{*}\mathcal{F}(K) \to E^{*}\mathcal{F}(M')$,
where the pushforward $t_{*}$ is thought of as integrating over
the fibres of $t$.

\subsection{Extended TQFTs and Iterated Spans}
Although relatively easy to define, Atiyah's notion of TQFTs suffers
from a number of defects, and recently much work has focused on the
more sophisticated notion of \emph{extended} topological quantum field theories.
This was first formulated in terms of \emph{$n$-categories} by Baez and
Dolan~\cite{BaezDolanTQFT}, building on earlier work by a number of
mathematicians, including Lawrence~\cite{LawrenceTQFT} and
Freed~\cite{FreedHigherAlg}. Roughly
speaking, an $n$-category is a structure that has objects, morphisms
between objects, 2-morphisms between morphisms, and so on up to
$n$-morphisms. For the definition of Baez and Dolan we want to
consider an $n$-category $\txt{Bord}_{n}$ whose objects are compact
$0$-manifolds, with morphisms given by 1-dimensional cobordisms
between 0-manifolds, and in general $i$-morphisms for $i = 1,\ldots,n$
given by $i$-dimensional cobordisms between manifolds with corners,
taking diffeomorphism classes of these for the $n$-morphisms. Given
such a symmetric monoidal $n$-category $\txt{Bord}_{n}$, with the
tensor product again given by taking disjoint unions, an
\emph{$n$-dimensional extended TQFT} valued in a symmetric monoidal
$n$-category $\mathcal{C}$ should be a symmetric monoidal functor
$\txt{Bord}_{n} \to \mathcal{C}$.

We would like to define an extended analogue of the classical
topological field theories we considered above. For example, if we
have an $n$-manifold $K$ whose boundary $\partial K$ is a manifold
with corners, i.e.\ we have a decomposition of $\partial K$ as $M
\cup_{A \amalg B} N$ where $M$ and $N$ are both $(n-1)$-manifolds with
boundary the $(n-2)$-manifold $A \amalg B$, then by restricting the
fields on $K$ we get a span of spans
\[
\begin{tikzcd}[row sep=scriptsize, column sep=tiny]
{}& & \mathcal{F}(K) \arrow{dl} \arrow{dr}  & &\\
& \mathcal{F}(M) \arrow{dl} \arrow{drrr}  && \mathcal{F}(N) \arrow[crossing over]{dlll} \arrow{dr}&\\
\mathcal{F}(A) & & & & \mathcal{F}(B),
\end{tikzcd}
\]
so we want to have such 2-fold spans as 2-morphisms in the
target. More precisely, if $\mathbf{C}$ is a category with finite
limits we'd like to construct a 2-category where the objects are the
objects of $\mathbf{C}$, the 1-morphisms are spans in $\mathbf{C}$,
and the 2-morphisms from $X \from A \to Y$ to $X \from B \to Y$ are
given by (isomorphism classes of) 2-fold spans, i.e.\ diagrams of the
form
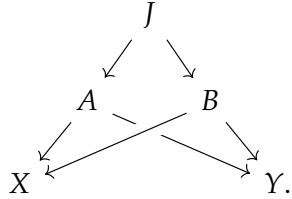
\[
\begin{tikzcd}[row sep=scriptsize, column sep=tiny]
{}& & J \arrow{dl} \arrow{dr}  & &\\
& A \arrow{dl} \arrow{drrr}  && B \arrow[crossing over]{dlll} \arrow{dr}&\\
X & & & & Y.\end{tikzcd}
\]

We can think of this 2-fold span as a span in the slice category
$\mathbf{C}_{/X \times Y}$ (whose objects are themselves spans from
$X$ to $Y$), which suggests that the general target for a classical
extended TQFT should be an $n$-category $\Span_{n}(\mathbf{C})$ where
an $i$-morphism between objects $X$ and $Y$ is inductively defined to
be an $(i-1)$-morphism in $\Span_{n-1}(\mathbf{C}_{/X \times
  Y})$. This $n$-category should also have a symmetric monoidal
structure induced by the Cartesian product in $\mathbf{C}$.

To give precise definitions of both of the $n$-categories
$\txt{Bord}_{n}$ and $\Span_{n}(\mathbf{C})$ we need to consider
\emph{weak} $n$-categories: as is typically the case for interesting
structures that we want to organize as $n$-categories, the composition
of (higher) morphisms is not strictly associative, only associative up
to a coherent choice of (specified) invertible higher
morphisms. Unfortunately, although the notion of weak $n$-categories
intuitively makes sense, making it precise becomes increasingly
intractable as $n$ increases.  For our $n$-categories we also want
symmetric monoidal structures, which introduces additional
complications --- a complete definition of $\txt{Bord}_{2}$ as a
symmetric monoidal 2-category has been given by
Schommer-Pries~\cite{SchommerPries2TQFT}, but for larger $n$ it seems
that an appropriate notion of symmetric monoidal $n$-category has not
even been defined.

We will therefore instead work with \emph{$(\infty,n)$-categories}. In
the same way as $n$-categories these have $i$-morphisms for $i =
1,\ldots,n$, but in addition they have \emph{invertible} $i$-morphisms
for $i > n$. This might seem an even more complicated notion to
rigorously formalize than that of $n$-categories, but it turns out
that by using homotopy theory we can set up notions of
$(\infty,n)$-category that are quite easy to work with in practice,
such as the iterated complete Segal spaces of Barwick~\cite{BarwickThesis}
(which is the model we will make use of in this paper) and the complete
$\Theta_{n}$-spaces of Rezk~\cite{RezkThetaN}. Heuristically, the
advantage of working with $(\infty,n)$-categories is that we can often
avoid dealing with coherence issues by not making any choices, e.g.\ of
compositions, and instead only need to check that the space of
possible choices is contractible.

Our first main result gives a construction of
$(\infty,n)$-categories of iterated spans with the properties we discussed
above:
\begin{thm}
  Suppose $\mathcal{C}$ is an \icat{} (i.e.\ $(\infty,1)$-category) with
  finite limits. Then there exists an $(\infty,n)$-category
  $\Span_{n}(\mathcal{C})$ such that:
  \begin{enumerate}[(i)]
  \item The objects of $\Span_{n}(\mathcal{C})$ are the objects of $\mathcal{C}$.
  \item The 1-morphisms of $\Span_{n}(\mathcal{C})$ are spans in
    $\mathcal{C}$.
  \item For $X$ and $Y \in \mathcal{C}$, the $(\infty,n-1)$-category
    of maps from $X$ to $Y$ in $\Span_{n}(\mathcal{C})$ is
    $\Span_{n-1}(\mathcal{C}_{/X \times Y})$.
  \item The $(\infty,n)$-category $\Span_{n}(\mathcal{C})$ has a
    natural symmetric monoidal structure induced by the Cartesian product in
  $\mathcal{C}$.
\end{enumerate}
\end{thm}
We'll construct these $(\infty,n)$-categories (in the form of $n$-fold
Segal spaces) in \S\ref{sec:spancat} and
prove that they are complete in \S\ref{sec:complitspan}, where we also
identify their mapping $(\infty,n-1)$-categories, and we construct the
symmetric monoidal structure in \S\ref{sec:dual}. The definition we
consider generalizes that of Barwick~\cite{BarwickQ} in the case $n =
1$.

Note that even if $\mathbf{C}$ is an ordinary category, the \icat{}
$\Span_{1}(\mathbf{C})$ will still have invertible 2-morphisms, given
by isomorphisms of spans --- since composing spans by taking pullbacks is
well-defined up to a canonical choice of such an invertible
2-morphism, we do not have to take isomorphism classes at the
top level.

We would like to consider these $(\infty,n)$-categories as targets for
extended TQFTs, with these also reformulated using the language
of $(\infty,n)$-categories. Such a definition was introduced by Lurie, 
and a sketch of it goes as follows:
\begin{defn}
  Let $\txt{Bord}_{n}$ be the $(\infty,n)$-category whose objects are
  closed $0$-manifolds and whose $i$-morphisms for $i \leq n$ are
  $i$-dimensional cobordisms between $(i-1)$-manifolds with corners,
  with $(n+1)$-morphisms being diffeomorphisms of such cobordisms,
  $(n+2)$-morphisms smooth homotopies of such diffeomorphisms, and so
  on. The disjoint union of manifolds gives this a symmetric monoidal
  structure, and if $\mathcal{C}$ is a symmetric monoidal $(\infty,n)$-category,
  a $\mathcal{C}$-valued $n$-dimensional \emph{extended topological quantum
  field theory} is a symmetric monoidal functor $\txt{Bord}_{n} \to
  \mathcal{C}$.
\end{defn}
In \cite{LurieCob} Lurie sketches a definition of the
symmetric monoidal $(\infty,n)$-categories $\txt{Bord}_{n}$, and a
complete construction has recently been carried out in full detail by
Calaque and Scheimbauer~\cite{CalaqueScheimbauerCob}.

Baez and Dolan conjectured that \emph{framed} extended TQFTs (where we
consider cobordisms that are equipped with a framing of their tangent
bundle) valued in any symmetric monoidal $n$-category $\mathbf{C}$ are
classified by the \emph{fully dualizable} objects in
$\mathbf{C}$. (Being fully dualizable is an inductively defined
algebraic notion; in the case $n =1$ it reduces to the usual notion of
dualizability for an object in a symmetric monoidal \icat{}. We will
review the general definition below in \S\ref{sec:adj}.)  This
conjecture is known as the \emph{Cobordism Hypothesis}. In
\cite{LurieCob}, Lurie introduced the natural generalization of the
Cobordism Hypothesis to $(\infty,n)$-categorical framed extended
TQFTs, and gave a detailed sketch of a proof thereof.

Our second main result gives a description of the fully dualizable
objects of the $(\infty,n)$-categories $\Span_{n}(\mathcal{C})$, and
thus, via the Cobordism Hypothesis, a classification of the framed
extended TQFTs with this target:
\begin{thm}
  Suppose $\mathcal{C}$ is an \icat{} with finite limits. Then all
  objects of the $(\infty,n)$-category $\Span_{n}(\mathcal{C})$ are
  fully dualizable with respect to the natural symmetric monoidal
  structure induced by the Cartesian product in $\mathcal{C}$.
\end{thm}
We'll show this in section \S\ref{sec:dual}. In fact, we'll show that
these $(\infty,n)$-categories \emph{have duals} in the sense of
\cite{LurieCob}, meaning that all the objects are dualizable and all
$i$-morphisms have left and right adjoints for all $i < n$.

\subsection{Iterated Spans with Local Systems}
In \cite{FreedHopkinsLurieTeleman}, Freed, Hopkins, Lurie, and Teleman
discuss extended TQFTs valued in $(\infty,n)$-categories where the
higher morphisms are iterated spans of spaces equipped with
\emph{local systems}. If $\mathcal{C}$ is an \icat{} (regarded as a
complete Segal space) and $X$ is a space, a \emph{$\mathcal{C}$-valued
  local system} on $X$ is just a functor $X \to \mathcal{C}$, or
equivalently a map of spaces $X \to \mathcal{C}_{0}$, where
$\mathcal{C}_{0}$ is the space of objects in $\mathcal{C}$. If we have
a span of spaces $X \from A \to Y$, we may consider a more elaborate
notion of local system on this span, namely a map of spans
\[
\begin{tikzcd}[row sep=tiny,column sep=tiny]
{}& A \arrow{rrr} \arrow{ddl} \arrow{dr} &   & &
\mathcal{C}_{1} \arrow{dr} & \\
  &   & Y \arrow{rrr} & & & \mathcal{C}_{0} \\
X \arrow{rrr}&   &   &\mathcal{C}_{0} \arrow[leftarrow,crossing over]{uur}& &
\end{tikzcd}
\]
where $\mathcal{C}_{1}$ is the space of morphisms in $\mathcal{C}$ and
the two maps $\mathcal{C}_{1} \to \mathcal{C}_{0}$ are the source and
target projections. Moreover, we can use the composition map in
$\mathcal{C}$ to compose such spans: given spans $X \from A \to Y$ and
$Y \from B \to Z$ over $\mathcal{C}_{0} \from \mathcal{C}_{1} \to
\mathcal{C}_{0}$, their composite is given by $X \from A \times_{Y} B
\to Z$ with the maps from $X$ and $Z$ to $\mathcal{C}_{0}$ as before,
but now equipped with the composite map
\[ A \times_{Y} B \to \mathcal{C}_{1} \times_{\mathcal{C}_{0}}
\mathcal{C}_{1} \to \mathcal{C}_{1},\] where the second map is the
composition in $\mathcal{C}$.  We can also use the map
$\mathcal{C}_{0}\to \mathcal{C}_{1}$ that assigns to objects their
identity maps to get identity maps for objects $X \to
\mathcal{C}_{0}$, so from the \icat{} $\mathcal{C}$ we should get a new
\icat{} where the objects are spaces with $\mathcal{C}$-valued local
systems and the morphisms are given by spans with local systems in
this sense. Increasing the category number, from an
$(\infty,n)$-category $\mathcal{C}$ we would expect to get an
$(\infty,n)$-category where the $k$-morphisms are $k$-fold spans of
spaces, equipped with a map to the $k$-fold span obtained from the
source and target maps from the space of $k$-morphisms in
$\mathcal{C}$ to the spaces of $i$-morphisms for all $i < k$. Our
third main result is a construction of such $(\infty,n)$-categories:
\begin{thm}
  Suppose $\mathcal{C}$ is an $(\infty,n)$-category. Then there is an $(\infty,n)$-category
  $\Span_{n}(\mathcal{S}; \mathcal{C})$ where the $k$-morphisms are $k$-fold spans of spaces
  with local systems in $\mathcal{C}$.
\end{thm}
We'll construct these $(\infty,n)$-categories in the form of $n$-fold
Segal spaces in \S\ref{sec:spanlocsys} and prove that they are
complete in \S\ref{sec:compllocsys}.

These are the $(\infty,n)$-categories considered as targets for
classical topological field theories by Freed, Hopkins, Lurie, and
Teleman, who propose that for good choices of $\mathcal{C}$ there
should be a symmetric monoidal linearization functor from
$\Span_{n}(\mathcal{S}; \mathcal{C})$, or at least from the
subcategory of spans of $\pi$-finite spaces, to $\mathcal{C}$.  We
will not construct any such linearizations here, but in
\S\ref{sec:duallocsys} we do describe (via the Cobordism Hypothesis)
the framed extended TQFTs with values in these $(\infty,n)$-categories:
\begin{thm}
  Suppose $\mathcal{C}$ is a symmetric monoidal
  $(\infty,n)$-category. Then $\Span_{n}(\mathcal{S}; \mathcal{C})$
  inherits a natural symmetric monoidal structure. Moreover, if
  $\mathcal{C}$ has duals, then so does $\Span_{n}(\mathcal{S};
  \mathcal{C})$.
\end{thm}
We'll also prove analogous results when $\mathcal{S}$ is replaced by
an arbitrary \itopos{} $\mathcal{X}$, with $\mathcal{C}$ an
$(\infty,n)$-category internal to $\mathcal{X}$.

\subsection{Lagrangian Correspondences}
To get a bit closer to the notion of a (non-extended) classical field
theory as considered in physics within the framework of TQFTs, we
would like to assign to a closed manifold $M$ a stack of fields
$\mathcal{F}(M)$ equipped with a \emph{symplectic structure}. The span
$\mathcal{F}(M) \from \mathcal{F}(X) \to \mathcal{F}(N)$ assigned to a
cobordism $X$ from $M$ to $N$ should then be a \emph{Lagrangian
  correspondence}, i.e.\ a Lagrangian morphism from $\mathcal{F}(X)$ to
$\mathcal{F}(M) \times \overline{\mathcal{F}(N)}$, where
$\overline{\mathcal{F}(N)}$ is $\mathcal{F}(N)$ equipped with the
negative of its symplectic form.  In \cite{PTVVShiftedSympl} Pantev,
To\"{e}n, Vaqui\'{e} and Vezzosi introduce a theory of symplectic
structures on derived algebraic stacks and Lagrangian morphisms
between them. The final main contribution in this paper is to make use
of this to construct \icats{} of symplectic derived stacks and
Lagrangian correspondences:
\begin{thm}
  The $n$-symplectic derived Artin stacks locally of finite
  presentation and the Lagrangian correspondences between them
  determine a subcategory $\Lag^{n}_{(\infty,1)}$ of $\Span_{1}(\dSt;
  \mathcal{A}^{2}_{\txt{cl}}[n])$, where $\dSt$ is the \itopos{} of
  derived stacks over a base field $k$ and
  $\mathcal{A}^{2}_{\txt{cl}}[n]$ is the derived stack of $n$-shifted
  closed 2-forms. Moreover, the \icat{} $\Lag^{n}_{(\infty,1)}$
  inherits a symmetric monoidal structure from $\Span_{1}(\dSt;
  \mathcal{A}^{2}_{\txt{cl}}[n])$ with respect to which all
  $n$-symplectic derived Artin stacks are dualizable.
\end{thm}
We will prove this in \S\ref{sec:lag}. This result partly generalizes
results of Calaque~\cite{CalaqueTFT} from the level of 1-categories to
\icats{}. Note that the \icat{} $\Lag^{n}_{(\infty,1)}$ can be viewed
as a derived algebro-geometric version of Weinstein's symplectic
``category'' \cite{WeinsteinSymplGeom,WeinsteinSymplCat}.

In a sequel to this paper, joint with Damien Calaque and Claudia
Scheimbauer, we will extend this by introducing a definition of
\emph{iterated} Lagrangian correspondences and using this to construct
$(\infty,n)$-categories of symplectic derived stacks.

\subsection{Related Work}
The classical TQFTs considered in this paper have previously been
discussed by a number of authors; particularly inspirational were the
accounts of Freed, Hopkins, Lurie, and
Teleman~\cite{FreedHopkinsLurieTeleman} and of
Calaque~\cite{CalaqueTFT}.

The construction of the \icat{} of spans in an \icat{} we use is due
to Barwick, who has made extensive use of this and variants of
it~\cite{BarwickQ,BarwickMackey,BarwickMackey2}. In unpublished work, Barwick has
also given an alternative definition of higher categories of iterated
spans, in the setting of Rezk's $\Theta_{n}$-spaces.

In their work \cite{DyckerhoffKapranovTwoSeg} on 2-Segal spaces,
Dyckerhoff and Kapranov introduce an alternative construction of an
$(\infty,2)$-category of spans.  I have also been informed that Lurie
has given a construction of the $(\infty,2)$-category of 2-fold spans
in the setting of scaled simplicial sets, though this is not currently
publicly available.

The idea that the $(\infty,n)$-category of iterated spans could most
easily be constructed as that underlying an $n$-uple \icat{} I gained
from the definition sketched by Schreiber in
\cite{SchreiberDiffCohlgyITopos}*{\S 3.9.14.2}. Schreiber and
collaborators have also extensively studied quantization by
linearizing iterated spans of stacks, for example in
\cite{SchreiberQuantLinHT,SchreiberCFTviaCohHT} and
\cite{SchreiberDiffCohlgyITopos}*{\S 3.9.14}; they consider not
necessarily topological quantum field theories valued in iterated
spans in a cohesive $\infty$-topos under the name \emph{local
  prequantum field theories}. Nuiten~\cite{NuitenMaTh} has also
recently studied the quantization of these.

Analogues of the higher categories we construct here have previously
been defined in low dimensions: a weak double category of 2-fold spans
in a category was constructed by  Morton~\cite{MortonDouble},
and a monoidal 3-category of spans in a 2-category was constructed by
Hoffnung~\cite{HoffnungSpans}, with the dualizability of its objects
subsequently proved by Stay~\cite{StayCompactClosed}. In the
1-categorical setting, a construction of weak $n$-fold categories
of iterated cospans in a category has been carried out by
Grandis~\cite{GrandisCospans}.

Morton has also studied extended TQFTs valued in 2-fold spans of
groupoids~\cite{MortonTVSpsGpds} and 2-fold spans of groupoids
equipped with $U(1)$-valued cocycles~\cite{MortonTwistTQFT}, and has
constructed linearization functors to linear categories (or \emph{2-vector
spaces}) in both cases.

Finally, the $(\infty,n)$-categories $\Span_{n}(\mathcal{S};
\mathcal{C})$ appear in Lurie's work on the cobordism hypothesis
\cite{LurieCob}, under the name $\txt{Fam}_{n}(\mathcal{C})$, but only
a sketch of a definition is given there.

\subsection{Overview}
We begin by reviewing some background on \icats{} in
\S\ref{sec:background}. We then briefly recall Rezk's theory of
(complete) Segal spaces in \S\ref{sec:Segsp} and its relationship to
that of \icats{}, before introducing the model of $(\infty,n)$-categories we
will use, namely iterated Segal spaces, in \S\ref{sec:segsp}. Then we
construct the $(\infty,n)$-category $\Span_{n}(\mathcal{C})$ of
iterated spans in an \icat{} $\mathcal{C}$ as an $n$-fold Segal space
in \S\ref{sec:spancat}, and the $(\infty,n)$-category
$\Span_{n}(\mathcal{X}; \mathcal{D})$ of iterated spans in an
\itopos{} $\mathcal{X}$ equipped with local systems in an
$(\infty,n)$-category $\mathcal{D}$ internal to $\mathcal{X}$ in
\S\ref{sec:spanlocsys}.

Next, we recall the definition of a \emph{complete} $n$-fold Segal space (and
its generalization to a general $\infty$-topos) and prove some
technical results about these in \S\ref{sec:completeseg}. We then show
that the $n$-fold Segal space $\Span_{n}(\mathcal{C})$ is complete in
\S\ref{sec:complitspan} and that $\Span_{n}(\mathcal{X}; \mathcal{D})$
is complete in \S\ref{sec:compllocsys}.

In \S\ref{sec:adj} we discuss the notion of (symmetric) monoidal
$(\infty,n)$-categories in the form they will appear later (and prove
these are equivalent to the definitions found in \cite{HA}), before we
review 
the notions of duals and adjoints in $(\infty,n)$-categories in
\S\ref{sec:adj}, where we also
generalize these to $(\infty,n)$-categories internal to an
\itopos{}. Then in \S\ref{sec:dual} we prove that
$\Span_{n}(\mathcal{C})$ is symmetric monoidal and that all its
objects are fully dualizable, and in \S\ref{sec:duallocsys} we show
the same holds for $\Span_{n}(\mathcal{X}; \mathcal{D})$ provided
$\mathcal{D}$ is a symmetric monoidal $(\infty,n)$-category in
$\mathcal{X}$ with duals.

Finally, in \S\ref{sec:lag} we construct an \icat{} of symplectic
derived algebraic stacks and Lagrangian correspondences, and prove
that all of its objects are dualizable.

\subsection{Acknowledgments}
I first learned about \icats{} of spans from conversations with Clark
Barwick back in 2010. The present work was inspired by a number of
discussions during my visit to the MSRI programme on algebraic
topology in the spring of 2014, in particular with Hiro Tanaka and
Owen Gwilliam. I also thank Oren Ben-Bassat, Damien Calaque, Theo
Johnson-Freyd, Gregor Schaumann, Claudia Scheimbauer, Chris
Schommer-Pries, and Peter Teichner for helpful comments.

\section{Background and Notation}\label{sec:background}
In this section we will describe our perspective on \icats{} and
recall some key constructions we'll make use of later on, in the hope
that this will make the paper easier to follow for readers who are not
intimately familiar with the literature on \icats{}. (For more
background on \icats{}, we recommend Groth's expository article
\cite{GrothInftyCourse} and Rezk's lecture notes \cite{RezkNotes}.)
At the end, we also describe some of our notational conventions.

\subsection{$\infty$-Categories}
As we mentioned above, the basic idea of an \emph{$\infty$-category}
is that this should be a structure that has objects and $i$-morphisms
between $(i-1)$-morphisms for $i = 1,2,\ldots$, where these are all
invertible for $i > 1$. Moreover, the composition of morphisms should
not be strictly associative, but only associative up to a compatible
choice of invertible higher morphisms. If we insist on making explicit
choices of composites and of the invertible higher morphisms specifying
the associativity data, we get a theory that is essentially
intractable, if we can make sense of it at all. A key idea in the
homotopical approaches to \icats{} is that it is better to not make
any choices, but instead consider the space of all possible choices,
which is contractible.

By far the best-developed implementation of this idea is that of
\emph{quasicategories}. A quasicategory is simply a simplicial set
satisfying the right lifting property for the \emph{inner horn
  inclusions} $\Lambda^{n}_{k} \hookrightarrow \Delta^{n}$ ($0 < k <
n$). The quasicategories are the fibrant objects in a model structure
on simplicial sets, due to Joyal; we will write $\sSet^{J}$ for this
model category. Quasicategories were first introduced by Boardman and
Vogt~\cite{BoardmanVogt} under the name \emph{restricted Kan complexes}, and
their theory has later been very extensively developed by
Joyal~\cite{JoyalUABNotes} and Lurie~\cite{HTT,HA}, to the extent that
most basic notions in category theory have analogues for
quasicategories, generally behaving ``as you would expect''.

Let us warn the reader that, when working in some quasicategory
$\mathfrak{C}$, we will use the same vocabulary for these
quasicategorical notions as we would use if $\mathfrak{C}$ were a
category. Thus if we speak of a commutative diagram in $\mathfrak{C}$
of shape $\mathfrak{I}$ we mean a functor of quasicategories (i.e.\ map
of simplicial sets) $\mathfrak{I} \to \mathfrak{C}$, even if
$\mathfrak{I}$ is (the nerve of) an ordinary category --- note that in
the latter case such a diagram includes the choice of homotopies in
all commuting triangles and so on for higher simplices. Thus if we say
we have a commutative square 
\nolabelcsquare{A}{B}{C}{D} in $\mathfrak{C}$, this implicitly
includes the data of a homotopy between the two composite maps $A \to
D$.

A key advantage of the quasicategorical model is that many important
constructions have simple combinatorial incarnations. For example, if
$\mathfrak{C}$ is a quasicategory then for any simplicial set the
internal Hom of simplicial sets $\mathfrak{C}^{K}$ is a quasicategory
(\cite{HTT}*{Corollary 2.3.2.5}) --- this represents the \icat{} of
functors from the \icat{} generated by $K$ (given by a fibrant
replacement in the Joyal model structure) to $\mathfrak{C}$. To
emphasize this, we will usually denote the internal Hom by $\Fun(K,
\mathfrak{C})$ when $\mathfrak{C}$ is a quasicategory.

It will be very convenient to apply the quasicategorical viewpoint
also to \icats{} themselves, i.e.\ we will think of them as living in a
quasicategory $\mathfrak{Cat}_{\infty}$ (which can be obtained as the
coherent nerve of a simplicial category of quasicategories) rather than in
the model category $\sSet^{J}$ of simplicial sets with the Joyal model
structure. This is a precise version of the somewhat vague idea of
working with \icats{} ``model-independently''. For most of the paper
this will allow us to avoid referring explicitly to the implementation
of \icats{} as simplicial sets; this typically lets us make
definitions and constructions that are more conceptual, thereby
hopefully making it easier for the reader to see what is actually
going on.

On a few occasions we will, however, need to make constructions in the
model category $\sSet^{J}$. To avoid confusion we will always be very
explicit about this: we will refer to objects of
$\mathfrak{Cat}_{\infty}$ as \emph{\icats{}} and fibrant objects of
$\sSet^{J}$ as quasicategories, and say things like ``let
$\mathfrak{C} \in \sSet$ be a quasicategory representing the \icat{}
$\mathcal{C}$''. Note also that, since ordinary categories form a full
subcategory of $\mathfrak{Cat}_{\infty}$, we will not distinguish notationally between a
category $\mathbf{C}$ and its image in $\mathfrak{Cat}_{\infty}$ --- on the other hand,
when we think of it as living in $\sSet^{J}$ via its nerve
$\mathrm{N}\mathbf{C}$ we will explicitly indicate this.

In fact, though it may seem slightly perverse at first sight, it is pleasant to
take this to the next level: we want to be able to work with the
\icat{} (rather than the quasicategory) of \icats{}, so we take
$\CatI$ to be the \emph{\icat{}} represented by the
\emph{quasicategory} $\mathfrak{Cat}_{\infty}$. (Here we are
implicitly passing to a larger Grothendieck universe of sets.)

If $\mathcal{C}$ is an \icat{}, we write $\Map_{\mathcal{C}}(x,y)$ for
the space of maps from $x$ to $y$ in $\mathcal{C}$. Using the
quasicategory model there are a
variety of (weakly equivalent) ways to describe these mapping spaces
as simplicial sets; see \cite{DuggerSpivakMap} for an extensive
discussion and comparisons.

We will now briefly review some key concepts from the theory of
\icats{} that we will make repeated use of in this paper. We will
mainly describe them from our \icatl{} perspective, but we will also
mention how they can be implemented via quasicategories.

\subsection{$\infty$-Groupoids}
Just as a groupoid is a category where all the morphisms are
invertible, an \emph{$\infty$-groupoid} is an \icat{} all of whose
morphisms are invertible. Grothendieck's \emph{Homotopy Hypothesis}
asserts that $\infty$-groupoids are equivalent to homotopy types. For
the homotopical approaches to higher categories that we are concerned
with here, this idea is taken as a starting point for the theory; in
the case of quasicategories, $\infty$-groupoids correspond to those
quasicategories that are \emph{Kan complexes},which are of course a
well-known model for homotopy types. (Moreover, the weak equivalences
in $\sSet^{J}$ restrict to the usual weak equivalences between Kan
complexes.) Given this equivalence, we will often refer to
$\infty$-groupoids as just \emph{spaces}.

We write $\mathcal{S}$ for the \icat{} of $\infty$-groupoids or
spaces. As a quasicategory, this is modelled by the coherent nerve
$\mathrm{N}(\sSet^{\circ})$ of the simplicial category $\sSet^{\circ}$
of Kan complexes (which is the full subcategory spanned by the Kan
complexes in the simplicial category $\sSet$ of simplicial sets).

\begin{defn}
  The fully faithful inclusion $\mathcal{S} \hookrightarrow \CatI$ has
  both a left and a right adjoint. We write $\iota$ for the right
  adjoint, which takes an \icat{} to its underlying $\infty$-groupoid
  (obtained by forgetting the non-invertible morphisms) and
  $\|\blank\|$ for the left adjoint, obtained by inverting all the
  morphisms in an \icat{}.
\end{defn}
\begin{remark}
  If $\mathfrak{C}$ is a quasicategory representing an \icat{}
  $\mathcal{C}$, then the $\infty$-groupoid $\|\mathcal{C}\|$ is
  represented by any Kan complex that is a fibrant replacement for the
  simplicial set $\mathfrak{C}$ in the usual Kan-Quillen model
  structure on $\sSet$.
\end{remark}

\begin{defn}
  If $\mathcal{C}$ is an \icat{}, we say that $\mathcal{C}$ is
  \emph{weakly contractible} if the space $\|\mathcal{C}\|$ is
  contractible.
\end{defn}

We call a functor $\mathcal{C}^{\op} \to \mathcal{S}$ a
\emph{presheaf} on $\mathcal{C}$, and write $\mathcal{P}(\mathcal{C})$
for the \icat{} $\Fun(\mathcal{C}^{\op}, \mathcal{S})$ of presheaves.

\subsection{Joins and Slice $\infty$-Categories}
We'll frequently make use of slice \icats{}, which can be defined as follows:
\begin{defn}
  If $\mathcal{C}$ is an \icat{} and $x$ is an object of
  $\mathcal{C}$, then the \emph{overcategory} $\mathcal{C}_{/x}$ is
  defined by the pullback square \csquare{\mathcal{C}_{/x}}{\Fun([1],
    \mathcal{C})}{\{x\}}{\mathcal{C},}{}{}{\txt{ev}_{1}}{} where
  $\txt{ev}_{1}$ is the functor given by evaluation at $1 \in
  [1]$.
  Similarly, if $p \colon \mathcal{I} \to \mathcal{C}$ is a functor,
  we define $\mathcal{C}_{/p}$ by the pullback square
  \csquare{\mathcal{C}_{/p}}{\Fun([1] \times \mathcal{I},
    \mathcal{C})}{\mathcal{C} \times \{p\}}{\Fun(\mathcal{I},
    \mathcal{C}) \times \Fun(\mathcal{I},
    \mathcal{C}).}{}{}{(\txt{ev}_{0},\txt{ev}_{1})}{\txt{const} \times
    *} Undercategories are of course also defined analogously.
\end{defn}
By the universal property of pullbacks, for any \icat{}
$\mathcal{D}$ we have a pullback square
\nolabelcsquare{\Fun(\mathcal{D}, \mathcal{C}_{/p})}{\Fun(\mathcal{D}
  \star \mathcal{I}, \mathcal{C})}{\{p\}}{\Fun(\mathcal{I},
  \mathcal{C}),}  where the \emph{join}
$\mathcal{D} \star \mathcal{I}$ is defined as follows:
\begin{defn}
If $\mathcal{C}$ and $\mathcal{D}$ are \icats{}, their \emph{join}
$\mathcal{C} \star \mathcal{D}$ is defined by the pushout
\[ \mathcal{C} \amalg_{\mathcal{C} \times \mathcal{D} \times \{0\}}
\mathcal{C} \times \mathcal{D} \times [1] \amalg_{\mathcal{C} \times
  \mathcal{D} \times \{1\}} \mathcal{D}.\]
in $\CatI$. The \emph{cones} $\mathcal{C}^{\triangleleft}$ and
$\mathcal{C}^{\triangleright}$ are then defined as $[0] \star
\mathcal{C}$ and $\mathcal{C} \star [0]$, respectively.
\end{defn}

If $\mathfrak{C}$ is a quasicategory representing $\mathcal{C}$, the
functor $\txt{ev}_{1} \colon \Fun([1], \mathcal{C}) \to \mathcal{C}$
can be represented by $\txt{ev}_{1} \colon \Fun(\Delta^{1},
\mathfrak{C}) \to \mathfrak{C}$. This is a fibration in the Joyal
model structure (combine \cite{HTT}*{Corollary 2.4.7.12} with
\cite{HTT}*{Corollary 2.4.6.5}) and so if we define (using the
notation of \cite{HTT}*{\S 4.2.1}) the simplicial set
$\mathfrak{C}^{/x}$ by the pullback
\csquare{\mathfrak{C}^{/x}}{\Fun(\Delta^{1},
  \mathfrak{C})}{\{x\}}{\mathfrak{C},}{}{}{\txt{ev}_1}{} then this is
a homotopy pullback square. Thus the quasicategory $\mathfrak{C}^{/x}$
represents $\mathcal{C}_{/x}$. Moreover, this has the universal
property that for any simplicial set $K$, we have a pullback square
\nolabelcsquare{\Fun(K, \mathfrak{C}^{/x})}{\Fun(K \diamond
  \Delta^{0}, \mathfrak{C})}{\{x\}}{\mathfrak{C},} 
where for simplicial sets $K$ and $L$, the simplicial set $K \diamond
L$ is defined as the pushout
\[ K \amalg_{K \times L \times \{0\}}
K \times L \times \Delta^{1}
\amalg_{K \times L \times \{1\}} L.\]
Since this is a homotopy pushout in the Joyal model structure, if
$\mathfrak{C}$ and $\mathfrak{D}$ are quasicategories representing
\icats{} $\mathcal{C}$ and $\mathcal{D}$, then the simplicial set
$\mathfrak{C} \diamond \mathfrak{D}$ represents $\mathcal{C} \star
\mathcal{D}$.

However, the simplicial set $\mathfrak{C} \diamond \mathfrak{D}$ is
generally not a quasicategory. It is therefore often convenient to use
instead an alternative model for the join by using the \emph{join of
  simplicial sets} (see \cite{HTT}*{Definition 1.2.8.1}), which we
also denote using $\star$: if $\mathfrak{C}$ and $\mathfrak{D}$ are
quasicategories, the join $\mathfrak{C} \star \mathfrak{D}$ is a
quasicategory by \cite{HTT}*{Proposition 1.2.8.3} and the natural
map
$\mathfrak{C} \diamond \mathfrak{D} \to \mathfrak{C} \star
\mathfrak{D}$ is a weak equivalence by \cite{HTT}*{Proposition
  4.2.1.2}. (In other words, $\mathfrak{C} \star
\mathfrak{D}$ is a fibrant replacement for $\mathfrak{C} \diamond
\mathfrak{D}$ in the Joyal model structure.)

We can use this model for the join to get an alternative quasicategory
representing $\mathcal{C}_{/x}$, by replacing
$K \diamond \Delta^{0}$ by $K \star \Delta^{0}$ in the pullback
square above. The resulting quasicategory is denoted
$\mathfrak{C}_{/x}$ in \cite{HTT}. There is
a natural map $\mathfrak{C}^{/x} \to \mathfrak{C}_{/x}$ (induced by
the natural inclusion $K \diamond \Delta^{0} \to K \star \Delta^{0}$)
and this is a weak equivalence in the Joyal model structure by
\cite{HTT}*{Proposition 4.2.1.5}. An analogous discussion applies, of
course, to quasicategories $\mathfrak{C}^{/p}$ and $\mathfrak{C}_{/p}$
as well as to the dual notions for undercategories.

\subsection{Cartesian and CoCartesian Fibrations}
Cartesian fibrations are the \icatl{} analogue of \emph{Grothendieck
  fibrations}; in both cases they are characterized by the existence
of \emph{Cartesian morphisms}:
\begin{defn}
  If $F \colon \mathcal{E} \to \mathcal{B}$ is a functor of \icats{},
  we say a morphism $\bar{f} \colon x \to y$ in $\mathcal{E}$ with
  image $f \colon a \to b$ in $\mathcal{B}$ is \emph{$F$-Cartesian} if
  for every $z \in \mathcal{E}$ over $c \in \mathcal{B}$ the
  commutative square \csquare{\Map_{\mathcal{E}}(z,
    x)}{\Map_{\mathcal{E}}(z, y)}{\Map_{\mathcal{B}}(c,
    a)}{\Map_{\mathcal{B}}(c, b)}{\bar{f}_*}{}{}{f_*} is Cartesian. 
  Equivalently (since the mapping space $\Map_{\mathcal{E}}(z, x)$ is
  the fibre at $z$ of the forgetful functor $\mathcal{C}_{/x} \to
  \mathcal{C}$) $f$ is $F$-Cartesian if the commutative square
  \nolabelcsquare{\mathcal{E}_{/x}}{\mathcal{E}_{/y}}{\mathcal{B}_{/a}}{\mathcal{B}_{/b}}
  is Cartesian.
  We say $F$ is a \emph{Cartesian fibration} if for every morphism $f
  \colon a \to b$ in $\mathcal{B}$ and every $y \in \mathcal{E}$ with
  $F(y) \simeq b$ there exists an $F$-Cartesian morphism $\bar{f}
  \colon y \to x$ with $F(\bar{f}) \simeq f$. The notions of
  \emph{co}Cartesian morphisms and \emph{co}Cartesian fibrations are
  defined dually, i.e.\ $F$ is a coCartesian fibration \IFF{} $F^{\op}$
  is a Cartesian fibration.
\end{defn}

\begin{remark}
  It can be shown that if $p \colon \mathfrak{E} \to \mathfrak{B}$ is
  an inner fibration (meaning $p$ has the right lifting property for
  the inner horn inclusions $\Lambda^{n}_{k} \hookrightarrow
  \Delta^{n}$) representing a functor of \icats{} $F \colon
  \mathcal{E} \to \mathcal{B}$, then $F$ is a Cartesian fibration in
  our sense \IFF{} $p$ is a Cartesian fibration in the sense of
  \cite{HTT}*{Definition 2.4.2.1}. This is not entirely obvious ---
  see \cite{MazelGeeCart} for a detailed proof. To avoid confusion we
  will use the term \emph{(co)Cartesian inner fibration} for a
  (co)Cartesian fibration of simplicial sets in Lurie's sense.
\end{remark}

Grothendieck proved (in \cite{SGA1}) that Grothendieck opfibrations
over a category $\mathbf{C}$ correspond to (pseudo)functors from
$\mathbf{C}$ to the category of categories. Lurie's
\emph{straightening equivalence} from \cite{HTT}*{\S 3.2} establishes
an analogous equivalence between Cartesian fibrations over an \icat{}
$\mathcal{C}$ and functors from $\mathcal{C}^{\op}$ to the \icat{}
$\CatI$ of \icats{}. If we let
$\Cat_{\infty/\mathcal{C}}^{\txt{Cart}}$ denote the subcategory of
$\Cat_{\infty/\mathcal{C}}$ whose objects are the Cartesian fibrations
and whose morphisms are the maps that preserve Cartesian morphisms,
then this gives the following statement in our language:
\begin{thm}[Lurie]
  There is an equivalence of \icats{}
  $\Cat_{\infty/\mathcal{C}}^{\txt{Cart}} \simeq
  \Fun(\mathcal{C}^{\op}, \CatI)$.
\end{thm}
This is extremely useful, as it is in practice impossible to ``write
down'' functors to $\CatI$, whereas we can much more easily describe
(co)Cartesian fibrations, e.g.\ by manipulating preexisting fibrations.

\begin{remark}
  In ordinary category theory, the Grothendieck fibration associated
  to a functor is given by the \emph{Grothendieck construction}, which
  can be identified with the lax colimit (a certain weighted colimit)
  of the functor. In the \icatl{} setting, the functor
  $\Fun(\mathcal{C}^{\op},\CatI) \to \CatI$ that takes a functor to
  the source of the associated Cartesian fibration can also be
  identified with the lax colimit, by \cite{freepres}*{Corollary 7.6}.
\end{remark}

\subsection{Left and Right Fibrations}
An important special case of Cartesian fibrations are those whose
fibres are spaces (as opposed to general \icats{}); these are called
\emph{right fibrations}. Dually, coCartesian fibrations whose fibres
are spaces are called \emph{left fibrations}. Right fibrations can
also be characterized as those functors $\mathcal{E} \to \mathcal{B}$
such that \emph{every} morphism in $\mathcal{E}$ is Cartesian.

\begin{remark}
  The terms \emph{left} and \emph{right} fibrations are motivated by
  the incarnations of these concepts on the level of quasicategories:
  If $p \colon \mathfrak{E} \to \mathfrak{B}$ is an inner fibration
  representing a functor of \icats{} $F \colon \mathcal{E} \to
  \mathcal{B}$, then $F$ is a right fibration \IFF{} $p$ satisfies a
  very simple condition: $p$ must have the right lifting property
  with respect to the horn inclusions $\Lambda^{n}_{k} \hookrightarrow
  \Delta^{n}$ for $0 < k \leq n$. Similarly, $F$ is a left fibration
  \IFF{} $p$ has the right lifting property with respect to the horn
  inclusions where $0 \leq k < n$.
\end{remark}

An easier version of Lurie's straightening equivalence for Cartesian
fibrations gives an equivalence of
\icats{} 
\[\Cat_{\infty/\mathcal{C}}^{\txt{RFib}} \simeq
\Fun(\mathcal{C}^{\op}, \mathcal{S})\] where
$\Cat_{\infty/\mathcal{C}}^{\txt{RFib}}$ is the full subcategory of
$\Cat_{\infty/\mathcal{C}}$ spanned by the right fibrations.

\subsection{Cofinal and Coinitial Functors}
We will often need to know that objects defined as (co)limits over
diagrams of different, but related, shapes are equivalent. Just as in
ordinary category theory, the notion of cofinal functors (and the dual
notion of coinitial functors) is a very useful tool for proving such
statements. 
\begin{defn}
  A functor $F \colon \mathcal{A} \to \mathcal{B}$ of \icats{} is
  \emph{cofinal} if for every diagram $p \colon
  \mathcal{B} \to \mathcal{C}$, the induced functor $\mathcal{C}_{p/}
  \to \mathcal{C}_{p\circ F/}$ is an equivalence. Dually, $F$ is
  \emph{coinitial} if $F^{\op} \colon \mathcal{A}^{\op} \to
  \mathcal{B}^{\op}$ is cofinal, i.e.\ the functor $\mathcal{C}_{/p}
  \to \mathcal{C}_{/p \circ F}$ is an equivalence for every
  $p$. 
\end{defn}
Since a colimit of $p$ is the same thing as a final object in
$\mathcal{C}_{p/}$, we see that if $F$ is cofinal then $p$ has a
colimit \IFF{} $p \circ F$ has a colimit, and these colimits are
necessarily given by the same object in $\mathcal{C}$. 

The key criterion for cofinality is \cite{HTT}*{Theorem 4.1.3.1}:
\begin{thm}[\cite{HTT}*{Theorem 4.1.3.1}]
  A functor $F \colon \mathcal{A} \to \mathcal{B}$ is cofinal \IFF{}
  for every $b \in \mathcal{B}$ the slice \icat{} $\mathcal{A}_{b/} :=
  \mathcal{A} \times_{\mathcal{B}} \mathcal{B}_{b/}$ is weakly
  contractible.
\end{thm}

\subsection{$\infty$-Topoi}
Just as a (Grothendieck) topos is a category that abstracts some key
properties of the category of sets, an \emph{$\infty$-topos} is an
\icat{} with key properties of the \icat{} $\mathcal{S}$ of spaces. A
terse definition is that $\infty$-topoi are the \icats{} that arise as
left exact accessible localizations (meaning the localization functor
preserves finite limits and sufficiently compact objects) of presheaf
\icats{} $\Fun(\mathcal{C}^{\op}, \mathcal{S})$ where $\mathcal{C}$ is
a small \icat{} with finite limits; see \cite{HTT}*{Theorem 6.1.0.6,
  Proposition 6.1.5.3} for several equivalent characterizations.

The key examples of $\infty$-topoi are \icats{} of presheaves
$\mathcal{P}(\mathcal{C})$ (where $\mathcal{C}$ has finite limits) and
\icats{} of sheaves of spaces on topological spaces, or more generally
on sites. The latter are important in the context of derived algebraic
geometry.

\begin{remark}
  At a number of points below we will prove results for a general
  $\infty$-topos $\mathcal{X}$, as the proofs are no more work once
  certain definitions have been set up. This is motivated by the
  possibility of applications in derived algebraic geometry, but apart
  from in \S\ref{sec:lag} we do not make use of this generality in
  this paper. The reader should therefore feel free to assume that
  $\mathcal{X}$ is just the \icat{} $\mathcal{S}$ of spaces, and to
  omit the parts of \S\ref{sec:completeseg} and \S\ref{sec:adj} where
  certain results and constructions are generalized from spaces to an
  arbitrary $\infty$-topos.
\end{remark}

\subsection{Notation}
We generally reuse the notation and terminology used by
Lurie in \cite{HTT,LurieCob,HA}. We note the following
conventions, some of which differ slightly from those of Lurie:
\begin{itemize}
\item $\simp$ is the simplicial indexing category, with objects the
  non-empty finite totally ordered sets $[n] := \{0, 1, \ldots, n\}$
  and morphisms order-preserving functions between them.
\item For $[n] \in \simp$ we will abbreviate $(\simp_{/[n]})^{\op}$ to
  $\simp_{/[n]}^{\op}$, and for $I \in \simp^{k}$ we will abbreviate
  $(\simp^{k}_{/I})^{\op}$ to $\simp^{k,\op}_{/I}$.
\item If $\mathcal{C}$ is an \icat{}, we write $\iota \mathcal{C}$ for
  the \emph{interior} or \emph{underlying space} of $\mathcal{C}$,
  i.e.\ the largest subspace of $\mathcal{C}$ that is a Kan complex.
\item If $f \colon \mathcal{C} \to \mathcal{D}$ is left adjoint to a
  functor $g \colon \mathcal{D} \to \mathcal{C}$, we will refer to the
  adjunction as $f \dashv g$.
\item We make use of Grothendieck universes to avoid having to deal
  with set-theoretical size issues: we fix three nested universes and
  refer to their elements as small, large, and very large sets,
  respectively. If we have \icats{} of small and large versions of the
  same objects, we will distinguish the large version with a
  circumflex: thus $\mathcal{S}$ is the (large) \icat{} of small
  spaces and $\widehat{\mathcal{S}}$ is the (very large) \icat{} of
  large spaces; similarly $\CatI$ is the \icat{} of small \icats{} and
  $\LCatI$ that of large \icats{}.
\end{itemize}

\section{Segal Spaces}\label{sec:Segsp}
In this section we review the description of \icats{} as complete
Segal spaces. This was introduced by Rezk in \cite{RezkCSS}, though we
will discuss it from an \icatl{} perspective rather than the
model-categorical one used by Rezk.

\begin{defn}
  Suppose $\mathcal{C}$ is an \icat{} with finite limits. A
  \emph{category object} in $\mathcal{C}$ is a simplicial object $C_{\bullet}
  \colon \simp^{\op} \to \mathcal{C}$ such that the natural maps
  \[ C_{n} \to C_{1} \times_{C_{0}} \cdots \times_{C_{0}} C_{1},\]
  induced by the maps $\sigma_{i}\colon [0] \to [n]$ sending $0$ to
  $i$ and $\rho_{i} \colon [1] \to [n]$ sending $0$ to $i-1$ and $1$
  to $i$, are equivalences in $\mathcal{C}$ for all $n$. We
  write $\Cat(\mathcal{C})$ for the full subcategory of
  $\Fun(\simp^{\op}, \mathcal{C})$ spanned by the category objects.
\end{defn}

\begin{defn}
  A \emph{Segal space} is a category object in the \icat{}
  $\mathcal{S}$ of spaces. We write $\Seg(\mathcal{S})$ for the
  \icat{} of Segal spaces.
\end{defn}

A category object in $\Set$ is the same thing as an ordinary category.
Similarly, a Segal space $X$ describes the algebraic structure of an
\icat{}: We think of the space $X_{0}$ as the space of objects and
$X_{1}$ as the space of morphisms; the two face maps
$X_{1} \rightrightarrows X_{0}$ assign the source and target object to
each morphism, and the degeneracy $s_{0} \colon X_{0} \to X_{1}$
assigns an identity morphism to every object. Then
$X_{n}\simeq X_{1} \times_{X_{0}} \cdots \times_{X_{0}} X_{1}$ is the
space of composable sequences of $n$ morphisms, and the face map
$d_{1} \colon [1] \to [2]$ gives a composition
\[ X_{1}\times_{X_{0}} X_{1} \isofrom X_{2} \xto{d_{1}} X_{1}.\] The
remaining data in $X_{\bullet}$ gives the homotopy-coherent
associativity data for this composition and its compatibility with the
identity maps.

However, although category objects in $\Set$ are categories,
isomorphisms in the category $\Cat(\Set)$ do not give the right notion
of equivalence of categories: to describe the correct homotopy theory
of categories we must invert the fully faithful and essentially
surjective functors. If done in an \icatl{} (or at least
2-categorical) setting, this produces the (2,1)-category of
categories, functors, and natural isomorphisms. An equivalence here,
in the 2-categorical sense, is precisely an equivalence of categories.
The same phenomenon occurs for \icats{}: Segal spaces encode the
algebraic structure of composition and units in \icats{}, but
the right notion of equivalence between \icats{} corresponds to the
(non-algebraic) notion of fully faithful and essentially surjective
morphisms of Segal spaces, in the following sense:
\begin{defn}
  Let $E^{n}$ denote the contractible groupoid with $n$ objects and a unique
  morphism between any pair of objects; we can regard this as a Segal
  space by thinking of it as a category object in sets and applying
  the inclusion $\Set \hookrightarrow \mathcal{S}$. For $X \in
  \Seg(\mathcal{S})$ we define a simplicial space by $\iota_{\bullet}X
  := \Map_{\Seg(\mathcal{S})}(E^{\bullet}, X)$; we write $\iota X$ for
  the colimit of this simplicial diagram --- this is the
  \emph{classifying space of equivalences} in $X$. We say a morphism
  $f \colon X \to Y$ is \emph{fully faithful and essentially
    surjective} if
  \begin{enumerate}[(1)]
  \item The map $\iota X \to \iota Y$ is an
    equivalence of spaces.
  \item The diagram \nolabelcsquare{X_{1}}{Y_{1}}{X_{0}\times
      X_{0}}{Y_{0}\times Y_{0},} with the vertical maps coming from
    the two maps $[0] \to [1]$, is a pullback square.
  \end{enumerate}
\end{defn}

To get the correct \icat{} of \icats{} we need to localize the \icat{}
$\Seg(\mathcal{S})$ at the fully faithful and essentially surjective
morphisms. The main result of \cite{RezkCSS} is that this localization
is given by the full subcategory of the \emph{complete} Segal spaces:
\begin{defn}
  A Segal space $X$ is \emph{complete} if the natural map $X_{0} \to
  \iota X$ is an equivalence.
\end{defn}

\begin{thm}[Rezk, \cite{RezkCSS}*{Theorem 7.7}]
  Let $\txt{CSS}(\mathcal{S})$ denote the full subcategory of
  $\Seg(\mathcal{S})$ spanned by the complete Segal spaces. Then the
  inclusion $\txt{CSS}(\mathcal{S}) \hookrightarrow \Seg(\mathcal{S})$
  has a left adjoint, which exhibits $\txt{CSS}(\mathcal{S})$ as the
  localization of $\Seg(\mathcal{S})$ at the fully faithful and
  essentially surjective morphisms.
\end{thm}

Rezk also proves some alternative characterizations of complete
objects, which we recall for use later on:
\begin{thm}[Rezk]\label{thm:rezkcompl}
  Let $X$ be a Segal space. There are two obvious inclusions $[1] \to
  E^{1}$ which induce maps $\Map(E^{1}, X) \to X_{1}$. We write
  $X_{\txt{eq}}$ for the subspace of $\mathcal{C}_{1}$ consisting of
  the components in the image of either of these maps (they
  are the same since $E^{1}$ has an autoequivalence that swaps the
  two). Then the map $\Map(E^{1}, X) \to X_{\txt{eq}}$ is an equivalence,
  and the following are equivalent:
  \begin{enumerate}[(i)]
  \item $X$ is complete.
  \item The simplicial space $\iota_{\bullet}X$ is constant.
  \item The map $X_{0} \to X_{\txt{eq}}$ induced
    by the degeneracy map $s_{0}$ is an equivalence.
  \item The map $X_{0} \to \Map(E^{1}, X)$ induced
    by composition with either of the maps $[0] \to E^{1}$ is an
    equivalence.
  \end{enumerate}
\end{thm}
\begin{proof}
  This is Theorem 6.2 and Proposition 6.4 of \cite{RezkCSS}.
\end{proof}

By analogy with the case of ordinary categories, we would expect that
the \icat{} $\txt{CSS}(\mathcal{S})$ is equivalent to $\CatI$. This is
indeed true, as was proved by Joyal and Tierney:
\begin{thm}[Joyal-Tierney \cite{JoyalTierney}]
  The functor $\simp \to \CatI$ given by the usual inclusion of
  ordered sets into categories induces, via the Yoneda embedding, a
  functor $\CatI \to \mathcal{P}(\simp) := \Fun(\simp^{\op},
  \mathcal{S})$. This is fully faithful and its essential image
  consists precisely of the complete Segal spaces. In other words it
  restricts to an equivalence $\CatI \isoto \txt{CSS}(\mathcal{S})$.
\end{thm}

It will be useful to introduce a reformulation of the definition of a
category object. To give this we must first introduce some notation:
\begin{defn}\label{defn:inert}
  A map $\phi \colon [n] \to [m]$ is \emph{inert} if $\phi$ is the
  inclusion of a subinterval, i.e.\ we have $\phi(i) = \phi(0)+i$ for
  all $i$. We write $\simp_{\txt{int}}$ for the subcategory of $\simp$
  containing only the inert maps. Let $\Cell^{1}$ denote the full
  subcategory of $\simp_{\txt{int}}$ spanned by the objects
  $[0]$ and $[1]$, i.e.\ the category
  \[ [0] \rightrightarrows [1].\] 
  For $[n] \in \simp$ we write $\Cell^{1}_{/[n]}$ for the category 
  $\Cell^{1} \times_{\simp_{\txt{int}}} (\simp_{\txt{int}})_{/[n]}$ of
  inert maps from $[0]$ and $[1]$ to $[n]$.
\end{defn}
\begin{remark}
  This is a special case of the general notion of an inert map defined
  by Barwick~\cite{BarwickOpCat} in the context of \emph{operator
    categories}, which can be viewed as settings for different kinds
  of algebraic structures.
\end{remark}

\begin{remark}
  The category $\Cell^{1}_{/[n]}$ can be depicted as
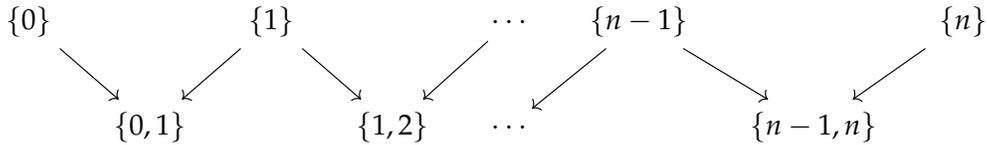
\[  \begin{tikzcd}[column sep=small]
\{0\} \arrow{dr} & & \{1\} \arrow{dl} \arrow{dr} & & \cdots \arrow{dl}& \{n-1\} \arrow{dl} \arrow{dr} & & \{n\}\arrow{dl}\\
 & \{0,1\} & & \{1,2\} & \cdots & &\{n-1,n\} &
  \end{tikzcd}
\]
\end{remark}

\begin{lemma}\label{lem:segspkanext}
  Let $\mathcal{C}$ be an \icat{} with finite limits. A simplicial
  object $X \colon \simp^{\op} \to \mathcal{C}$ is a category object
  \IFF{} its restriction $X|_{\simp^{\op}_{\txt{int}}}$ is the right
  Kan extension of its restriction to $\Cell^{1,\op}$, or in other words,
  if $j$ denotes the inclusion $\Cell^{1,\op} \to
  \simp^{\op}_{\txt{int}}$, the unit map
  \[ X|_{\simp^{\op}_{\txt{int}}} \to
  j_{*}j^{*}X|_{\simp^{\op}_{\txt{int}}}\] of the right Kan extension
  adjunction $j^{*}\dashv j_{*}$ is an equivalence.
\end{lemma}
\begin{proof}
  The functor $X|_{\simp^{\op}_{\txt{int}}}$ is the right Kan
  extension of its restriction to $\Cell^{1,\op}$ \IFF{} for every
  object $[n] \in \simp^{\op}_{\txt{int}}$ the natural map
  \[ X([n]) \to \lim_{([i] \to [n]) \in (\Cell^{1}_{/[n]})^{\op}}
  X([i]) \] is an equivalence. But from the definition of $\Cell^{1}$
  we see that this is precisely the limit that appears in the
  definition of a Segal space.
\end{proof}

\section{Iterated Segal Spaces}\label{sec:segsp}
In this section we briefly review the model for
$(\infty,n)$-categories we use in this paper: the iterated Segal
spaces of Barwick~\cite{BarwickThesis}. Following
\cite{LurieGoodwillie} we'll state the basic definitions using the
language of \icats{}.

\begin{defn}
  An \emph{$n$-fold category object} in an \icat{} $\mathcal{C}$ is
  inductively defined to be a category object in the \icat{} of
  $(n-1)$-fold category objects. We write $\Cat^{n}(\mathcal{C}) :=
  \Cat(\Cat^{n-1}(\mathcal{C}))$ for the \icat{} of $n$-fold category
  objects in $\mathcal{C}$. We refer to an $n$-fold category object in
  $\mathcal{S}$ as an \emph{$n$-uple Segal space}.
\end{defn}

\begin{remark}
  The term $n$-uple Segal space is motivated by the observation that
  2-uple (or double) Segal spaces encode the algebraic structure of
  double \icats{}, i.e.\ category objects in $\CatI$. More generally,
  $n$-uple Segal spaces can be considered as a model for $n$-uple
  \icats{}, i.e.\ internal \icats{} in internal \icats{} in \ldots\ in
  \icats{}. 
\end{remark}

\begin{remark}
  Unwinding the definition, we see that an $n$-uple Segal space
  $\mathcal{D} \colon (\simp^{\op})^{\times n}\to \mathcal{S}$
  consists of the data of:
  \begin{itemize}
  \item a space $\mathcal{D}_{0,\ldots,0}$ of objects
  \item spaces $\mathcal{D}_{1,0,\ldots,0}$, \ldots,
    $\mathcal{D}_{0,\ldots,0,1}$ of $n$ different kinds of 1-morphism,
        each with a source and target in $\mathcal{D}_{0,\ldots,0}$,
  \item spaces $\mathcal{D}_{1,1,0,\ldots,0}$, etc., of ``commutative
    squares'' between any two kinds of 1-morphism,
  \item spaces $\mathcal{D}_{1,1,1,0,\ldots,0}$, etc., of
    ``commutative cubes'' between any three kinds of 1-morphism,
  \item \ldots
  \item a space $\mathcal{D}_{1,1,\ldots,1}$ of ``commutative
    $n$-cubes'',
\end{itemize}
together with units (from the degeneracies in $(\simp^{\op})^{\times
  n}$) and coherently homotopy-associative composition laws (from the
face maps) for all these different types of maps.
\end{remark}
We can view the algebraic structure of an $(\infty,n)$-category as given by the same kind of data,
except that there is only one type of 1-morphism, etc., so we require
certain spaces to be ``degenerate'', i.e.\ equivalent to the space
$\mathcal{D}_{0,\ldots,0}$ via a degeneracy. This leads to Barwick's definition
of an $n$-fold Segal object in an \icat{}:
\begin{defn}
  Suppose $\mathcal{C}$ is an \icat{} with finite limits. A
  \emph{1-fold Segal object} in $\mathcal{C}$ is just a category object in $\mathcal{C}$. For $n > 1$ we
  inductively define an \emph{$n$-fold Segal object} in $\mathcal{C}$ to be an $n$-fold
  category object $\mathcal{D}$ such that
  \begin{enumerate}[(i)]
  \item the $(n-1)$-fold category object
    $\mathcal{D}_{0,\bullet,\ldots,\bullet}$ is constant,
  \item the $(n-1)$-fold category object
    $\mathcal{D}_{k,\bullet,\ldots,\bullet}$ is an $(n-1)$-fold Segal
    object for all $k$.
  \end{enumerate}
  We write $\Seg_{n}(\mathcal{C})$ for the full subcategory of
  $\Cat^{n}(\mathcal{C})$ spanned by the $n$-fold Segal objects. When
  $\mathcal{C}$ is the \icat{} $\mathcal{S}$ of spaces, we refer to
  $n$-fold Segal objects in $\mathcal{S}$ as \emph{$n$-fold Segal spaces}.
\end{defn}
\begin{remark}
  Unwinding the definition, we see that an $n$-fold Segal space $\mathcal{D}$ consists of
\begin{itemize}
\item a space $\mathcal{D}_{0,\ldots,0}$ of objects,
\item a space $\mathcal{D}_{1,0,\ldots,0}$ of 1-morphisms,
\item a space $\mathcal{D}_{1,1,0,\ldots,0}$ of 2-morphisms,
\item \ldots
\item a space $\mathcal{D}_{1,\ldots,1}$ of $n$-morphisms,
\end{itemize}
together with units and coherently homotopy-associative composition
laws for these morphisms.
\end{remark}

\begin{remark}
  The notion of $n$-fold Segal spaces describes precisely the
  \emph{algebraic} structure we expect from
  $(\infty,n)$-categories. Just as in the case of Segal spaces, to get
  the right homotopy theory of $(\infty,n)$-categories we must
  supplement this algebraic structure with the (``non-algebraic'')
  notion of fully faithful and essentially surjective functors. Thus
  the \icat{} of $(\infty,n)$-categories is obtained from
  $\Seg_{n}(\mathcal{S})$ by inverting the fully faithful and
  essentially surjective morphisms. As in the case $n = 1$, this
  localization can be obtained by restricting to a full subcategory of
  \emph{complete} objects; we will discuss this below in
  \S\ref{sec:completeseg}.
\end{remark}

It will be useful to restate the definition of an $n$-fold category object
non-inductively, via the analogue of Lemma~\ref{lem:segspkanext}. To
state this we first need some notation:
\begin{defn}
  We say an object $I = ([i_{1}],\ldots,[i_{k}]) \in \simp^{k}$ is a
  \emph{cell} if $i_{j} = 0$ or $1$ for all $i$. We write $\Cell^{k}$
  for the full subcategory of $\simp^{k}_{\txt{int}} :=
  (\simp_{\txt{int}})^{\times k}$ spanned by the cells, i.e.\ $(\Cell^{1})^{\times
    k}$. For $I \in \simp^{k}$ we write $\Cell^{k}_{/I}$ for the
  pullback $\Cell^{k} \times_{\simp^{k}_{\txt{int}}}
  (\simp_{\txt{int}}^{k})_{/I}$; if $I = ([i_{1}], \ldots, [i_{k}])$
  then this category is equivalent to $\Cell^{1}_{/[i_{1}]} \times
  \cdots \times \Cell^{1}_{/[i_{k}]}$.
\end{defn}

\begin{lemma}\label{lem:kupSegKanExt}
  A functor $\Phi \colon \simp^{k,\op} \to \mathcal{C}$ is
  a $k$-fold category object \IFF{} the restriction
  $\Phi|_{\simp^{k,\op}_{\txt{int}}}$ is the right Kan extension
  of its restriction to $\Cell^{k,\op}$, or in other words, if $j^{k}$
  denotes the inclusion $j^{k} \colon \Cell^{k,\op}
  \hookrightarrow \simp^{k,\op}_{\txt{int}}$, the unit
  map \[\Phi|_{\simp^{k,\op}_{\txt{int}}} \to
  j^{k}_{*}j^{k,*}\Phi|_{\simp^{k,\op}_{\txt{int}}}\] is an
  equivalence.
\end{lemma}
\begin{proof}
  We prove this by induction on $n$ --- the case $n = 1$ being
  Lemma~\ref{lem:segspkanext}. The functor
  $\Phi|_{\simp^{k,\op}_{\txt{int}}}$ is the right Kan extension of
  its restriction to $\Cell^{k,\op}$ \IFF{} for every object $I =
  ([i_{1}],\ldots,[i_{k}])\in \simp^{k,\op}_{\txt{int}}$ the natural
  map
  \[ \Phi(I) \to \lim_{(C \to I) \in (\Cell^{k}_{/I})^{\op}}
  \Phi(C) \] is an equivalence. Since $\Cell^{k}_{/I}$ is the product
  $\Cell^{1}_{/[i_{1}]} \times \Cell^{k-1}_{/I'}$, where $I' =
  ([i_{2}],\ldots,[i_{k}])$, this limit can (e.g.\ using
  \cite{enrbimod}*{Corollary 5.7}) be rewritten as the iterated limit
  \[ \lim_{([j] \to [i_{1}]) \in (\Cell^{1}_{/[i_{1}]})^{\op}}
  \lim_{(J' \to I') \in (\Cell^{k-1}_{/I'})^{\op}} \Phi([j], J')).\]
  Taking $i_{1} = 0,1$ (where $(\Cell^{1}_{/[i_{1}]})^{\op}$ has an
  initial object) we see by the inductive hypothesis that the
  condition holds in these cases
  \IFF{} $\Phi([0], \blank)$ and $\Phi([1], \blank)$ are $(k-1)$-fold
  category objects. Moreover, if the condition holds in these cases we
  can rewrite the limit for a general $I$ as \[ \lim_{([j] \to
    [i_{1}]) \in (\Cell^{1}_{/[i_{1}]})^{\op}} \Phi([j], I').\] Thus
  we have that $\Phi|_{\simp^{k,\op}_{\txt{int}}}$ is the right Kan
  extension of its restriction to $\Cell^{k,\op}$ \IFF{} $\Phi([0],
  \blank)$ and $\Phi([1], \blank)$ are $(k-1)$-fold category objects, and
  $\Phi$ is a category object in $\Fun(\simp^{k-1,\op},
  \mathcal{C})$. Since $(k-1)$-fold category objects are closed under
  limits, this is equivalent to $\Phi$ being a $k$-fold category
  object in $\mathcal{C}$.
\end{proof}

\begin{defn}
  We write $C_{k}$ for the \emph{k}-cell, i.e.\ the generic
  $k$-morphism, thought of as an $n$-fold Segal space for any $n >
  k$. Concretely, it is the representable $n$-fold simplicial object
  represented by $([1],\ldots,[1],[0],\ldots,[0])$ where $[1]$ occurs
  $k$ times. If $\mathcal{D}$ is an $n$-fold Segal space, we
  write $\Mor_{k}(\mathcal{D})$ for the space
  $\Map(C_{k},\mathcal{D})$ of $k$-morphisms in $\mathcal{D}$,
  i.e.\ $\mathcal{D}_{1,\ldots,1,0,\ldots,0}$. For $k = 0$ we also
  write $\Ob(\mathcal{D})$ for $\mathcal{D}_{0,\ldots,0}$.
\end{defn}

\begin{defn}
  Suppose $\mathcal{C}$ is an $n$-fold Segal space. The two face maps
  $[1] \to [0]$ induce a map of $(n-1)$-fold Segal spaces from
  $\mathcal{C}_{1}$ to the constant $(n-1)$-fold Segal space
  $\mathcal{C}_{0} \times \mathcal{C}_{0}$. Given two objects $X, Y$
  of $\mathcal{C}$, i.e.\ a point of $\mathcal{C}_{0,\ldots,0} \times
  \mathcal{C}_{0,\ldots,0}$, we define the \emph{mapping
    $(\infty,n-1)$-category} $\mathcal{C}(X,Y)$ to be the pullback
  \nolabelcsquare{\mathcal{C}(X,Y)}{\mathcal{C}_{1}}{\{(X,Y)\}}{\mathcal{C}_{0}^{\times
      2}.} 
  Since $(n-1)$-fold Segal objects are closed under limits
  in $(n-1)$-fold simplicial spaces, this is again an $(n-1)$-fold
  Segal space.
\end{defn}

\begin{defn}
  Suppose $\mathcal{C}$ is an $n$-fold Segal object in
  $\mathcal{X}$. The \emph{underlying $k$-fold Segal object}
  $u_{(\infty,k)}\mathcal{C}$ of $\mathcal{C}$ is the $k$-fold
  simplicial object obtained by restricting $\mathcal{C}$ along the
  inclusion $\simp^{k,\op} \to \simp^{n,\op}$ that is $[0]$ in the
  last $n-k$ components.
\end{defn}

Our next goal is to prove that there is a canonical way to extract an
$n$-fold Segal space from an $n$-uple Segal space; in the next section
we will apply this to construct an $n$-fold Segal space of iterated
spans from an $n$-uple Segal space.

\begin{propn}\label{propn:underfold}
  Let $\mathcal{C}$ be an \icat{} with finite limits. The inclusion $\Seg_{n}(\mathcal{C}) \hookrightarrow
  \Cat^{n}(\mathcal{C})$ has a right adjoint $U^{n}_{\Seg} \colon
  \Cat^{n}(\mathcal{C}) \to \Seg_{n}(\mathcal{C})$. (We will usually
  abbreviate $U^{n}_{\Seg}$ to just $U_{\Seg}$ and let the integer $n$
  be implicitly determined by the context.)
\end{propn}

\begin{remark}
  The basic idea, in the case $n = 2$, is that for $X$ a double
  category object we form $U_{\Seg}X$ by taking pullbacks of
  simplicial objects
  \nolabelcsquare{(U_{\Seg}X)_{m,\bullet}}{X_{m,\bullet}}{X_{0,0}^{\times
      (m+1)}}{X_{0,\bullet}^{\times (m+1)},} where $X_{0,0}$ denotes
  the constant simplicial space with this value, the right vertical
  map is induced by the inert maps $\rho_{i} \colon [1] \to [m]$, and
  the bottom horizontal map is given by the degeneracies
  $X_{0,0} \to X_{0,k}$. This pullback extracts, for example,
  precisely the part of $X_{1,1}$ that is ``constant'' or degenerate
  in the second coordinate, i.e.\ lies over the image of the degeneracy
  $X_{0,0}^{\times 2} \to X_{0,1}^{\times 2}$. We can also think of
  this as given by a single pullback of bisimplicial objects : if,
  given a simplicial object $C_{\bullet}$, we write $i_{*}C$ for the
  bisimplicial object given by
  $(i_{*}C)_{n,\bullet} \simeq C_{\bullet}^{\times (n+1)}$, with the
  face maps in the first coordinate given by projections and the
  degeneracies by diagonal maps, then $U_{\Seg}X$ is given by the
  pullback of bisimplicial objects
  \nolabelcsquare{U_{\Seg}X}{X}{i_{*}X_{0,0}}{i_{*}X_{0,\bullet},}
  where again $X_{0,0}$ denotes the constant simplicial object with
  this value. For $n$-fold category objects we then want to iterate
  this procedure.
\end{remark}

To make this idea rigorous we first prove a sequence of easy lemmas:
\begin{lemma}\label{lem:fullsubradj}
  Suppose $\pi \colon \mathcal{E} \to \mathcal{C}$ is a Cartesian
  fibration and $j \colon \mathcal{C}_{0} \to \mathcal{C}$
  is a functor with a right adjoint $r \colon
  \mathcal{C} \to \mathcal{C}_{0}$. Let
  \csquare{\mathcal{E}_{0}}{\mathcal{E}}{\mathcal{C}_{0}}{\mathcal{C}}{J}{\pi_{0}}{\pi}{j}
  be a pullback square. Then the functor $J$ has a right adjoint $R
  \colon \mathcal{E} \to \mathcal{E}_{0}$ such that the counit map
  $JR(X) \to X$ is a $\pi$-Cartesian morphism over the counit map
  $jr\pi(X) \to \pi(X)$.
\end{lemma}
\begin{proof}
  By \cite{HTT}*{Proposition 3.1.2.1}, the functor of \icats{} $\Fun(\mathcal{E}, \mathcal{E}) \to \Fun(\mathcal{E},
  \mathcal{C})$ induced by composition with $\pi$ is a Cartesian
  fibration, and a morphism $\alpha \colon F \to G$ in
  $\Fun(\mathcal{E}, \mathcal{E})$ is Cartesian \IFF{} the morphism
  $\alpha_{x} \colon F(x) \to G(x)$ is $\pi$-Cartesian for every $x
  \in \mathcal{E}$.

  Let $\epsilon \colon \mathcal{C} \times [1] \to \mathcal{C}$ be the
  counit natural transformation $jr \to \id_{\mathcal{C}}$. Then we
  may choose a Cartesian morphism $\bar{\epsilon} \colon \mathcal{E}
  \times \Delta^{1}\to \mathcal{E}$ in $\Fun(\mathcal{E},
  \mathcal{E})$ over $\epsilon \circ \pi \colon jr\pi \to \pi$ with
  target $\id_{\mathcal{E}}$.  The morphism $\bar{\epsilon}_{x}$ is 
  $\pi$-Cartesian for every $x \in \mathcal{E}$. We let $R' :=
  \bar{\epsilon}|_{\mathcal{E} \times \{0\}}$. By construction we then
  have a commutative diagram
  \[
  \begin{tikzcd}
    \mathcal{E} \arrow[bend left=20]{drr}{R'} \arrow[bend right=20]{ddr}{r\pi} \arrow[dashed]{dr}{R} \\
     & \mathcal{E}_{0} \arrow{r}{J} \arrow{d}{\pi_{0}}
     \arrow[phantom]{dr}[very near start]{\ulcorner} & \mathcal{E} \arrow{d}{\pi}
     \\
     & \mathcal{C}_{0} \arrow{r}{j} & \mathcal{C},
  \end{tikzcd}
  \]
  and by the universal property of the pullback $\mathcal{E}_{0}$
  there exists a dashed arrow
  $R \colon \mathcal{E} \to \mathcal{E}_{0}$. By (the dual of)
  \cite{HTT}*{Proposition 5.2.2.8}, to show that $R$ is right adjoint
  to $J$ it suffices to show that for all $X \in \mathcal{E}_{0}$ and
  $Y \in \mathcal{E}$, the map
  \[ \Map_{\mathcal{E}_{0}}(X, RY) \to \Map_{\mathcal{E}}(JX,JRY)
  \to \Map_{\mathcal{E}}(JX, Y) \] arising from composition with
  $\bar{\epsilon}_{Y}$ is an equivalence. Consider the commutative
  diagram
  \[
  \begin{tikzcd}
    \Map_{\mathcal{E}_{0}}(X, RY) \arrow{d} \arrow{r} &
    \Map_{\mathcal{E}}(JX, JRY) \arrow{d} \arrow{r} &
    \Map_{\mathcal{E}}(JX, Y) \arrow{d} \\
    \Map_{\mathcal{C}_{0}}(x, ry) \arrow{r} & \Map_{\mathcal{C}}(jx,
    jry) \arrow{r} & \Map_{\mathcal{C}}(jx, y).
  \end{tikzcd}
\]
where we write $x := \pi_{0}X$, $y := \pi Y$, and use the
identifications $\pi_{0}RY \simeq ry$, $\pi J X \simeq jx$, and
$\pi JRY \simeq jry$. Here the left-hand square is Cartesian since
$\mathcal{E}_{0}$ is a pullback, and the right-hand square is
Cartesian since the map $\bar{\epsilon}_{Y}$ is a $\pi$-Cartesian
morphism; this means the composite square is also Cartesian.  The
composite in the bottom row is an equivalence as $\epsilon$ is the
counit for the adjunction $j \dashv r$, so this means the composite
upper horizontal map is also an equivalence, as required.
\end{proof}

\begin{lemma}\label{lem:ev0catcart}
  Let $\mathcal{C}$ be an \icat{} with finite limits, and write $i$
  for the inclusion $\{[0]\} \hookrightarrow \simp^{\op}$. Then the
  functor $i^{*} \colon \Cat(\mathcal{C}) \to \mathcal{C}$ given by
  composition with $i$ is a Cartesian fibration.
\end{lemma}
\begin{proof}
  We first observe that the functor $i^{*}$ has a right adjoint.  The
  right Kan extension functor $i_{*}$ gives a right adjoint to
  $i^{*}$, considered as a functor from $\Fun(\simp^{\op},
  \mathcal{C})$ to $\mathcal{C}$. For $C \in \mathcal{C}$ the
  simplicial object $i_{*}C$ is given by $(i_{*}C)_{n} \simeq
  C^{\times (n+1)}$, with face maps given by projections and
  degeneracies by diagonal maps, and so lies in the full subcategory
  $\Cat(\mathcal{C})$. Thus $i_{*}$ restricts to a right adjoint to
  $i^{*} \colon \Cat(\mathcal{C}) \to \mathcal{C}$.

  We now apply the criterion of \cite{nmorita}*{Corollary 5.28}: Since
  $i^{*}$ has a right adjoint, to see that $i^{*}$ is a Cartesian
  fibration it is enough to check that, given $X_{\bullet} \in
  \Cat(\mathcal{C})$, a map $C \to X_{0}$, and a pullback square
  \nolabelcsquare{Y_{\bullet}}{X_{\bullet}}{i_{*}C}{i_{*}X_{0},} the
  induced map $Y_{0} \to (i_{*}C)_{0} \to C$ is an equivalence.
  
  Limits in $\Cat(\mathcal{C})$ are computed in $\Fun(\simp^{\op},
  \mathcal{C})$, and limits in functor \icats{} are computed
  objectwise, so this pullback square gives a pullback square when
  evaluated at every object of $\simp^{\op}$. Thus we in particular
  have a pullback square
  \nolabelcsquare{Y_{0}}{X_{0}}{(i_{*}C)_{0}}{(i_{*}X_{0})_{0}} in
  $\mathcal{C}$. Here the right vertical map is an equivalence, so the
  left vertical map is also an equivalence.
\end{proof}

\begin{lemma}\label{lem:funadj}\ 
  \begin{enumerate}[(i)]
  \item Suppose $f \colon \mathcal{C} \rightleftarrows \mathcal{D} :
    g$ is an adjunction. Then there is an
    induced adjunction $g^{*} \colon \Fun(\mathcal{C}, \mathcal{E})
    \rightleftarrows \Fun(\mathcal{D}, \mathcal{E}) : f^{*}$ on
    functor \icats{}.
  \item Suppose $\mathcal{C}$ is an \icat{} with an initial object
    $\emptyset$. Then the functor $\Fun(\mathcal{C}, \mathcal{E}) \to
    \mathcal{E}$ given by evaluation at $\emptyset$ is left adjoint to
    the constant diagram functor $\mathcal{E} \to \Fun(\mathcal{C},
    \mathcal{E})$.
  \end{enumerate}
\end{lemma}
\begin{proof}
  To prove (i), we just choose unit and counit transformations for $f
  \dashv g$; composing with these gives unit and counit
  transformations for an adjunction $g^{*} \dashv f^{*}$, since the
  adjunction identities can be deduced from the adjunction identities
  for $f \dashv g$. Now (ii) is a special case of (i), applied to the
  adjunction  $\{\emptyset\} \rightleftarrows \mathcal{C}$.
\end{proof}

\begin{lemma}\label{lem:segncatadj}
  Let $\mathcal{C}$ be an \icat{} with finite limits. The inclusion
  $\Seg_{n}(\mathcal{C}) \hookrightarrow
  \Cat(\Seg_{n-1}(\mathcal{C}))$ admits a right adjoint for all $n$.
\end{lemma}
\begin{proof}
  We prove this by applying Lemma~\ref{lem:fullsubradj} to the
  pullback square
  \csquare{\Seg_{n}(\mathcal{C})}{\Cat(\Seg_{n-1}(\mathcal{C}))}{\mathcal{C}}{\Seg_{n-1}(\mathcal{C}),}{}{}{i^{*}}{c}
  where $c$ is the constant-diagram functor and $i$ denotes the
  inclusion $\{[0]\} \times (\simp^{\op})^{\times (n-1)}
  \hookrightarrow (\simp^{\op})^{\times n}$.

  To see that $c$ has a right adjoint, let $e$ denote the
  inclusion $\{([0],\ldots,[0])\} \hookrightarrow
  (\simp^{\op})^{\times (n-1)}$. By Lemma~\ref{lem:funadj} the functor
  $e^{*} \colon \Fun(\simp^{n-1,\op}, \mathcal{C}) \to \mathcal{C}$
  given by evaluation at $([0],\ldots,[0])$ is right adjoint to the
  constant-diagram functor $c \colon \mathcal{C} \to
  \Fun(\simp^{n-1,\op}, \mathcal{C})$. Since $c$ factors through
  $\Seg_{n-1}(\mathcal{C})$, this restricts to an adjunction
  \[ c : \mathcal{C} \rightleftarrows \Seg_{n-1}(\mathcal{C}) :
  e^{*}.\] Moreover, since $e^{*}c$ is the identity functor on
  $\mathcal{C}$, the functor $c$ is fully faithful. 

  It follows from Lemma~\ref{lem:ev0catcart} that $i^{*}$ is a Cartesian
  fibration, so the hypotheses of  Lemma~\ref{lem:fullsubradj} are
  satisfied, which implies the existence of the required right adjoint
  $\Cat(\Seg_{n-1}(\mathcal{C})) \to \Seg_{n}(\mathcal{C})$.
\end{proof}

\begin{lemma}\label{lem:adjoncat}
  Suppose $\mathcal{C}$ and $\mathcal{D}$ are \icats{} with finite
  limits, and
  \[ F : \mathcal{C} \rightleftarrows \mathcal{D} : G \]
  is an adjunction such that the left adjoint $F$ preserves finite
  limits. Then composition with $F$ and $G$ gives an adjunction
  \[ F_{*} : \Cat(\mathcal{C}) \rightleftarrows \Cat(\mathcal{D}) : G_{*}.\]
\end{lemma}
\begin{proof}
  Composition with $F$ and $G$ induces an adjunction
  \[ F_{*} : \Fun(\simp^{\op}, \mathcal{C}) \rightleftarrows
  \Fun(\simp^{\op}, \mathcal{D}) : G_{*}\]
  by the same argument as in the proof of Lemma~\ref{lem:funadj}, and since $F$ preserves finite limits both $F_{*}$ and $G_{*}$
  preserve the full subcategories of category objects, giving the
  desired restricted adjunction between these.
\end{proof}

\begin{proof}[Proof of Proposition~\ref{propn:underfold}]
  Since the full subcategory $\Seg_{k}(\mathcal{C})$ of
  $\Cat(\Seg_{k-1}(\mathcal{C}))$ is closed under finite
  limits, by combining Lemma~\ref{lem:segncatadj} with
  Lemma~\ref{lem:adjoncat} it follows that the inclusion
  \[\Cat^{i}(\Seg_{k}(\mathcal{C})) \hookrightarrow
  \Cat^{i+1}(\Seg_{k-1}(\mathcal{C}))\] has a right adjoint for all
  $i$. Composing these adjoints, we see that the composite inclusion
  \[ \Seg_{n}(\mathcal{C}) \hookrightarrow
  \Cat(\Seg_{n-1}(\mathcal{C}))  \hookrightarrow \cdots
  \hookrightarrow \Cat^{n}(\mathcal{C}) \]
  also has a right adjoint.
\end{proof}

\begin{defn}
  Let $\mathcal{C}$ be an $n$-uple Segal space. We refer to
  $U_{\Seg}\mathcal{C}$ as the \emph{underlying $n$-fold Segal space}
  of $\mathcal{C}$.
\end{defn}

\begin{remark}
  In the definition of $n$-fold Segal space, we privileged one of the
  spaces of 1-morphisms, etc. By making different choices (i.e.\ by
  permuting the coordinates in $\simp^{n,\op}$) we get $n!$
  different $n$-fold Segal spaces from an $n$-uple Segal space.
\end{remark}

\section{The $(\infty,k)$-Category of Iterated
  Spans}\label{sec:spancat} 
In this section we construct the $(\infty,k)$-category
$\Span_{k}(\mathcal{C})$ of iterated spans in an \icat{} $\mathcal{C}$
with finite limits, in the form of a $k$-fold Segal space. As we noted
in the introduction, we expect the $(\infty,k-1)$-category of
morphisms from $X$ to $Y$ in $\Span_{k}(\mathcal{C})$ to be
$\Span_{k-1}(\mathcal{C}_{/X \times Y})$.  To avoid dealing with
coherence issues our construction will not be immediately recognizable
as being of this form, but we'll see below in
Proposition~\ref{propn:spanmaps} that the mapping
$(\infty,k-1)$-categories do indeed admit this inductive
description. Moreover, as the diagrams describing an $n$-fold span
become increasingly complicated as $n$ increases, it turns out to be
much easier to define $\Span_{k}(\mathcal{C})$ as the underlying
$k$-fold Segal space of a $k$-uple Segal space
$\SPAN_{k}(\mathcal{C})$. This will have all of its $k$ types of
$1$-morphisms given by spans in $\mathcal{C}$, and the ``commutative
squares'' are given by diagrams in $\mathcal{C}$ of the
form
\[
\begin{tikzcd}[row sep=scriptsize, column sep=scriptsize]
X & A\arrow{l} \arrow{r} & Y \\
C \arrow{u} \arrow{d}& J \arrow{u} \arrow{d} \arrow{l} \arrow{r}& D \arrow{u} \arrow{d} \\
Z & B \arrow{l} \arrow{r}& W,
\end{tikzcd}
\]
whose shape is given by the \emph{product} of the diagram shape describing
spans; similarly, the higher ``commutative $i$-cubes'' are described
by the $i$-fold product of this shape. Notice that when we restrict to
the underlying $k$-fold Segal space we consider only those diagrams of
the above form where the maps $X \from C \to Z$ and $Y \from D \to W$
are all identities, in which case this diagram is equivalent to a
2-fold span. In fact, since it is no more work and will be useful in
the next section, we'll construct an $n$-fold category object
$\SPAN^{+}_{k}(\mathcal{C})$ in $\CatI$ --- this may be regarded as a
$(k+1)$-uple \icat{}, from which we may extract an
$(\infty,k+1)$-category $\Span^{+}_{k}(\mathcal{C})$ extending
$\Span_{k}(\mathcal{C})$, with $(k+1)$-morphisms given by natural
transformations of $k$-fold span diagrams.

In the case $k = 1$, the construction we use is due to
Barwick~\cite{BarwickQ}, and in general we consider a simple inductive
generalization of Barwick's definition. We begin by introducing some
notation:
\begin{defn}
  Let $\bbS^{n}$ be the partially ordered set with objects pairs
  $(i,j)$ with $0 \leq i \leq j \leq n$, where $(i,j) \leq (i',j')$ if
  $i \leq i'$ and $j' \leq j$. A map of totally ordered sets $\phi \colon [n] \to [m]$ induces a
  functor $\bbSigma^{n} \to \bbSigma^{m}$ by sending $(i,j)$ to
  $(\phi(i), \phi(j))$; we thus get a functor $\bbSigma^{\bullet} \colon \simp
  \to \Cat$. We will also write $\bbS^{m_{1},\ldots, m_{k}}$ for
  the product $\bbS^{m_{1}} \times \cdots \times \bbS^{m_{k}}$, which
  defines a functor $\bbS^{\bullet,\ldots,\bullet} \colon
  \simp^{k} \to \Cat$.
\end{defn}

\begin{defn}
  Let $\bbLn$ denote the full subcategory of $\bbSn$ spanned by those
  pairs $(i,j)$ such that $j-i \leq 1$. These subcategories are not in
  general preserved by the functors $\bbSigma(\phi)$ for $\phi$ in
  $\simp$, but they are preserved by the functors arising from
  inert maps. We thus get a functor $\bbL^{\bullet}
  \colon \simp_{\txt{int}} \to \Cat$ with a natural transformation $i
  \colon \bbL^{\bullet} \to \bbS^{\bullet}|_{\simp_{\txt{int}}}$. We
  will also write $\bbL^{m_{1},\ldots, m_{k}}$ for the product
  $\bbL^{m_{1}} \times \cdots \times \bbL^{m_{k}}$, which defines a
  functor $\bbL^{\bullet,\ldots,\bullet} \colon \simp^{k}_{\txt{int}} \to \Cat$, with a natural transformation $i \colon
  \bbL^{\bullet,\ldots,\bullet} \to
  \bbS^{\bullet,\ldots,\bullet}|_{\simp^{k}_{\txt{int}}}$.
\end{defn}

\begin{exs}\ 
  \begin{enumerate}[(i)]
  \item The category $\bbS^{0} = \bbL^{0}$ is the trivial
    one-object category.
  \item The category $\bbS^{1} = \bbL^{1}$ can be depicted as
    \[ (0,0) \leftarrow (0,1) \to (1,1). \]
  \item The category $\bbS^{2}$ can be depicted as
  
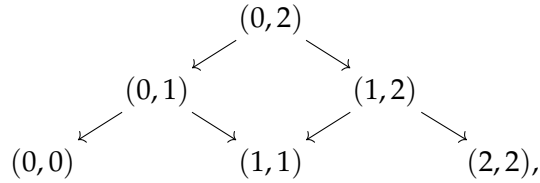
\[ %
\begin{tikzpicture} %
\matrix (m) [matrix of math nodes,row sep=1em,column sep=1em,text height=1.5ex,text depth=0.25ex] %
{  &       & (0,2) &       &  \\
   & (0,1) &       & (1,2) &   \\
(0,0)  &   & (1,1) &       & (2,2),   \\ };
\path[->,font=\footnotesize] %
(m-1-3) edge (m-2-2)
(m-1-3) edge (m-2-4)
(m-2-2) edge (m-3-1)
(m-2-2) edge (m-3-3)
(m-2-4) edge (m-3-3)
(m-2-4) edge (m-3-5);
\end{tikzpicture}%
\]%
and the subcategory $\bbL^{2}$ as
  
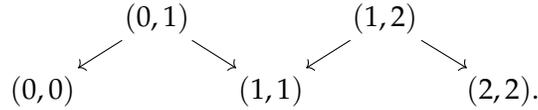
\[ %
\begin{tikzpicture} %
\matrix (m) [matrix of math nodes,row sep=1em,column sep=1em,text height=1.5ex,text depth=0.25ex] %
{     & (0,1) &       & (1,2) &   \\
(0,0)  &   & (1,1) &       & (2,2).   \\ };
\path[->,font=\footnotesize] %
(m-1-2) edge (m-2-1)
(m-1-2) edge (m-2-3)
(m-1-4) edge (m-2-3)
(m-1-4) edge (m-2-5);
\end{tikzpicture}%
\]%
\item The category $\bbS^{3}$ can be depicted as
  
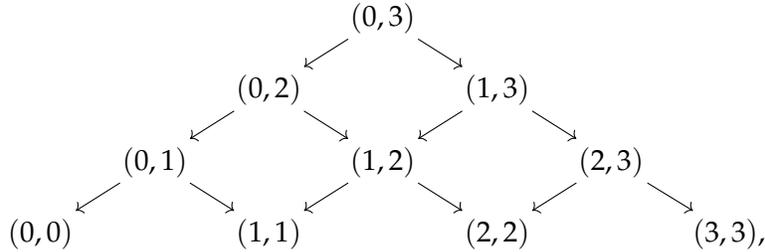
\[ %
\begin{tikzpicture} %
\matrix (m) [matrix of math nodes,row sep=1em,column sep=1em,text height=1.5ex,text depth=0.25ex] %
{
&  &       & (0,3) &       &  &\\
&  & (0,2) &       & (1,3) &  &\\
&(0,1)   &  & (1,2) &  & (2,3) & \\
(0,0)  &   & (1,1) &  & (2,2) & & (3,3),   \\ };
\path[->,font=\footnotesize] %
(m-1-4) edge (m-2-3)
(m-1-4) edge (m-2-5)
(m-2-3) edge (m-3-2)
(m-2-3) edge (m-3-4)
(m-2-5) edge (m-3-4)
(m-2-5) edge (m-3-6)
(m-3-2) edge (m-4-1)
(m-3-2) edge (m-4-3)
(m-3-4) edge (m-4-3)
(m-3-4) edge (m-4-5)
(m-3-6) edge (m-4-5)
(m-3-6) edge (m-4-7);
\end{tikzpicture}%
\]%
and the subcategory $\bbL^{3}$ as
  
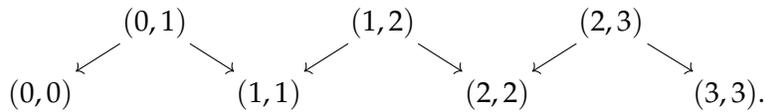
\[ %
\begin{tikzpicture} %
\matrix (m) [matrix of math nodes,row sep=1em,column sep=1em,text height=1.5ex,text depth=0.25ex] %
{     & (0,1) &       & (1,2) &    &  (2,3) \\
(0,0)  &   & (1,1) &       & (2,2) &  & (3,3).   \\ };
\path[->,font=\footnotesize] %
(m-1-2) edge (m-2-1)
(m-1-2) edge (m-2-3)
(m-1-4) edge (m-2-3)
(m-1-4) edge (m-2-5)
(m-1-6) edge (m-2-5)
(m-1-6) edge (m-2-7)
;
\end{tikzpicture}%
\]%

  \end{enumerate}
\end{exs}

\begin{remark}\label{rmk:bbSvsInt}
  We can think of the object $(i,j) \in \bbS^{n}$ as a subinterval
  $(i, i+1, \ldots, j)$ of the ordered set $[n] \in \simp$. This gives
  a functor $\bbS^{n} \to \simp^{\op}_{/[n]}$ that takes $(i,j)$ to
  the inert map $[j-i] \to [n]$ that sends $t \in [j-i]$ to $i+t$,
  with a map $(i,j) \to (i',j')$ sent to the unique inert map $[j'-i']
  \to [j-i]$ such that the diagram
  \nolabelopctriangle{{[j'-i']}}{{[j-i]}}{{[n]}} commutes. This
  functor identifies $\bbS^{n}$ with the subcategory
  $\simp_{\txt{int},/[n]}^{\op}$ of inert maps, which will be
  useful later. Similarly, restricting the functor to $\bbL^{n}$, this
  is identified with the full subcategory $\txt{Cell}^{1,\op}_{/[n]}$
  as defined in Definition~\ref{defn:inert}.
\end{remark}

\begin{defn}
  Suppose $\mathcal{C}$ is an \icat{} with finite limits. A functor $f
  \colon \bbS^{n_{1},\ldots,n_{k}} \to \mathcal{C}$ is
  \defterm{Cartesian} if $f$ is a right Kan extension of its
  restriction to $\bbL^{n_{1},\ldots,n_{k}}$ along the inclusion $i =
  i_{n_{1},\ldots,n_{k}}$. Equivalently, $f$ is Cartesian \IFF{} the
  unit map \[f \to i_{*}i^{*}f\] is an equivalence, where $i^{*}
  \colon \Fun(\bbS^{n_{1},\ldots,n_{k}}, \mathcal{C}) \to
  \Fun(\bbL^{n_{1},\ldots,n_{k}}, \mathcal{C})$ is the functor given
  by composition with $i$, and $i_{*}$ denotes its right adjoint,
  given by right Kan extension. We denote the full subcategory of
  the \icat{}  $\Fun(\bbS^{n_{1},\ldots,n_{k}}, \mathcal{C})$ spanned by the
  Cartesian functors by $\Fun^{\Cart}(\bbS^{n_{1},\ldots,n_{k}},
  \mathcal{C})$, and its underlying space by
  $\Map^{\Cart}(\bbS^{n_{1},\ldots,n_{k}}, \mathcal{C})$.
\end{defn}

We can also formulate this condition inductively:
\begin{lemma}\label{lem:cartindcond}
  Suppose $\mathcal{C}$ is an \icat{} with finite limits. Then the
  following are equivalent, for a functor $F \colon
  \bbS^{n_{1},\ldots,n_{k}} \to \mathcal{C}$:
  \begin{enumerate}[(1)]
  \item $F$ is Cartesian.
  \item $F$ is a right Kan extension of its restriction to
    $\bbL^{n_{1}} \times \bbS^{n_{2},\ldots,n_{k}}$, and for every $I
    \in \bbL^{n_{1}}$ the restriction $F_{I}$ of $F$ to $\{I\} \times
    \bbS^{n_{2},\ldots, n_{k}}$ is Cartesian.
  \item The functor $\tilde{F} \colon \bbS^{n_{1}} \to
    \Fun(\bbS^{n_{2},\ldots,n_{k}}, \mathcal{C})$ corresponding to $F$
    is Cartesian, and for every $X \in \bbL^{n_{1}}$ the image $\tilde{F}(X)
    \in \Fun(\bbS^{n_{2},\ldots,n_{k}}, \mathcal{C})$ is Cartesian.
  \end{enumerate}
\end{lemma}
\begin{proof}
  Let $F'$ denote the restriction of $F$ to $\bbL^{n_{1}} \times
  \bbS^{n_{2},\ldots,n_{k}}$. We first prove that the following pair of
  conditions are equivalent:
  \begin{enumerate}[(a)]
  \item $F'$ is a right Kan extension of its restriction to
    $\bbL^{n_{1}} \times \bbL^{n_{2},\ldots,n_{k}}$.
  \item For every $I \in \bbL^{n_{1}}$, the restriction $F'_{I}$ of
    $F'$ to $\{I\} \times \bbS^{n_{2},\ldots,n_{k}}$ is Cartesian.
  \end{enumerate}
  Condition (a) says that for every $I \in \bbL^{n_{1}}$, $J \in
  \bbS^{n_{2},\ldots,n_{k}}$ the natural map
  \[ F(I,J) \to \lim_{X \in (\bbL^{n_{1}} \times
    \bbL^{n_{2},\ldots,n_{k}})_{(I,J)/}} F(X)\]
  is an equivalence, whereas (b) says that for every $I \in \bbL^{n_{1}}$, $J \in
  \bbS^{n_{2},\ldots,n_{k}}$ the natural map
  \[ F(I,J) \to \lim_{Y \in \bbL^{n_{2},\ldots,n_{k}}_{J/}} F(I,Y)\]
  is an equivalence. But the inclusion $\{I\} \times
  \bbL^{n_{2},\ldots,n_{k}}_{J/} \to \bbL^{n_{1}}_{I/} \times
  \bbL^{n_{2},\ldots,n_{k}}_{J/}$ is coinitial (as coinitial maps are
  closed under products by the dual of \cite{HTT}*{Corollary
    4.1.1.13}), so these two limits are canonically equivalent, hence
  (a) and (b) are indeed equivalent.

  Thus condition (2) holds \IFF{} $F$ is the right Kan extension
  of its restriction $F'$ and $F'$ is the right Kan extension of its
  restriction to $\bbL^{n_{1},\ldots,n_{k}}$. The transitivity
  of Kan extensions implies this is the same as condition (1).

  To see that (2) is equivalent to (3), we are reduced to showing that
  the following two conditions are equivalent:
  \begin{enumerate}[(a')]
  \item $F$ is a right Kan extension of its restriction to
    $\bbL^{n_{1}} \times \bbS^{n_{2},\ldots,n_{k}}$.
  \item $\tilde{F} \colon \bbS^{n_{1}} \to
    \Fun(\bbS^{n_{2},\ldots,n_{k}}, \mathcal{C})$ is Cartesian.
  \end{enumerate}
  Here (a') says that for every $I \in \bbS^{n_{1}}$ and $J \in
  \bbS^{n_{2},\ldots,n_{k}}$ the natural map
  \[ F(I,J) \to \lim_{X \in (\bbL^{n_{1}} \times
    \bbS^{n_{2},\ldots,n_{k}})_{(I,J)/}} F(X) \] is an equivalence. By
  the same argument as above, the inclusion $\bbL^{n_{1}}_{I/} \times
  \{J\} \hookrightarrow (\bbL^{n_{1}} \times
  \bbS^{n_{2},\ldots,n_{k}})_{(I,J)/}$ is coinitial, so this is
  equivalent to $F(I,J)$ being $\lim_{Y \in \bbL^{n_{1}}_{I/}}
  F(Y,J)$. As limits in functor \icats{} are computed objectwise, this
  is equivalent to (b'), which completes the proof.
\end{proof}

\begin{remark}
  Expanding the definition, in the case $k = 1$ we see that a functor
  $f \colon \bbSn \to \mathcal{C}$ is Cartesian \IFF{} it is obtained
  by taking iterated pullbacks of the $n$ spans $f_{1},\ldots,f_{n}$
  given by restricting $f$ along the inclusions $\bbS(\rho_{i}) \colon
  \bbS^{1} \to \bbSn$ coming from the maps $\rho_{i}\colon [1] \to
  [n]$ in $\simp$ that send $0$ to $i-1$ and $1$ to $i$. In other
  words, $f$ is Cartesian \IFF{} it presents the $n$-fold composite of
  these spans as 1-morphisms in our desired \icat{} of spans. The
  preceding Lemma then says that, similarly, a functor
  $\bbS^{n_{1},\ldots,n_{k}} \to \mathcal{C}$ is Cartesian if it
  presents the appropriate composite of spans in $k$ different
  directions.
\end{remark}

\begin{defn}
  Suppose $\mathcal{C}$ is an \icat{} with finite limits. Let
  $\overline{\SPAN}^{+}_{k}(\mathcal{C}) \to \simp^{k,\op}$ be a
  coCartesian fibration associated to the functor $\simp^{k,\op}
  \to \CatI$ given by $\Fun(\bbS^{\bullet,\ldots,\bullet},
  \mathcal{C})$. Then we define $\SPAN^{+}_{k}(\mathcal{C})$ to be the
  full subcategory of $\overline{\SPAN}^{+}_{k}(\mathcal{C})$ spanned by
  the Cartesian functors $\bbS^{n_{1},\ldots,n_{k}} \to \mathcal{C}$
  for all $n_{1},\ldots,n_{k}$.
\end{defn}

Our goal is now to prove that $\SPAN^{+}_{k}(\mathcal{C}) \to
\simp^{k,\op}$ is the coCartesian fibration associated to a $k$-fold
category object in $\CatI$. We will first show that this is indeed a
coCartesian fibration, which is a consequence of the following result:
\begin{propn}\label{propn:cartcomp}
  Suppose $\mathcal{C}$ is an \icat{} with finite limits and  $F
  \colon \bbS^{m} \to \mathcal{C}$ is a Cartesian functor. Then for
  any map $\phi \colon [n] \to [m]$ in $\simp$, the composite functor
  \[ \phi^{*}F \colon \bbS^{n} \to \bbS^{m} \to \mathcal{C}\]
  is again Cartesian.
\end{propn}

Before we give the proof, we first prove the key technical observation
we need, which requires a bit of notation:
\begin{defn}
  Suppose $\phi \colon [n] \to [m]$ is an injective map in
  $\simp$. Let $\bbL^{m}[\phi]$ denote the full subcategory of
  $\bbS^{m}$ spanned by the objects $I = (i,j)$ such that either $I
  \in \bbL^{m}$ or $I \geq \phi(J)$ for some $J \in \bbL^{n}$.  
\end{defn}

\begin{lemma}\label{lem:bbLcoinitial}\ 
  \begin{enumerate}[(i)]
  \item Suppose $\phi \colon [n] \to [m]$ is a surjective map in
    $\simp$. Then for every $I \in \bbS^{n}$, the induced functor
    $\bbL^{n}_{I/} \to \bbL^{m}_{\phi(I)/}$ is coinitial.
  \item Suppose $\phi \colon [n] \to [m]$ is an injective map in
    $\simp$. Then for every $I \in \bbS^{n}$, the induced functor
    $\bbL^{n}_{I/} \to \bbL^{m}[\phi]_{\phi(I)/}$ is coinitial.
  \end{enumerate}
\end{lemma}
\begin{proof}
  If $I = (i,j)$ then the category $\bbL^{n}_{I/}$ is equivalent to
  $\bbL^{n'}$ where $n' = j-i$. If we let $\phi'$ denote the
  restriction of $\phi$ to a map $\phi' \colon [n'] \cong \{i, i+1,
  \ldots, j\} \to \{\phi(i), \phi(i)+1, \ldots, \phi(j)\} \cong [m']$,
  where $m' := \phi(j)-\phi(i)$, then we also have
  $\bbL^{m}_{\phi(I)/} \simeq \bbL^{m'}$ in case (i) and
  $\bbL^{m}[\phi]_{\phi(I)/} \simeq \bbL^{m'}[\phi']$ in case
  (ii). Replacing $\phi$ by $\phi'$ we may therefore 
  without loss of generality assume
  that $I = (0,n)$ and $\phi(I) = (0,m)$. Note that for (ii) this
  implies that $\bbL^{m}[\phi]$ consists precisely of those objects
  $I$ such that $I \geq \phi(J)$ for some $J \in \bbL^{n}$.

  Using (the dual of) \cite{HTT}*{Theorem 4.1.3.1}, to prove (i) we
  must show that for every $J \in \bbL^{m}$ the category
  $\bbL^{n}_{/J}$ is weakly contractible. Since the category
  $\bbL^{m}$ is a 
  partially ordered set, we can identify this slice category with the full
  subcategory of $\bbL^{n}$ spanned by the objects $X$ such that $J
  \geq \phi(X)$. If $J = (j,k)$, set $j' := \max \{x : \phi(x) \leq
  j\}$ and $k' := \min \{x : \phi(x) \geq k\}$; then $X = (x,y)$
  satisfies $J \geq \phi(X)$ \IFF{} $x \leq j'$ and $k' \leq y$.  The
  integer $k-j$ is either $0$ or $1$. If $k -j = 1$, then (as $\phi$
  is surjective) $k' = j'+1$ and so $\bbL^{n}_{/J}$ consists of the
  single object $(j',j'+1)$. On the other hand, if $k = j$, then $k'
  \leq j'$ and $\bbL^{n}_{/J}$ contains the objects $(k'-1,k')$,
  $(k',k')$, $(k',k'+1)$, \ldots, $(j',j')$, $(j',j'+1)$. Thus
  $\bbL^{n}_{/J}$ is an iterated pushout $[1] \amalg_{[0]} [1]
  \amalg_{[0]}\cdots \amalg_{[0]} [1]$ and hence weakly contractible.

  To prove (ii), appealing to \cite{HTT}*{Theorem 4.1.3.1} again we
  need to show that for every $J \in \bbL^{m}[\phi]$ the category
  $\bbL^{n}_{/J}$ is weakly contractible. We can again identify this
  with the full subcategory of $\bbL^{n}$ spanned by the objects $X$
  such that $\phi(X) \leq J$, i.e.\ those $X = (x,y)$ such that
  $\phi(x) \leq j \leq k \leq \phi(y)$. Let $j'$ and $k'$ be defined
  as before, then again this holds \IFF{} $x \leq j'$ and $y \geq
  k'$. Since $\phi$ is injective we have $j' \leq k'$, and by
  definition of $\bbL^{m}[\phi]$ there exists some $Y = (a,b) \in
  \bbL^{n}$ such that $\phi(Y) \leq J$, which forces $k' \leq
  j'+1$. Thus $(j',k')$ is an object of $\bbL^{n}$, and by
  construction it is then terminal in $\bbL^{n}_{/J}$. Since any
  \icat{} with a terminal object is weakly contractible, this
  completes the proof.
\end{proof}

\begin{proof}[Proof of Proposition~\ref{propn:cartcomp}]
  We must show that for all $I \in \bbS^{n}$ the natural map
  \[ (\phi^{*}F)(I) \to \lim_{X \in \bbL^{n}_{/I}} (\phi^{*}F)(X) \]
  is an equivalence.

  It suffices to check this in the cases where $\phi$ is either
  injective or surjective, since these classes of maps form a
  factorization system on $\simp$. In the surjective case this map factors as
  \[ F(\phi(I)) \to \lim_{Y \in \bbL^{m}_{/\phi(I)}} F(Y) \to \lim_{X
    \in \bbL^{n}_{/I}} F(\phi(X)).\] Here the first map is an
  equivalence since $F$ is Cartesian, and the second map is an
  equivalence since the functor $\bbL^{n}_{/I} \to
  \bbL^{m}_{/\phi(I)}$ is coinitial by
  Lemma~\ref{lem:bbLcoinitial}(i).

  In the injective case, the map factors as
  \[ F(\phi(I)) \to \lim_{Y \in \bbL^{m}[\phi]_{/\phi(I)}} F(Y) \to
  \lim_{X \in \bbL^{n}_{/I}} F(\phi(X)).\] Since $F$ is Cartesian, it
  is also the right Kan extension of its restriction to
  $\bbL^{m}[\phi]$, since this full subcategory contains
  $\bbL^{m}$. Thus the first map is an equivalence, and the second is
  an equivalence by Lemma~\ref{lem:bbLcoinitial}(ii).
\end{proof}

\begin{cor}\label{cor:SPANleft}
  The restricted projection $\SPAN_{k}^{+}(\mathcal{C}) \to
  \simp^{k,\op}$ is a coCartesian fibration.
\end{cor}

\begin{proof}
  Let $\pi$ denote the projection $\oSPAN_{k}^{+}(\mathcal{C}) \to
  \simp^{k,\op}$ and write $\pi'$ for the restricted projection
  $\SPAN_{k}^{+}(\mathcal{C}) \to \simp^{k,\op}$. Since
  $\SPAN_{k}^{+}(\mathcal{C})$ is a full subcategory of
  $\oSPAN_{k}^{+}(\mathcal{C})$, a $\pi$-coCartesian morphism whose
  source and target are in $\SPAN_{k}^{+}(\mathcal{C})$ is necessarily
  $\pi'$-coCartesian. As $\pi$ is a coCartesian fibration, to see that
  $\pi'$ is one it therefore suffices to check that if $\alpha \to
  \beta$ is a coCartesian morphism in $\oSPAN^{+}_{k}(\mathcal{C})$
  such that $\alpha$ is a Cartesian functor, then $\beta$ is a
  Cartesian functor. In other words, given a Cartesian functor $\alpha
  \colon \bbS^{n_{1},\ldots,n_{k}} \to \mathcal{C}$ and morphisms
  $\phi_{i} \colon [n_{i}] \to [m_{i}]$ in $\simp$ for $i =
  1,\ldots,k$, we must show that the composite functor
  $(\phi_{1},\ldots,\phi_{k})^{*}\alpha \colon
  \bbS^{m_{1},\ldots,m_{k}} \to \mathcal{C}$ is also Cartesian.
  Using Lemma~\ref{lem:cartindcond} we can check this iteratively,
  which means we only need to prove the case $k = 1$. But this case is
  precisely Proposition~\ref{propn:cartcomp}.
\end{proof}

\begin{propn}\label{propn:Lncolim}
  The compatible maps $\bbL^{1},\bbL^{0} \to \bbL^{n}$ induced by the
  inert maps $[0],[1] \to [n]$ give a functor $\bbL^{n}_{\Seg} := \bbL^{1}
  \amalg_{\bbL^{0}} \cdots \amalg_{\bbL^{0}} \bbL^{1} \to \bbL^{n}$,
  where the colimit is formed in $\CatI$.  This is an equivalence of
  \icats{}.
\end{propn}
\begin{proof}
  We will describe the \icatl{} colimit $\bbL^{n}_{\Seg}$ as a
  homotopy colimit in the Joyal model structure. The object $\bbL^{n}$
  in $\CatI$ is represented by the nerve $\mathrm{N}\bbL^{n}$ in
  $\sSet^{J}$, which is a quasicategory. The morphisms
  $\mathrm{N}\bbL^{0} \to \mathrm{N}\bbL^{1}$ induced by the
  inclusions $[0] \to [1]$ are levelwise injective, i.e.\ cofibrations
  in the Joyal model structure. Therefore the iterated (1-categorical)
  pushout $\mathrm{N}\bbL^{1} \amalg_{\mathrm{N}\bbL^{0}} \cdots
  \amalg_{\mathrm{N}\bbL^{0}} \mathrm{N}\bbL^{1}$ is a homotopy
  colimit and a fibrant replacement of it represents
  $\bbL^{n}_{\Seg}$; it is therefore sufficient to show that the
  natural map $\mathrm{N}\bbL^{1} \amalg_{\mathrm{N}\bbL^{0}} \cdots
  \amalg_{\mathrm{N}\bbL^{0}} \mathrm{N}\bbL^{1} \to
  \mathrm{N}\bbL^{n}$ is a weak equivalence in $\sSet^{J}$. But this
  is in fact an \emph{isomorphism} of simplicial sets.
\end{proof}

\begin{propn}\label{propn:SPANkuple}
  The functor associated to the coCartesian fibration $\SPAN^{+}_{k}(\mathcal{C})
  \to \simp^{k,\op}$ is a $k$-fold category object.
\end{propn}
\begin{proof}
  Unwinding the definitions, we must show that for each
  $([n_{1}],\ldots,[n_{k}])$ in $\simp^{k}$, the natural map
  \[ \Fun^{\Cart}(\bbS^{n_{1},\ldots,n_{k}}, \mathcal{C}) \to
  \lim_{([i_{1}], \ldots, [i_{k}]) \in \txt{Cell}^{k,\op}_{/([n_{1}],\ldots,[n_{k}])}} \Fun(\bbS^{i_{1}},
  \mathcal{C}) \times \cdots \times \Fun(\bbS^{i_{k}}, \mathcal{C})\]
  is an equivalence.  Now using Proposition~\ref{propn:Lncolim} and
  the fact that products in $\CatI$ commute with colimits, it follows
  that the target of this map is equivalent to
  $\Fun(\bbL^{n_{1},\ldots,n_{k}}, \mathcal{C})$, and under this
  equivalence the Segal map corresponds to the map given by composing
  with the inclusion $\bbL^{n_{1},\ldots,n_{k}} \hookrightarrow
  \bbS^{n_{1},\ldots,n_{k}}$. Since this is fully faithful, and
  $\Fun^{\Cart}(\bbS^{n_{1},\ldots,n_{k}}, \mathcal{C})$ is
  precisely the space of functors that are right Kan extensions along
  this inclusion, it follows from \cite{HTT}*{Proposition 4.3.2.15} that our map is an equivalence.
\end{proof}

\begin{defn}
  Let $\oSPAN_{k}(\mathcal{C}) \to \simp^{k,\op}$ and
  $\SPAN_{k}(\mathcal{C}) \to \simp^{k,\op}$ be the
  underlying left fibrations of the coCartesian fibrations
  $\oSPAN^{+}_{k}(\mathcal{C})$ and $\SPAN^{+}_{k}(\mathcal{C})$,
  respectively. These correspond to the multisimplicial spaces
  $\Map(\bbS^{\bullet,\ldots,\bullet}, \mathcal{C})$ and
  $\Map^{\Cart}(\bbS^{\bullet,\ldots,\bullet}, \mathcal{C})$, which
  are obtained by composing $\Fun(\bbS^{\bullet,\ldots,\bullet},
  \mathcal{C})$ and $\Fun^{\Cart}(\bbS^{\bullet,\ldots,\bullet},
  \mathcal{C})$ with the underlying $\infty$-groupoid functor
  $\iota$. Since $\iota$ preserves limits, being a right adjoint, the
  latter is a $k$-uple Segal space, which we also denote
  $\SPAN_{k}(\mathcal{C})$.
\end{defn}

\begin{defn}
  We define the \emph{$(\infty,k)$-category $\Span_{k}(\mathcal{C})$
    of iterated spans} in $\mathcal{C}$ to be the $k$-fold Segal space
  $U_{\Seg}\SPAN_{k}(\mathcal{C})$ underlying the $k$-uple Segal space
  $\SPAN_{k}(\mathcal{C})$.
\end{defn}

\begin{remark}
  Using the complete Segal space model for \icats{}, we may regard $\CatI$
  as a full subcategory of $\Seg(\mathcal{S})$. We may then regard
  $\SPAN^{+}_{k}(\mathcal{C})$ as a $k$-uple Segal object in
  $\Seg(\mathcal{S})$, i.e.\ a $(k+1)$-uple Segal space. We let
  $\Span^{+}_{k}(\mathcal{C})$ be the underlying $(k+1)$-fold Segal
  space $U_{\Seg}\SPAN^{+}(\mathcal{C})$; this is an
  $(\infty,k+1)$-category that extends $\Span_{k}(\mathcal{C})$ by
  taking morphisms of $k$-fold spans as $(k+1)$-morphisms.
\end{remark}

\begin{remark}
  We may regard $\SPAN^{+}_{k}(\blank)$ as a functor from \icats{}
  with finite limits to $k$-fold category objects in \icats{} with
  finite limits. Moreover, as has been proved by David
  Li-Bland~\cite{LiBlandSpan}, this functor preserves limits, which
  means that we can apply $\SPAN^{+}_{k}$ levelwise to get a functor
  from $m$-fold category objects in \icats{} with finite limits to
  $(k+m)$-fold category objects. We can also iterate the construction
  $\SPAN^{+}_{1}$ to recover $\SPAN^{+}_{k}$ as $(\SPAN^{+}_{1})^{k}$
  for $k > 1$.
\end{remark}

\section{The $(\infty,k)$-Category of Iterated Spans with Local
  Systems}\label{sec:spanlocsys}
Our goal in this section is to use the $(\infty,k)$-category
$\Span_{k}(\mathcal{S})$ to make, for every $k$-fold Segal space
$\mathcal{C}$, an $(\infty,k)$-category $\Span_{k}(\mathcal{S};
\mathcal{C})$ of $k$-fold spans with $\mathcal{C}$-valued local
systems. In fact, our construction works more generally: if
$\mathcal{X}$ is an \icat{} with finite limits and $\mathcal{C}$ is a
$k$-fold Segal object in $\mathcal{X}$, we will define an
$(\infty,k)$-category $\Span_{k}(\mathcal{X}; \mathcal{C})$ of
$k$-fold iterated spans in $\mathcal{X}$ equipped with
$\mathcal{C}$-valued local systems.  To define this we will first show
that any $k$-fold Segal object in $\mathcal{X}$ determines a section
of the projection $\SPAN_{k}^{+}(\mathcal{X}) \to
\simp^{k,\op}$, and then use results of Lurie to construct
a fibrewise slice category for this section.

\begin{defn}
  Let $\widehat{\bbS} \to \simp^{\op}$ be the Grothendieck fibration
  associated to the functor $\bbS^{\bullet} \colon \simp \to
  \Cat$. Explicitly, $\widehat{\bbS}$ is the category with objects
  pairs $([n], (i,j))$ with $[n] \in \simp$ and $0 \leq i \leq j \leq
  n$, and morphisms $([n], (i,j)) \to ([m], (i',j'))$ given by a
  morphism $\phi \colon [m] \to [n]$ in $\simp$ and a morphism $(i,j)
  \to (\phi(i'), \phi(j'))$ in $\bbS^{n}$. Then the product
  $\widehat{\bbS}^{k} \to \simp^{k,\op}$ is the
  Grothendieck fibration associated to the functor
  $\bbS^{\bullet,\ldots,\bullet} \colon \simp^{k} \to \Cat$.
\end{defn}

By \cite{freepres}*{Proposition 7.3}, the \icat{}
$\overline{\SPAN}^{+}_{k}(\mathcal{X})$ has a universal property: for
every functor of \icats{} $\mathcal{C} \to \simp^{k,\op}$, the
\icat{} $\Fun_{\simp^{k,\op}}(\mathcal{C},
\overline{\SPAN}^{+}_{k}(\mathcal{X}))$ is naturally equivalent to $\Fun(\mathcal{C}
\times_{\simp^{k,\op}} \widehat{\bbS}^{k},
\mathcal{X})$. In particular, giving a section $\simp^{k,\op} \to
\overline{\SPAN}^{+}_{k}(\mathcal{X})$ is equivalent to giving a 
functor $\widehat{\bbS}^{k} \to \mathcal{X}$.

\begin{defn}
  Let $\Pi \colon \widehat{\bbS} \to \simp^{\op}$ denote the functor
  that sends $([n], (i,j))$ to $[j-i]$ and a map $(\phi \colon [m] \to
  [n], (i,j) \to (\phi(i'), \phi(j')))$ to the map $[j-i] \to [j'-i']$
  in $\simp^{\op}$ corresponding to the map of ordered sets taking $s
  \in [j'-i']$ to $\phi(i'+s) - i \in [j-i]$. We write $\Pi_{n} \colon
  \bbS^{n} \to \simp^{\op}$ for the restriction of $\Pi$ to the fibre
  $\bbS^{n}$ --- this takes $(i,j) \in \bbS^{n}$ to $[j-i]$ and a map
  $(i,j) \to (i',j')$ to the (inert) inclusion $[j-i] \to [j'-i']$ taking $s
  \in [j-i]$ to $s+i'-i$.
\end{defn}

Thus any map $\Phi \colon \simp^{k,\op} \to \mathcal{X}$
determines a section $s_{\Phi} \colon \simp^{k,\op} \to
\overline{\SPAN}^{+}_{k}(\mathcal{X})$ via the functor $\Phi \circ
\Pi^{k} \colon \widehat{\bbS}^{k} \to \simp^{k,\op} \to \mathcal{X}$.

\begin{lemma}\label{lem:lKanslice}
  Suppose we have a fully faithful inclusion $\phi \colon \mathcal{I}
  \to \mathcal{J}$ and a commutative triangle
  \opctriangle{\mathcal{I}}{\mathcal{J}}{\mathcal{C}}{\phi}{f}{g}
  such that $g$ is the left Kan extension of $f$ along $\phi$. Given
  $X \in \mathcal{J}$, let $\mathcal{I}_{/X}$ denote the fibre product
  $\mathcal{I} \times_{\mathcal{J}} \mathcal{J}_{/X}$, let $i
  \colon \mathcal{I}_{/X} \to \mathcal{I}$ and $j \colon
  \mathcal{J}_{/X} \to \mathcal{J}$ denote the forgetful
  functors, and let $\phi_{/X}$ denote the induced map
  $\mathcal{I}_{/X} \to \mathcal{J}_{/X}$. Then $g \circ j$ is the
  left Kan extension of $f \circ i$ along $\phi_{/X}$. 
\end{lemma}
\begin{proof}
  We must show that for any object $\alpha \colon Y \to X$ in
  $\mathcal{J}$, the object $gj(\alpha)$ is the colimit of
  the diagram \[\mathcal{I}_{/X}
  \times_{\mathcal{J}_{/X}} (\mathcal{J}_{/X})_{/\alpha} \to
  \mathcal{I}_{/X} \xto{i} \mathcal{I} \xto{f} \mathcal{C}.\]
  But $(\mathcal{J}_{/X})_{/\alpha}$ is equivalent to
  $\mathcal{J}_{/Y}$, and so $\mathcal{I}_{/X}
  \times_{\mathcal{J}_{/X}} (\mathcal{J}_{/X})_{/\alpha}$ is
  equivalent to $\mathcal{I}_{/Y}$, and the diagram is equivalent to
  the composite $\mathcal{I}_{/Y} \to \mathcal{I} \xto{f}
  \mathcal{C}$. Since $g$ is the left Kan extension of $f$, we know
  that $gj(\alpha) \simeq g(Y)$ is the colimit of this diagram, which
  completes the proof.
\end{proof}

\begin{lemma}
  Suppose $\Phi \colon \simp^{k,\op} \to \mathcal{X}$ is a
  $k$-fold category object in $\mathcal{X}$.
  Then the section $s_{\Phi} \colon
  \simp^{k,\op} \to \overline{\SPAN}^{+}_{k}(\mathcal{X})$
  factors through $\SPAN^{+}_{k}(\mathcal{X})$.
\end{lemma}
\begin{proof}
  From the proof of \cite{freepres}*{Proposition 7.3} it follows that
  the value of the section $s_{\Phi}$ at $I = ([n_{1}],\ldots,[n_{k}])
  \in \simp^{k,\op}$ is the composite
  \[ \bbS^{n_{1},\ldots,n_{k}} \xto{\Pi_{n_{1},\ldots,n_{k}}}
  \simp^{k,\op} \xto{\Phi} \mathcal{X},\] where
  $\Pi_{n_{1},\ldots,n_{k}}$ denotes the product $\Pi_{n_{1}} \times
  \cdots \times \Pi_{n_{k}}$. We thus need to check that each of these
  functors is Cartesian. But under the identification of
  $\bbS^{n_{1},\ldots,n_{k}}$ with $\simp^{k,\op}_{\txt{int},/([n_{1}],\ldots,[n_{k}])}$ of Remark~\ref{rmk:bbSvsInt}, the
  functor $\Pi_{n_{1},\ldots,n_{k}}$ corresponds to the forgetful
  functor $\simp^{k,\op}_{\txt{int},/([n_{1}],\ldots,[n_{k}])} \to
\simp^{k,\op}$. By
  Lemma~\ref{lem:kupSegKanExt}, the restriction of $\Phi$ to
  $\simp^{k,\op}_{\txt{int}}$ is the right Kan extension of
  its restriction to $\Cell^{k,\op}$, from which it follows by (the
  dual of) Lemma~\ref{lem:lKanslice} that $\Pi_{n_{1},\ldots,n_{k}}$
  is the right Kan extension of its restriction to
  $\bbL^{n_{1},\ldots,n_{k}}$, since this corresponds to
  $\Cell^{k,\op}_{/([n_{1}],\ldots,[n_{k}])}$ under this
  identification.
\end{proof}

\begin{lemma}\label{lem:parslicecoC}
  Suppose $f \colon \mathcal{E} \to \mathcal{B}$ is a coCartesian
  fibration and $s \colon \mathcal{B} \to \mathcal{E}$ is a
  section. Then there exists an \icat{} $\mathcal{E}_{//s}$ and a
  coCartesian fibration $\mathcal{E}_{//s} \to \mathcal{B}$ with the
  universal property that
  $\Fun_{/\mathcal{B}}(\mathcal{C}, \mathcal{E}_{//s})$ is naturally
  given by a pullback square
  \nolabelcsquare{\Fun_{/\mathcal{B}}(\mathcal{C},
    \mathcal{E}_{//s})}{\Fun_{/\mathcal{B}}(\mathcal{C} \times [1]
    \amalg_{\mathcal{C} \times \{1\}} \mathcal{B},
    \mathcal{E})}{\{s\}}{\Fun_{/\mathcal{B}}(\mathcal{B},
    \mathcal{E}).}
\end{lemma}
\begin{proof}
  Let $\mathfrak{E} \to \mathfrak{B}$ be a coCartesian inner fibration
  representing $f$ such that there is a section
  $s' \colon \mathfrak{B} \to \mathfrak{E}$ representing $s$. Then,
  following \cite{HTT}*{Definition 4.2.2.1}, we define a simplicial
  set $\mathfrak{E}_{//s'}$ over $\mathfrak{B}$ by taking
  $\Hom_{/\mathfrak{B}}(K, \mathfrak{E}_{//s'})$ to be defined by the
  pullback square \nolabelcsquare{\Hom_{/\mathfrak{B}}(K,
    \mathfrak{E}_{//s'})}{\Hom_{/\mathfrak{B}}(K \times \Delta^{1}
    \amalg_{K \times \{1\}} \mathfrak{B},
    \mathfrak{E})}{\{s'\}}{\Hom_{/\mathfrak{B}}(\mathfrak{B},
    \mathfrak{E}).}  Then the projection
  $\mathfrak{E}_{//s'} \to \mathfrak{B}$ is a coCartesian inner
  fibration by \cite{HTT}*{Proposition 4.2.2.4}. We define
  $\mathcal{E}_{//s} \to \mathcal{B}$ to be the morphism in $\CatI$ it
  represents. The defining property of the simplicial set
  $\Hom_{/\mathfrak{B}}(K, \mathfrak{E}_{//s'})$ implies that the
  internal Hom of simplicial sets
  $\Fun_{/\mathfrak{B}}(K, \mathfrak{E}_{//s'})$ is given by an
  analogous pullback square \nolabelcsquare{\Fun_{/\mathfrak{B}}(K,
    \mathfrak{E}_{//s'})}{\Fun_{/\mathfrak{B}}(K \times \Delta^{1}
    \amalg_{K \times \{1\}} \mathfrak{B},
    \mathfrak{E})}{\{s'\}}{\Fun_{/\mathfrak{B}}(\mathfrak{B},
    \mathfrak{E})} of simplicial sets.  Since the inclusion
  $\mathfrak{B} \hookrightarrow K \times \Delta^{1} \amalg_{K \times
    \{1\}} \mathfrak{B}$
  is a cofibration and the Joyal model structure is enriched in itself
  (see \cite{JoyalUABNotes}*{Theorem 5.13}) the right vertical map
  here is a fibration in the Joyal model structure, hence this is a
  homotopy pullback square; this gives the pullback squares we want in
  $\CatI$.
\end{proof}

\begin{defn}
  Suppose $\Phi \colon \simp^{k,\op} \to \mathcal{X}$ is a
  $k$-fold category object in $\mathcal{X}$. We then define
  \[\SPAN_{k}^{+}(\mathcal{X}; \Phi) \to \simp^{k,\op}\] to be the
  coCartesian fibration $\SPAN_{k}^{+}(\mathcal{X})_{//s_{\Phi}} \to
  \simp^{k,\op}$ of Lemma~\ref{lem:parslicecoC}.
\end{defn}

\begin{propn}
  Suppose $\Phi \colon \simp^{k,\op} \to \mathcal{X}$ is a $k$-fold
  category object in $\mathcal{X}$. Then the functor
  $\simp^{k,\op} \to \CatI$ associated to the coCartesian
  fibration $\SPAN_{k}^{+}(\mathcal{X}; \Phi) \to
  \simp^{k,\op}$ is a $k$-fold category object in $\CatI$.
\end{propn}
\begin{proof}
  Using the defining property of
  $\SPAN^{+}_{k}(\mathcal{X}; \Phi)$ we see that a map from an \icat{}
  $\mathcal{C}$ to the fibre
  $\SPAN_{k}^{+}(\mathcal{X};
  \Phi)_{([n_{1}],\ldots,[n_{k}])}$ is naturally equivalent
  to a map \[\mathcal{C}^{\triangleright} \to
  \SPAN_{k}^{+}(\mathcal{X})_{([n_{1}],\ldots,[n_{k}])}\] that restricts to
  $s_{\Phi}([n_{1}],\ldots,[n_{k}])$ at the cone point. In other
  words, we have a natural equivalence 
  \[\begin{split} \SPAN_{k}^{+}(\mathcal{X};
    \Phi)_{([n_{1}],\ldots,[n_{k}])} & \simeq
  (\SPAN_{k}^{+}(\mathcal{X})_{([n_{1}],\ldots,[n_{k}])})_{/s_{\Phi}([n_{1}],\ldots,[n_{k}])}
  \\ & \simeq \Fun^{\txt{Cart}}(\bbS^{n_{1},\ldots,n_{k}},
  \mathcal{X})_{/\Phi \circ \Pi_{n_{1},\ldots,n_{k}}}.
\end{split}
\] Now on the
  one hand we have a pullback square
  \nolabelcsquare{\Fun^{\txt{Cart}}(\bbS^{n_{1},\ldots,n_{k}},
    \mathcal{X})_{/\Phi \circ \Pi_{n_{1},\ldots,n_{k}}}}{\Fun([1],
    \Fun^{\txt{Cart}}(\bbS^{n_{1},\ldots,n_{k}}, \mathcal{X}))}{\{\Phi
    \circ
    \Pi_{n_{1},\ldots,n_{k}}\}}{\Fun^{\txt{Cart}}(\bbS^{n_{1},\ldots,n_{k}},
    \mathcal{X}).}  On the other hand, since limits commute, we have a
  pullback square \nolabelcsquare{\lim
    \Fun^{\txt{Cart}}(\bbS^{i_{1},\ldots,i_{k}}, \mathcal{X})_{/\Phi
      \circ \Pi_{i_{1},\ldots,i_{k}}}}{\lim \Fun([1],
    \Fun^{\txt{Cart}}(\bbS^{i_{1},\ldots,i_{k}},
    \mathcal{X}))}{\lim \{\Phi \circ
    \Pi_{i_{1},\ldots,i_{k}}\}}{\lim
    \Fun^{\txt{Cart}}(\bbS^{i_{1},\ldots,i_{k}}, \mathcal{X}),}  
  where the limit runs over $([i_1],\ldots,[i_k]) \in
      \Cell^{k,\op}_{/([n_{1}],\ldots,[n_{k}])}$.
  Since $\Phi \circ \Pi_{n_{1},\ldots,n_{k}}$ restricts to $\Phi \circ
  \Pi_{i_{1},\ldots,i_{k}}$ under the appropriate inclusion, these two
  squares are equivalent. But then we get that the natural map
  \[\Fun^{\txt{Cart}}(\bbS^{n_{1},\ldots,n_{k}},
  \mathcal{X})_{/\Phi \circ \Pi_{n_{1},\ldots,n_{k}}} \to \lim_{([i_1],\ldots,[i_k]) \in
    \Cell^{k,\op}_{/([n_{1}],\ldots,[n_{k}])}}
  \Fun^{\txt{Cart}}(\bbS^{i_{1},\ldots,i_{k}}, \mathcal{X})_{/\Phi
    \circ \Pi_{i_{1},\ldots,i_{k}}}\]
  is an equivalence, and so $\SPAN^{+}_{k}(\mathcal{X}; \Phi)$ is a
  $k$-fold category object.
\end{proof}

\begin{defn}
  Suppose $\mathcal{C}$ is a $k$-fold Segal object in
  $\mathcal{X}$. We let $\SPAN_{k}(\mathcal{X}; \mathcal{C}) \to
  \simp^{k,\op}$ denote the left fibration obtained from
  the coCartesian fibration $\SPAN^{+}_{k}(\mathcal{X}; \mathcal{C})
  \to \simp^{k,\op}$ by discarding the non-coCartesian
  morphisms; this left fibration classifies a $k$-uple Segal
  space. The \emph{$(\infty,k)$-category $\Span_{k}(\mathcal{X};
    \mathcal{C})$ of iterated spans in $\mathcal{X}$ with local
    systems valued in $\mathcal{C}$} is the underlying $k$-fold Segal
  space $U_{\Seg}\SPAN_{k}(\mathcal{X}; \mathcal{C})$ associated to
  the $k$-uple Segal space $\SPAN_{k}(\mathcal{X}; \mathcal{C})$.
\end{defn}

\section{Complete Segal Objects in an
  $\infty$-Topos}\label{sec:completeseg}
In \S\ref{sec:Segsp} we recalled Rezk's result that the localization
of the \icat{} of Segal spaces at the fully faithful and essentially
surjective morphisms is given by the full subcategory of
\emph{complete} Segal spaces. In this section we begin by reviewing
the generalization of this to \emph{$n$-fold} Segal spaces, originally
proved by Barwick~\cite{BarwickThesis}. We will review the
reformulation of the theory due to Lurie~\cite{LurieGoodwillie}, which
allows for an inductive construction of complete $n$-fold Segal
spaces. Lurie's version of the theorem also works for $n$-fold Segal
objects in an arbitary $\infty$-topos, which describe \emph{internal
  $(\infty,n)$-categories} or more concretely \emph{sheaves of
  $(\infty,n)$-categories} on an $\infty$-topos. We will then prove two
completeness criteria we'll make use of below: first, we show that
completeness of an $n$-fold Segal object in an \itopos{} $\mathcal{X}$
can be checked on the $n$-fold Segal spaces of maps from objects of
$\mathcal{X}$, and second, we give an inductive criterion for the
completeness of an $n$-fold Segal space using the $(n-1)$-fold Segal
spaces of maps.

We begin by reviewing Lurie's notion of a \emph{distributor}. This is
a technical definition that encapsulates the properties needed to make
sense of complete Segal objects, which hold for both an $\infty$-topos
and the \icat{} of complete Segal objects in an $\infty$-topos
(Theorem~\ref{thm:iterdistr}). Using distributors thus allows us to
give a single definition of complete Segal objects that can be
iterated to give a convenient inductive definition of $n$-fold
complete Segal objects in an $\infty$-topos.
\begin{defn}
  A \emph{distributor} consists of an \icat{} $\mathcal{Y}$ together
  with a full subcategory $\mathcal{X}$ such that:
  \begin{enumerate}[(1)]
  \item The \icats{} $\mathcal{X}$ and $\mathcal{Y}$ are presentable.
  \item The full subcategory $\mathcal{X}$ is closed under small
    limits and colimits in $\mathcal{Y}$.
  \item If $Y \to X$ is a morphism in $\mathcal{Y}$ such that $X \in
    \mathcal{X}$, then the pullback functor $\mathcal{X}_{/X} \to
    \mathcal{Y}_{/Y}$ preserves colimits.
  \item The functor $\mathcal{X}^{\op} \to \LCatI$ that sends $X \in
    \mathcal{X}$ to $\mathcal{Y}_{/X}$ (and a morphism $X \to X'$ to
    the pullback functor $\mathcal{Y}_{/X'} \to \mathcal{Y}_{/X}$)
    preserves small limits.
  \end{enumerate}
  If $\mathcal{X} \subseteq \mathcal{Y}$ is a distributor, an
  \emph{$\mathcal{X}$-Segal object} in $\mathcal{Y}$ is a Segal object
  $\mathcal{C} \colon \simp^{\op} \to \mathcal{Y}$ such that
  $\mathcal{C}_{0} \in \mathcal{X}$. We write
  $\Seg_{\mathcal{X}}(\mathcal{Y})$ for the full subcategory of
  $\Seg(\mathcal{Y})$ spanned by the $\mathcal{X}$-Segal objects.
\end{defn}

\begin{remark}
  It follows from \cite{HTT}*{Theorem 6.1.3.9} that if $\mathcal{X}$
  is an $\infty$-topos, then the tautological inclusion $\mathcal{X}
  \subseteq \mathcal{X}$ is a distributor.
\end{remark}

\begin{defn}\label{defn:gpdob}
  Write $\txt{Gpd}(\mathcal{X})$ for the full subcategory of
  $\Seg(\mathcal{X})$ spanned by the \emph{groupoid objects}, i.e.\ the
  simplicial objects $X$ such that for every partition $[n] = S \cup
  S'$ where $S \cap S'$ consists of a single element, the diagram
  \nolabelsmallcsquare{X([n])}{X(S)}{X(S')}{X(S \cap S')} is a
  pullback square. Let $\mathcal{X}\subseteq\mathcal{Y}$ be a distributor, and
  let $\Lambda \colon \mathcal{Y} \to \mathcal{X}$ denote the right
  adjoint to the inclusion $\mathcal{X} \hookrightarrow
  \mathcal{Y}$. The inclusion $\txt{Gpd}(\mathcal{X}) \hookrightarrow
  \Seg(\mathcal{X}) \hookrightarrow \Seg(\mathcal{Y})$ admits a right
  adjoint $\iota \colon \Seg(\mathcal{Y}) \to \txt{Gpd}(\mathcal{X})$,
  which is the composite of the functor $\Lambda \colon
  \Seg(\mathcal{Y}) \to \Seg(\mathcal{X})$ induced by $\Lambda$, and
  the functor $\iota \colon \Seg(\mathcal{X}) \to
  \txt{Gpd}(\mathcal{X})$ right adjoint to the inclusion, which exists
  by \cite{LurieGoodwillie}*{Proposition 1.1.14}. 
\end{defn}

\begin{defn}
  We say an $\mathcal{X}$-Segal object $F \colon \simp^{\op} \to
  \mathcal{Y}$ is \defterm{complete} if the groupoid object $\iota F$
  is constant, and write $\CSS_{\mathcal{X}}(\mathcal{Y})$ for the
  full subcategory of $\Seg_{\mathcal{X}}(\mathcal{Y})$ spanned by the
  complete $\mathcal{X}$-Segal objects. The inclusion
  $\CSS_{\mathcal{X}}(\mathcal{Y}) \hookrightarrow
  \Seg_{\mathcal{X}}(\mathcal{Y})$ admits a left adjoint by
  \cite{HTT}*{Lemma 5.5.4.17}.
\end{defn}

\begin{defn}
  Let $\mathcal{X} \subseteq \mathcal{Y}$ be a distributor, and suppose $f
  \colon \mathcal{C} \to \mathcal{D}$ is a morphism in
  $\Seg_{\mathcal{X}}(\mathcal{Y})$. We say that $f$ is \emph{fully
    faithful and essentially surjective} if:
  \begin{enumerate}[(1)]
  \item The map $|\Gpd(\mathcal{C})| \to |\Gpd(\mathcal{D})|$ is an
    equivalence in the
    \itopos{} $\mathcal{X}$.
  \item The diagram
    \nolabelcsquare{\mathcal{C}_{1}}{\mathcal{D}_{1}}{\mathcal{C}_{0}\times
      \mathcal{C}_{0}}{\mathcal{D}_{0}\times \mathcal{D}_{0}}
    is a pullback square in $\mathcal{Y}$.
  \end{enumerate}
\end{defn}

\begin{thm}[\cite{LurieGoodwillie}*{Theorem 1.2.13}]\label{thm:distcssloc}
  Let $\mathcal{X} \subseteq \mathcal{Y}$ be a distributor. Then the
  left adjoint \[L_{\mathcal{X}\subseteq\mathcal{Y}} \colon \Seg_{\mathcal{X}}(\mathcal{Y}) \to
  \CSS_{\mathcal{X}}(\mathcal{Y})\] exhibits
  $\CSS_{\mathcal{X}}(\mathcal{Y})$ as the localization of
  $\Seg_{\mathcal{X}}(\mathcal{Y})$ with respect to the fully faithful
  and essentially surjective morphisms.
\end{thm}

\begin{thm}[\cite{LurieGoodwillie}*{Proposition 1.3.2}]\label{thm:iterdistr}
  Suppose $\mathcal{X} \subseteq \mathcal{Y}$ is a distributor. Then
  so is $\mathcal{X} \subseteq \CSS_{\mathcal{X}}(\mathcal{Y})$, where
  we regard $\mathcal{X}$ as a full subcategory of
  $\CSS_{\mathcal{X}}(\mathcal{Y})$ via the diagonal embedding $c^{*}
  \colon \mathcal{X} \to \Fun(\simp^{\op}, \mathcal{X})$.
\end{thm}
We can therefore inductively define distributors $\mathcal{X}
\subseteq \CSS^{n}_{\mathcal{X}}(\mathcal{Y}) :=
\CSS_{\mathcal{X}}(\CSS^{n-1}_{\mathcal{X}}(\mathcal{Y}))$; we refer
to the objects of $\CSS^{n}_{\mathcal{X}}(\mathcal{Y})$ as
\emph{complete} $n$-fold $\mathcal{X}$-Segal objects in
  $\mathcal{Y}$.

\begin{defn}
  Let $\mathcal{X}$ be an \itopos{}. We write $\CSS^{n}(\mathcal{X})$
  for $\CSS^{n}_{\mathcal{X}}(\mathcal{X})$, which we may regard as a
  full subcategory of $\Seg_{n}(\mathcal{X})$. The inclusion
  $\CSS^{n}(\mathcal{X}) \hookrightarrow \Seg_{n}(\mathcal{X})$ has a
  left adjoint $L_{n,\mathcal{X}} \colon \Seg_{n}(\mathcal{X}) \to
  \CSS^{n}(\mathcal{X})$, obtained inductively as the composite
  \[ \Seg_{n}(\mathcal{X})
  \xto{\Seg_{\mathcal{X}}(L_{n-1,\mathcal{X}})}
  \Seg_{\mathcal{X}}(\CSS^{n-1}(\mathcal{X})) \xto{L_{\mathcal{X}
      \subseteq \CSS^{n-1}(\mathcal{X})}} \CSS^{n}(\mathcal{X}).\]
\end{defn}

\begin{remark}
  The \icat{} $\CSS^{n}(\mathcal{X})$ can be identified with the
  \icat{} of sheaves on $\mathcal{X}$ valued in the \icat{}
  $\CSS^{n}(\mathcal{S})$ of complete $n$-fold Segal spaces (in other
  words, sheaves of $(\infty,n)$-categories).
\end{remark}

We now prove the useful fact that the completion functor preserves
certain fibre products:
\begin{lemma}\label{lem:CSSlocprod}\ 
  \begin{enumerate}[(i)]
  \item Let $\mathcal{X} \subseteq \mathcal{Y}$ be a distributor. For
    $X \in \mathcal{X}$, write $c^{*}X$ for the constant diagram
    $\simp^{\op} \to \mathcal{Y}$ with value $X$; this is an
    $\mathcal{X}$-Segal object. The localization
    $L_{\mathcal{X}\subseteq\mathcal{Y}} \colon
    \Seg_{\mathcal{X}}(\mathcal{Y}) \to
    \CSS_{\mathcal{X}}(\mathcal{Y})$ preserves fibre products over
    $c^{*}X$ where $X \in \mathcal{X}$; in
    particular, $L_{\mathcal{X} \subseteq \mathcal{Y}}$ preserves
    products.
  \item  Let $\mathcal{X}$ be an \itopos{}. Then the localization $L_{n,\mathcal{X}} \colon
    \Seg_{n}(\mathcal{X}) \to \CSS^{n}(\mathcal{X})$ preserves products.
  \end{enumerate}
\end{lemma}
\begin{proof}
  Since the inclusion $\CSS_{\mathcal{X}}(\mathcal{Y}) \hookrightarrow
  \Seg_{\mathcal{X}}(\mathcal{Y})$ is a right adjoint, it preserves
  limits. Thus we must show that if $\mathcal{C}$ and $\mathcal{D}$
  are $\mathcal{X}$-Segal objects of $\mathcal{Y}$ over $c^{*}X$, then
  the natural map $L_{\mathcal{X}\subseteq\mathcal{Y}}(\mathcal{C}
  \times_{c^{*}X} \mathcal{D}) \to
  L_{\mathcal{X}\subseteq\mathcal{Y}}(\mathcal{C}) \times_{c^{*}X}
  L_{\mathcal{X}\subseteq\mathcal{Y}}(\mathcal{D})$ in
  $\Seg_{\mathcal{X}}(\mathcal{Y})$ is an equivalence. By
  Theorem~\ref{thm:distcssloc}, this is equivalent to proving that the
  map $\mathcal{C} \times_{c^{*}X} \mathcal{D} \to
  L_{\mathcal{X}\subseteq\mathcal{Y}}(\mathcal{C}) \times_{c^{*}X}
  L_{\mathcal{X}\subseteq\mathcal{Y}}(\mathcal{D})$ is fully faithful
  and essentially surjective. Condition (1) in the definition holds
  since pullbacks over $X$ preserve colimits in the \itopos{}
  $\mathcal{X}$, and the colimit in question is sifted, and condition
  (2) holds since limits commute. This proves (i); then (ii) follows
  inductively as the functor $L_{n,\mathcal{X}}$ is a composite of
  functors constructed from the functors in (i).
\end{proof}

As a consequence, we have:
\begin{lemma}
  The Cartesian product in $\CSS^{n}(\mathcal{X})$ preserves colimits
  separately in each variable.
\end{lemma}
\begin{proof}
  Colimits in $\CSS^{n}(\mathcal{X})$ are computed by applying the
  localization $L$ to the colimit of the same diagram in
  $\Seg_{n}(\mathcal{X})$. Thus the result follows by combining
  Lemma~\ref{lem:CSSlocprod} with the observation that the product
  preserves colimits in each variable in $\Seg_{n}(\mathcal{X})$.
\end{proof}
 
This lets us define internal Homs in complete Segal objects:
\begin{defn}\label{defn:MAP}
  We denote the internal Hom in $\CSS^{n}(\mathcal{X})$ of morphisms
  from $\mathcal{C}$ to $\mathcal{D}$ by
  $\mathcal{D}^{\mathcal{C}}$. If $X \in \mathcal{X}$ we abbreviate
  $\mathcal{D}^{c^{*}X}$ by $\mathcal{D}^{X}$. We also write
  $\MAP(\mathcal{C}, \mathcal{D})$ for the object of $\mathcal{X}$
  that represents the functor
  $\Map_{\CSS^{n}(\mathcal{X})}(\mathcal{C} \times c^{*}(\blank),
  \mathcal{D}) \colon \mathcal{X} \to \mathcal{S}$. Equivalently, this
  is just $(\mathcal{D}^{\mathcal{C}})_{0,\ldots,0}$.
\end{defn}

We now wish to prove a useful criterion for completeness of $n$-fold
Segal objects in an $\infty$-topos:
\begin{propn}\label{propn:complsegtoposcond} 
  Suppose $\mathcal{X}$ is an \itopos{} and $\mathcal{C}_{\bullet}$ is a Segal
  object in $\mathcal{X}$. Then $\mathcal{C}_{\bullet}$ is complete \IFF{} the
  Segal spaces $\Map_{\mathcal{X}}(X, \mathcal{C}_{\bullet})$ are complete for all
  $X \in \mathcal{X}$.
\end{propn}

For the proof it is convenient to first consider functoriality of
complete Segal objects in maps of distributors. The usual
notion of a map between \itopoi{} is that of a \emph{geometric
  morphism}: an adjunction where the left adjoint preserves finite
limits. The \icats{} of complete Segal objects are functorial with
respect to a slightly more general class of maps:
\begin{defn}
  Let $\mathcal{X}$ and $\mathcal{Y}$ be \itopoi{}. A
  \emph{pseudo-geometric morphism} from $\mathcal{X}$ to $\mathcal{Y}$
  is a functor $f_{*} \colon \mathcal{X} \to \mathcal{Y}$ such that
  $f_{*}$ admits a left adjoint $f^{*}$ which preserves pullbacks.
\end{defn}

\begin{defn}
  Let $\mathcal{X} \subseteq \mathcal{Y}$ and $\mathcal{X}' \subseteq
  \mathcal{Y}'$ be distributors. A \defterm{pseudo-geometric morphism} from
  $\mathcal{Y}$ to $\mathcal{Y}'$ is a functor $G \colon
  \mathcal{Y} \to \mathcal{Y}'$ such that:
  \begin{enumerate}[(1)]
  \item $G$ takes $\mathcal{X}$ to $\mathcal{X}'$.
  \item $G$ has a left adjoint $F \colon \mathcal{Y}' \to \mathcal{Y}$.
  \item $F$ takes $\mathcal{X}'$ to $\mathcal{X}$.
  \item If $\phi \colon \Delta^{1}\times \Delta^{1} \to \mathcal{Y}'$
    is a pullback diagram such that $\phi(1,1) \in \mathcal{X}'$, then
    $F(\phi)$ is a pullback diagram in $\mathcal{Y}$.
  \end{enumerate}
\end{defn}

\begin{remark}
  It is clear that a pseudo-geometric morphism of \itopoi{} as above
  is also a pseudo-geometric morphism of distributors.
\end{remark}

\begin{propn}\label{propn:CSSfunc}\ 
  \begin{enumerate}[(i)]
  \item Let $\mathcal{X} \subseteq \mathcal{Y}$ and $\mathcal{X}' \subseteq
  \mathcal{Y}'$ be distributors. Suppose $G \colon \mathcal{Y} \to
  \mathcal{Y}'$ is a pseudo-geometric morphism of distributors with left
  adjoint $F$. Then composition with $F$ and $G$ induces an adjunction
  \[  L F_{*} : \CSS_{\mathcal{X}'}(\mathcal{Y}')
  \rightleftarrows \CSS_{\mathcal{X}}(\mathcal{Y}) : G_{*}, \]
  and this is also a pseudo-geometric morphism.
\item Suppose $f^{*} \colon \mathcal{X}' \rightleftarrows \mathcal{X}
  : f_{*}$ is a
  pseudo-geometric morphism of \itopoi{}. Then the functors given by
  composition with $f^{*}$ and $f_{*}$ induce an adjunction
  \[ L_{n,\mathcal{X}}(f^{*})_{*} : \CSS^{n}(\mathcal{X}')
  \rightleftarrows \CSS^{n}(\mathcal{X}) : (f_{*})_{*}.\]
  \end{enumerate}
\end{propn}

\begin{proof}
  By Lemma~\ref{lem:funadj} we have an adjunction
  \[ F_{*} : \Fun(\simp^{\op}, \mathcal{Y}') \rightleftarrows
  \Fun(\simp^{\op}, \mathcal{Y}) : G_{*}. \]
  It is clear from the definition of a pseudo-geometric morphism that
  $F_{*}$ and $G_{*}$ preserve $\mathcal{X}'$- and $\mathcal{X}$-Segal
  objects, respectively, so there is an induced adjunction
    \[ F_{*} : \Seg_{\mathcal{X}'}(\mathcal{Y}') \rightleftarrows
  \Seg_{\mathcal{X}}(\mathcal{Y}) : G_{*}.\]
  We have a commutative diagram of left adjoints
  \csquare{\Gpd(\mathcal{X}')}{\Gpd(\mathcal{X})}{\Seg_{\mathcal{X}'}(\mathcal{Y}')}{\Seg_{\mathcal{X}}(\mathcal{Y}),}{F_{*}}{}{}{F_{*}}
  where the vertical morphisms denote the obvious inclusions, hence
  the corresponding diagram of right adjoints also commutes, giving an
  equivalence $G_{*}(\Gpd(\mathcal{C})) \simeq
  \Gpd(G_{*}\mathcal{C})$. It follows that $G_{*}$ preserves complete
  Segal objects, hence there is an induced adjunction 
  \[ L_{\mathcal{X} \subseteq\mathcal{Y}}F_{*} :
  \CSS_{\mathcal{X}'}(\mathcal{Y}') \rightleftarrows
  \CSS_{\mathcal{X}}(\mathcal{Y}) : G_{*}.\] To complete the proof of
  (i), we must show that this is a pseudo-geometric morphism. The
  functors $L_{\mathcal{X} \subseteq\mathcal{Y}}F_{*}$ and $G_{*}$
  preserve constant simplicial objects valued in $\mathcal{X}$ and
  $\mathcal{X}'$, so it remains to show that, given a pullback diagram
  \nolabelcsquare{\mathcal{C} \times_{X}
    \mathcal{D}}{\mathcal{C}}{\mathcal{D}}{c^{*}X} in
  $\CSS_{\mathcal{X}'}(\mathcal{Y}')$, where $c^{*}X$ is the constant
  simplicial object with value $X \in \mathcal{X}'$, its image under
  $L_{\mathcal{X}\subseteq\mathcal{Y}}F_{*}$ is also a pullback. Since
  limits in $\CSS_{\mathcal{X}'}(\mathcal{Y}')$ are computed in
  $\Seg_{\mathcal{X}'}(\mathcal{Y}')$, and these in turn are computed
  objectwise, it follows that $F_{*}$ takes this to a pullback diagram
  in $\Seg_{\mathcal{X}}(\mathcal{Y})$. Now applying
  Lemma~\ref{lem:CSSlocprod} we conclude that the image of this under
  $L_{\mathcal{X}\subseteq\mathcal{Y}}$ is also a pullback. This
  completes the proof of (i), and (ii) is just a special case of (i)
  obtained by induction.
\end{proof}

\begin{proof}[Proof of Proposition~\ref{propn:complsegtoposcond}]
  Let $r^{*} \colon \mathcal{S} \to \mathcal{X}$ denote the unique
  colimit-preserving functor such that $r^{*}(*)$ is a terminal object
  of $\mathcal{X}$, and let $r_{*} := \Map_{\mathcal{X}}(*, \blank)$
  be its right adjoint. By \cite{HTT}*{Proposition 6.3.4.1}, the
  adjunction $r^{*}\dashv r_{*}$ is a geometric morphism. It is clear
  that for any $X \in \mathcal{X}$ the functor $\Map_{\mathcal{X}}(X,
  \blank)$ has a left adjoint given by $X \times r^{*}(\blank)$, and
  this preserves pullbacks since $r^{*}$ preserves finite limits. Thus
  the adjunction $X \times r^{*}(\blank) \dashv \Map_{\mathcal{X}}(X, \blank)$ is a
  pseudo-geometric morphism of \itopoi{}, and so by the proof of
  Proposition~\ref{propn:CSSfunc} we have an equivalence
  $\Map_{\mathcal{X}}(X, \Gpd(\mathcal{C}_{\bullet})) \simeq
  \Gpd(\Map_{\mathcal{X}}(X, \mathcal{C}_{\bullet}))$. By the Yoneda
  Lemma the simplicial object
  $\Gpd(\mathcal{C}_{\bullet})$ is constant \IFF{}
  $\Map_{\mathcal{X}}(X, \Gpd(\mathcal{C}_{\bullet}))$ is constant for
  all $X \in \mathcal{X}$, so $\mathcal{C}_{\bullet}$ is complete
  \IFF{} $\Gpd(\Map_{\mathcal{X}}(X, \mathcal{C}_{\bullet}))$ is
  constant for all $X \in \mathcal{X}$, i.e.\ \IFF{}
  $\Map_{\mathcal{X}}(X, \mathcal{C}_{\bullet})$ is complete for all
  $X \in \mathcal{X}$.
\end{proof}

The remainder of this section is devoted to proving the following
inductive characterization of completeness for $n$-fold Segal spaces,
which we will make use of in the next section:
\begin{thm}\label{thm:CSScomplmapcond}
  Suppose $\mathcal{C}$ is an $n$-fold Segal space. Then the following
  are equivalent:
  \begin{enumerate}[(i)]
  \item $\mathcal{C}$ is complete.
  \item The Segal space $\mathcal{C}_{\bullet,0,\ldots,0}$ is
    complete, and the $(n-1)$-fold Segal spaces $\mathcal{C}(x,y)$ are
    complete for all objects $x, y$ in $\mathcal{C}$.
  \end{enumerate}
\end{thm}
\begin{remark}
  As a consequence of this, we get the expected inductive
  characterization of the fully faithful and essentially surjective
  morphisms in $\Seg_{n}(\mathcal{S})$, i.e.\ the morphisms that are
  inverted by the localization to $\CSS_{n}(\mathcal{S})$: They are
  precisely the morphisms $f \colon \mathcal{C} \to \mathcal{C}'$ that
  are \emph{essentially surjective}, in the sense that the underlying
  morphism of 1-fold Segal spaces is essentially surjective, and
  \emph{locally} fully faithful and essentially surjective, in the
  sense that for each pair of objects $x,y \in \mathcal{C}$ the
  induced morphism $\mathcal{C}(x,y) \to \mathcal{C}'(fx,fy)$ is a fully
  faithful and essentially surjective functor of $(n-1)$-fold Segal
  spaces.
\end{remark}

For convenience, we make the following inductive definition:
\begin{defn}
  Let $\mathcal{C}$ be an $n$-fold Segal space. We say that
  $\mathcal{C}$ is \emph{pseudo-complete} if
  \begin{enumerate}[(1)]
  \item the Segal space $\mathcal{C}_{\bullet,0,\ldots,0}$ is complete,
  \item the $(n-1)$-fold Segal spaces $\mathcal{C}(X,Y)$ are
    pseudo-complete for all objects $X, Y$ in $\mathcal{C}$.
  \end{enumerate}
\end{defn}
Our goal is then to show that an $n$-fold Segal space is complete
\IFF{} it is pseudo-complete. Before we give the proof we need to make
a number of observations:
\begin{lemma}\label{lem:seginCSScomplcond}
  Let $\mathcal{X} \subseteq \mathcal{Y}$ be a distributor, and
  suppose $\mathcal{C} \in
  \Seg_{\mathcal{X}}(\CSS^{n-1}_{\mathcal{X}}(\mathcal{Y}))$. Then
  $\mathcal{C}$ is in $\CSS^{n}_{\mathcal{X}}(\mathcal{Y})$ \IFF{} the
  Segal object $\mathcal{C}_{\bullet,0,\ldots,0}$ in
  $\Seg_{\mathcal{X}}(\mathcal{Y})$ is complete.
\end{lemma}
\begin{proof}
  The inclusion functor $\Seg \mathcal{X} \hookrightarrow
  \Seg_{\mathcal{X}}(\CSS^{n-1}_{\mathcal{X}}(\mathcal{Y}))$
  factors through the inclusion \[\Seg_{\mathcal{X}}(\mathcal{Y})
  \hookrightarrow
  \Seg_{\mathcal{X}}(\CSS^{n-1}_{\mathcal{X}}(\mathcal{Y}))\] induced by
  the functor $\mathcal{Y} \to \CSS^{n-1}_{\mathcal{X}}(\mathcal{Y})$
  that sends an object of $\mathcal{Y}$ to the constant
  $(n-1)$-simplicial object with that value. Thus the right adjoint
  $\Seg_{\mathcal{X}}\CSS^{n-1}_{\mathcal{X}}(\mathcal{Y}) \to
  \Seg(\mathcal{X})$ is the composite of the right adjoint
  $\Seg_{\mathcal{X}}\CSS^{n-1}_{\mathcal{X}}(\mathcal{Y}) \to
  \Seg_{\mathcal{X}}(\mathcal{Y})$, which is induced by evaluation at the initial
  object $([0],\ldots,[0]) \in \simp^{n-1,\op}$, and the
  localization $\Seg_{\mathcal{X}}(\mathcal{Y}) \to
  \Seg(\mathcal{X})$. In particular, the groupoid object
  $\Gpd(\mathcal{C})$ is equivalent to
  $\Gpd(\mathcal{C}_{\bullet,0,\ldots,0})$ and so $\mathcal{C}$ is
  complete \IFF{} $\mathcal{C}_{\bullet,0,\ldots,0}$ is.
\end{proof}

\begin{lemma}\label{lem:segcompl1cond}
  Let $\mathcal{C}$ be an $n$-fold Segal space. Then the following are
  equivalent:
  \begin{enumerate}[(i)]
  \item $\mathcal{C}$ is complete.
  \item The Segal space $\mathcal{C}_{\bullet,0,\ldots,0}$ is
    complete, and the $(n-1)$-fold Segal space
    $\mathcal{C}_{1,\bullet,\ldots,\bullet}$ is complete.
  \end{enumerate}
\end{lemma}
\begin{proof}
  By Lemma~\ref{lem:seginCSScomplcond} we know that $\mathcal{C}$ is
  complete \IFF{} $\mathcal{C}_{\bullet,0,\ldots,0}$ is complete and
  the $(n-1)$-fold Segal spaces
  $\mathcal{C}_{n,\bullet,\ldots,\bullet}$ are complete for each
  $n$. But $\mathcal{C}_{0,\bullet,\ldots,\bullet}$ is constant and so
  obviously complete, and thus for $n > 1$ the Segal condition implies
  that $\mathcal{C}_{n,\bullet,\ldots,\bullet}$ is complete if
  $\mathcal{C}_{1,\bullet,\ldots,\bullet}$ is complete, since complete
  Segal objects are closed under limits in the \icat{} of presheaves.
\end{proof}

\begin{remark}
  Applying Lemma~\ref{lem:segcompl1cond} inductively, we see that an
  $n$-fold Segal space $\mathcal{C}$ is complete \IFF{} the $n$ Segal
  spaces
  $\mathcal{C}_{\bullet,0,\ldots,0},
  \mathcal{C}_{1,\bullet,0,\ldots,0}, \ldots,
  \mathcal{C}_{1,\ldots,1,\bullet}$ are all complete; this is the
  definition of completeness used in
  \cite{BarwickSchommerPriesUnicity}. An alternative proof that this
  agrees with the definition of completeness we gave above is found in
  \cite{JohnsonFreydScheimbauerLax}*{Lemma 2.8}.
\end{remark}

\begin{lemma}\label{lem:CtoXmaps}
  Suppose given an $n$-fold Segal space $\mathcal{C}$ together with a
  map $\pi \colon \mathcal{C} \to X$ where $X$ is a constant Segal
  space, and for $x \in X$ let $\mathcal{C}_{x}$ denote the $n$-fold
  Segal space that is the fibre of $\pi$ at $x$. Then for any two
  objects $c, d \in \mathcal{C}$ there is a map $\mathcal{C}(c,d) \to
  \Omega_{\pi(c),\pi(d)}X$ whose fibres are of the form
  $\mathcal{C}_{\pi(c)}(c, d')$, where $d'$ is the image of $d$ in
  $\mathcal{C}_{\pi(c)}$ under the equivalence $\mathcal{C}_{\pi(c)}
  \simeq \mathcal{C}_{\pi(d)}$ determined by the path
  from $\pi(c)$ to $\pi(d)$.
\end{lemma}
\begin{proof}
  The map $\pi$ gives a commutative square
  \nolabelcsquare{\mathcal{C}_{1}}{X}{\Ob(\mathcal{C})^{\times
      2}}{X^{\times 2}} so taking fibres over a point $(c,d) \in
  \Ob(\mathcal{C})^{\times 2}$ we get a map $\mathcal{C}(c,d) \to
  \Omega_{(\pi c, \pi d)}X$. If $\pi c$ and $\pi d$ are not in the
  same component then both spaces are empty and we are
  done. Otherwise we want to identify the fibre of this map at a point
  $p \in \Omega_{(\pi c, \pi d)}X$. Consider the commutative diagram
  \[
  \begin{tikzcd}
    \mathcal{C}(c,d)_{p} \arrow{r} \arrow{d} &
    \mathcal{C}(c,d)\arrow{d} \arrow{r} & \mathcal{C}_{1} \arrow{d}
    \arrow{dr} \\
    * \arrow{r}{p} & \Omega_{(\pi c, \pi d)}X \arrow{r} \arrow{d}& Y \arrow{r} \arrow{d}& X \arrow{d}{\Delta}\\
    & * \arrow{r} & \Ob(\mathcal{C})^{\times 2} \arrow{r}& X^{\times 2},
  \end{tikzcd}
  \]
  where all four squares are pullbacks (and $Y$ is defined as the
  pullback of $X \to X^{\times 2}$ and $\Ob(\mathcal{C})^{\times 2}
  \to X^{\times 2}$). Thus we have a commutative diagram
  \[
  \begin{tikzcd}
    \mathcal{C}(c,d)_{p} \arrow{r} \arrow{d} & \mathcal{C}_{1}
    \arrow{d} \\
    * \arrow{r} \arrow{dr} & Y \arrow{d} \\
     & X,
  \end{tikzcd}
  \]
  where the top square is a pullback. Taking fibres at the given  point of
  $X$ (which we may identify with $\pi c$) we can factor this
  diagram as
  \[
  \begin{tikzcd}
    \mathcal{C}(c,d)_{p} \arrow{r} \arrow{d} & (\mathcal{C}_{\pi
      c})_{1} \arrow{r} \arrow{d} & \mathcal{C}_{1} \arrow{d}\\
    *\arrow{r} & (\mathcal{C}_{\pi c})_{0}^{\times 2} \arrow{d} \arrow{r}& Y \arrow{d}\\
      & * \arrow{r} & X,
  \end{tikzcd}
  \]
  where the two right-hand squares are pullbacks. Then as the top
  composite square is a pullback, so is the top left
  square. Thus we have identified $\mathcal{C}(c,d)_{p}$ with a
  mapping space in $\mathcal{C}_{\pi c}$, as required.
\end{proof}

\begin{lemma}\label{lem:pseudocomplfibres}
  Let $\pi \colon \mathcal{C} \to X$ be a map of $n$-fold Segal
  spaces, where $X$ is constant. Suppose the fibres $\mathcal{C}_{x}$
  are pseudo-complete for each $x \in X$. Then $\mathcal{C}$ is also
  pseudo-complete.
\end{lemma}
\begin{proof}
  We prove this by induction on $n$. The map $\pi$ induces a
  commutative diagram \nolabelopctriangle{\Map(E^{1},
    \mathcal{C}_{\bullet,0,\ldots,0})}{\mathcal{C}_{0,\ldots,0}}{X}
  where the map on fibres at $x \in X$ is $\Map(E^{1},
  (\mathcal{C}_{x})_{\bullet,0,\ldots,0}) \to
  (\mathcal{C}_{x})_{0,\ldots,0}$ as $\Map(E^{1},\blank)$ commutes
  with limits. This map is an equivalence for all
  $x$, since $\mathcal{C}_{x}$ is pseudo-complete, hence the
  horizontal map in the triangle is also an equivalence and thus
  $\mathcal{C}_{\bullet,0,\ldots,0}$ is complete.

  Now given objects $c, d \in \mathcal{C}$, by
  Lemma~\ref{lem:CtoXmaps} there is a map $\mathcal{C}(c,d) \to
  \Omega_{\pi(c),\pi(d)}X$ whose fibres are mapping $(n-1)$-fold Segal
  spaces in $\mathcal{C}_{\pi(c)}$. By assumption these are
  pseudo-complete, hence by the inductive hypothesis
  $\mathcal{C}(c,d)$ is pseudo-complete for all $c$, $d$. This
  completes the proof. 
\end{proof}

\begin{proof}[Proof of Theorem~\ref{thm:CSScomplmapcond}]
  We will show, by induction on $n$, that an $n$-fold Segal space is
  complete \IFF{} it is pseudo-complete. For $n = 1$ the two notions
  coincide, so there is nothing to prove. Suppose we have shown that
  they agree for $n < k$, and let $\mathcal{C}$ be a $k$-fold Segal
  space. By Lemma~\ref{lem:segcompl1cond} $\mathcal{C}$ is complete
  \IFF{} $\mathcal{C}_{\bullet,0,\ldots,0}$ is complete and
  the $(k-1)$-fold Segal space $\mathcal{C}_{1}$ is complete. Since by
  assumption the notions of complete and pseudo-complete $(k-1)$-fold
  Segal spaces coincide, it remains to show that $\mathcal{C}_{1}$ is
  complete \IFF{} the fibres $\mathcal{C}(c,d)$ at $(c,d) \in
  \Ob(\mathcal{C})^{\times 2}$ are all complete. One direction follows
  from applying Lemma~\ref{lem:pseudocomplfibres} to the map
  $\mathcal{C}_{1} \to \Ob(\mathcal{C})^{\times 2}$, and the other follows
  since complete $(k-1)$-fold Segal spaces are closed under
  limits in $(k-1)$-fold Segal spaces.  
\end{proof}

\section{Completeness for Iterated Spans}\label{sec:complitspan}
In this section we will prove that the $n$-fold Segal spaces
$\Span_{n}(\mathcal{C})$ that we constructed above are always
complete. We first consider the case $n = 1$, which is due to Barwick:
\begin{propn}[\cite{BarwickMackey}*{Proposition 3.4}]\label{propn:Span1comp}
  Let $\mathcal{C}$ be an \icat{} with finite limits. Then the Segal
  space $\Span_{1}(\mathcal{C})$ is complete.
\end{propn}
For completeness we include a slightly different proof than that of
Barwick, based on the following observation:
\begin{lemma}\label{lem:Span1eq}
  A span $X \xfrom{f} A \xto{g} Y$ in $\mathcal{C}$ is an equivalence in
  $\Span_{1}(\mathcal{C})$ \IFF{} the maps $f$ and $g$ are
  equivalences.
\end{lemma}
\begin{proof}
  It is clear that a span is an equivalence if both the maps in it are
  equivalences, so it remains to prove the converse. Suppose $Y \xfrom{h}
  B \xto{k} X$ is an inverse, then since composing gives the identity we
  have a diagram
    \[ %
\begin{tikzpicture} %
\matrix (m) [matrix of math nodes,row sep=1em,column sep=1em,text height=1.5ex,text depth=0.25ex] %
{  &       & X &       &  \\
   & A &       & B &   \\
X &   & Y &       & X,   \\ };
\path[->,font=\footnotesize] %
(m-1-3) edge node [above left] {$f'$} (m-2-2) 
(m-1-3) edge node [above right] {$k'$} (m-2-4)
(m-2-2) edge node [above left] {$f$} (m-3-1)
(m-2-2) edge node [above right] {$g$}(m-3-3)
(m-2-4) edge node [above left] {$h$} (m-3-3)
(m-2-4) edge node [above right] {$k$} (m-3-5);
\end{tikzpicture}%
\]%
where the composites $ff' \colon X \to A \to X$ and $kk' \colon X \to B \to X$ are
equivalent to $\id_{X}$. Composing in the other order we similarly
have maps $h' \colon Y \to B$ and $g' \colon Y \to A$ such that $hh'
\colon Y \to B \to Y$ and $gg' \colon Y \to
A \to Y$ are equivalent to $\id_{Y}$. Now taking a double composite in
$\Span_{1}(\mathcal{C})$ we get a commutative diagram
  \[ %
\begin{tikzpicture} %
\matrix (m) [matrix of math nodes,row sep=1em,column sep=1em,text height=1.5ex,text depth=0.25ex] %
{
&  &       & A &       &  &\\
&  & X &       & Y &  &\\
&A   &  & B &  & A & \\
X  &   & Y &  & X & & Y.   \\ };
\path[->,font=\footnotesize] %
(m-1-4) edge node [above left] {$f''$} (m-2-3)
(m-1-4) edge node [above right] {$g''$} (m-2-5)
(m-2-3) edge node [above left] {$f'$} (m-3-2)
(m-2-3) edge node [above right] {$k'$} (m-3-4)
(m-2-5) edge node [above left] {$h'$} (m-3-4)
(m-2-5) edge node [above right] {$g'$} (m-3-6)
(m-3-2) edge node [above left] {$f$} (m-4-1)
(m-3-2) edge node [above right] {$g$}  (m-4-3)
(m-3-4) edge node [above left] {$h$} (m-4-3)
(m-3-4) edge node [above right] {$k$} (m-4-5)
(m-3-6) edge node [above left] {$f$} (m-4-5)
(m-3-6) edge node [above right] {$g$} (m-4-7);
\end{tikzpicture}%
\]%
Since the composite is the original span $X \xfrom{f} A \xto{g} Y$, the
composite along the left edge is the original map $f \colon A \to X$. But since
the composite $ff' \simeq \id_{X}$, this means that $f'' \simeq f$,
and similarly $g'' \simeq g$. But then since
compositions are formed by taking pullbacks, we have a diagram
\[
\begin{tikzcd}
  A \arrow{r}{f} \arrow{d}{g} & X \arrow{r}{f'} \arrow{d}{k'} & A \arrow{d}{g} \\
  Y\arrow{r}{h'} & B\arrow{r}{h} & Y,
\end{tikzcd}
\]
where both squares are Cartesian. The composite square is also
Cartesian, and so as $hh' \simeq \id_{Y}$, we have $f'f \simeq
\id_{A}$. Thus $f'$ is a two-sided inverse of $f$, hence $f$ is an
equivalence. Similarly, we get $g'g \simeq \id_{A}$ and so $g$ is
also an equivalence. 
\end{proof}

\begin{proof}[Proof of Proposition~\ref{propn:Span1comp}]
  By Theorem~\ref{thm:rezkcompl} the space of equivalences in
  $\Span_{1}(\mathcal{C})$ consists of the components of
  the space $\Span_{1}(\mathcal{C})_{1} \simeq \Map(\bbS^{1}, \mathcal{C})$
  that correspond to equivalences. But by Lemma~\ref{lem:Span1eq} these
  are precisely the diagrams that land in the underlying
  $\infty$-groupoid $\iota \mathcal{C}$ of $\mathcal{C}$, so the space
  in question is $\Map(\bbS^{1}, \iota\mathcal{C})$. This is
  equivalent to the space of maps to $\iota \mathcal{C}$ from the
  $\infty$-groupoid $\|\bbS^{1}\|$ obtained by inverting the morphisms in
  $\bbS^{1}$. This is contractible, so the space of equivalences is
  equivalent to $\iota \mathcal{C} \simeq \Span_{1}(\mathcal{C})_{0}$
  as required.
\end{proof}

To extend this to iterated spans, we first identify the mapping
$(\infty,n-1)$-categories in $\Span_{n}(\mathcal{C})$:
\begin{propn}\label{propn:spanmaps}
  Let $\mathcal{C}$ be an \icat{} with finite limits. If $X$ and $Y$
  are objects of $\mathcal{C}$, then the $(k-1)$-fold Segal space
  $\Span_{k}(\mathcal{C})(X,Y)$ of maps from $X$ to $Y$ in
  $\Span_{k}(\mathcal{C})$ is equivalent to
  $\Span_{k-1}(\mathcal{C}_{/X\times Y})$.
\end{propn}

For the proof we need a simple observation:
\begin{lemma}\label{lem:overXYpullback}
  Suppose $X$ and $Y$ are objects of an \icat{} $\mathcal{C}$ that
  have a product $X \times Y$. Then for any \icat{} $\mathcal{K}$ there is a
  natural pullback square
  \nolabelcsquare{\Map(\mathcal{K}, \mathcal{C}_{/X \times Y})}{\Map(\bbS^{1} \times \mathcal{K},
    \mathcal{C})}{\{(c_{X}, c_{Y})\}}{\Map(\bbS^{0} \times \mathcal{K},
    \mathcal{C})^{\times 2},}
  where $c_{X}$ and $c_{Y}$ denote the functors constant at $X$ and $Y$.
\end{lemma}
\begin{proof}
  Since $X \times Y$ is a product, the \icat{} $\mathcal{C}_{/X \times
    Y}$ is equivalent to $\mathcal{C}_{/p}$ where $p$ is the diagram
  $\{0,1\} \to \mathcal{C}$ sending $0$ to $X$ and $1$ to $Y$. The
  \icat{} $\mathcal{C}_{/p}$ has the universal property that for all
  \icats{} $\mathcal{K}$ there are natural pullback squares
  \nolabelcsquare{\Map(\mathcal{K}, \mathcal{C}_{/p})}{\Map(\mathcal{K} \star \{0,1\},
    \mathcal{C})}{\{(X,Y)\}}{\Map(\{0,1\}, \mathcal{C}).}  There is an
  evident equivalence between $\bbS^{1}$ and
  $\{0,1\}^{\triangleleft}$, i.e.\ $* \amalg_{\{0,1\} \times \{0\}}
  \{0,1\} \times [1]$, and since products in $\CatI$ preserve
  colimits this gives an equivalence
 \[ \mathcal{K} \times \bbS^{1} \simeq \mathcal{K} \amalg_{\mathcal{K} \times \{0,1\} \times \{0\}} \mathcal{K} \times \{0,1\} \times
 [1].\]
 Moreover, the \icat{} $\mathcal{K} \star \{0,1\}$ is equivalent to the pushout (in $\CatI$)
 \[ \mathcal{K} \amalg_{\mathcal{K} \times \{0,1\} \times \{0\}} \mathcal{K} \times \{0,1\} \times
 [1] \amalg_{\mathcal{K} \times \{0,1\} \times \{1\}} \{0,1\},\] thus we
 get a pullback square \nolabelcsquare{\Map(\mathcal{K} \star \{0,1\},
   \mathcal{C})}{\Map(\mathcal{K} \times \bbS^{1}, \mathcal{C})}{\Map(\{0,1\},
   \mathcal{C})}{\Map(\mathcal{K} \times \{0,1\}, \mathcal{C}).}  Putting these
 two pullback squares together then completes the proof.
\end{proof}

\begin{proof}[Proof of Proposition~\ref{propn:spanmaps}]
  By Lemma~\ref{lem:overXYpullback} we have natural pullback squares
  \nolabelcsquare{\oSPAN_{k-1}(\mathcal{C}_{/X \times
      Y})_{[n_{1}],\ldots,[n_{k-1}]}}{\oSPAN_{k}(\mathcal{C})_{[1],[n_{1}],\ldots,[n_{k-1}]}}{\{(c_{X},c_{Y})\}}{\oSPAN_{k}(\mathcal{C})_{[0],[n_{1}],\ldots,[n_{k-1}]}^{\times
      2}.}  Using Lemma~\ref{lem:cartindcond} we see that this
  restricts to a pullback square
  \nolabelcsquare{\SPAN_{k-1}(\mathcal{C}_{/X \times
      Y})_{[n_{1}],\ldots,[n_{k-1}]}}{\SPAN_{k}(\mathcal{C})_{[1],[n_{1}],\ldots,[n_{k-1}]}}{\{(c_{X},c_{Y})\}}{\SPAN_{k}(\mathcal{C})_{[0],[n_{1}],\ldots,[n_{k-1}]}^{\times
      2},} since a functor to $\mathcal{C}_{/X \times Y}$ is Cartesian
  \IFF{} its composite with $\mathcal{C}_{/X \times Y} \to
  \mathcal{C}$ is Cartesian, as this forgetful functor detects
  pullbacks. Thus we have a pullback square of $(k-1)$-uple Segal
  spaces 
  \nolabelcsquare{\SPAN_{k-1}(\mathcal{C}_{/X \times
      Y})}{\SPAN_{k}(\mathcal{C})_{1}}{\{(X,Y)\}}{\SPAN_{k}(\mathcal{C})_{0}^{\times
      2}.}
  The functor $U^{k-1}_{\Seg}$ is a right adjoint, so applying it we get a
  pullback square
\nolabelcsquare{\Span_{k-1}(\mathcal{C}_{/X \times
      Y})}{U_{\Seg}^{k-1}(\SPAN_{k}(\mathcal{C})_{1})}{\{(X,Y)\}}{U_{\Seg}^{k-1}(\SPAN_{k}(\mathcal{C})_{0})^{\times
      2}.}
  On the other hand, our constuction of $U_{\Seg}^{k}$ gives us a
  pullback square
  \nolabelcsquare{\Span_{k}(\mathcal{C})_{1}}{U_{\Seg}^{k-1}(\SPAN_{k}(\mathcal{C})_{1})}{\Span_{k}(\mathcal{C})_{0}^{\times
    2}}{U_{\Seg}^{k-1}(\SPAN_{k}(\mathcal{C})_{0})^{\times
      2}.}
  The map $\{(X,Y)\} \to U_{\Seg}^{k-1}\SPAN_{k}(\mathcal{C})_{0}^{\times
      2}$ factors through the constant $(k-1)$-simplicial object
    $\Span_{k}(\mathcal{C})_{0}^{\times 2}$, so we get a commutative diagram
   \[
   \begin{tikzcd}
     \Span_{k-1}(\mathcal{C}_{/X \times
      Y}) \arrow{r} \arrow{d}& \Span_{k}(\mathcal{C})_{1}
    \arrow{r}\arrow{d} & U_{\Seg}^{k-1}(\SPAN_{k}(\mathcal{C})_{1})
    \arrow{d} \\
    \{(X,Y)\} \arrow{r} & \Span_{k}(\mathcal{C})_{0}^{\times
    2} \arrow{r} &  U_{\Seg}^{k-1}\SPAN_{k}(\mathcal{C})_{0}^{\times
      2}
   \end{tikzcd}
   \]
   where the right-hand square and the composite square are both
   Cartesian. Then the left-hand square is also Cartesian, which
   completes the proof.
\end{proof}

\begin{cor}
  Let $\mathcal{C}$ be an \icat{} with finite limits. Then the
  $n$-fold Segal space $\Span_{n}(\mathcal{C})$ is complete.
\end{cor}
\begin{proof}
  We prove this by induction on $n$. The case $n = 1$ is
  Proposition~\ref{propn:Span1comp}. Suppose the result holds for all
  $n < k$. By Theorem~\ref{thm:CSScomplmapcond} to show that
  $\Span_{k}(\mathcal{C})$ is complete it suffices to prove that the
  Segal space $\Span_{k}(\mathcal{C})_{\bullet,0,\ldots,0}$ is
  complete, and the $(n-1)$-fold Segal spaces
  $\Span_{k}(\mathcal{C})(X,Y)$ are complete for all $X$, $Y$ in
  $\mathcal{C}$. But $\Span_{k}(\mathcal{C})_{\bullet,0,\ldots,0}$ is
   equivalent to $\Span_{1}(\mathcal{C})$, which we know is
  complete, and by Proposition~\ref{propn:spanmaps} we can identify
  $\Span_{k}(\mathcal{C})(X,Y)$ with $\Span_{k-1}(\mathcal{C}_{/X
    \times Y})$, which is complete by the inductive hypothesis.
\end{proof}

\section{Completeness for Iterated Spans with Local Systems}\label{sec:compllocsys}
In this section we will show that that the $k$-fold Segal space
$\Span_{k}(\mathcal{X}; \mathcal{C})$ is complete, provided
$\mathcal{X}$ is an $\infty$-topos and $\mathcal{C}$ is a complete
$k$-fold Segal object in $\mathcal{X}$. We first consider the case $k
= 1$, which follows from the following observation:
\begin{lemma}\label{lem:eqinSpanXC}
  Let $\mathcal{X}$ be an \itopos{} and $\mathcal{C}$ a Segal object
  in $\mathcal{X}$. Then there is an equivalence
  \[\Span_{1}(\mathcal{X}; \mathcal{C})_{\txt{eq}} \simeq \coprod_{X
    \in \pi_{0}\mathcal{X}} \Map_{\mathcal{X}}(X, \mathcal{C})_{\txt{eq}}.\]
\end{lemma}
\begin{proof}
  By definition, $\Span_{1}(\mathcal{X}; \mathcal{C})_{\txt{eq}}$ is
  the subspace of $\Span_{1}(\mathcal{X}; \mathcal{C})_{1} \simeq
  \iota\Fun(\bbS^{1}, \mathcal{C})_{/s_{\mathcal{C}}([1])}$ consisting
  of those components that correspond to equivalences. The forgetful
  functor $\Span_{1}(\mathcal{X}; \mathcal{C}) \to
  \Span_{1}(\mathcal{X})$ induces a map $\Span_{1}(\mathcal{X};
  \mathcal{C})_{\txt{eq}} \to \Span_{1}(\mathcal{X})_{\txt{eq}}$, so
  by Proposition~\ref{propn:Span1comp} the underlying span of an
  equivalence is trivial. We may thus identify $\Span_{1}(\mathcal{X};
  \mathcal{C})_{\txt{eq}}$ with a collection of components in
  $\coprod_{X \in \pi_{0}(\mathcal{X})} \Map_{\mathcal{X}}(X,
  \mathcal{C}_{1})$. It is moreover immediate from the definition of
  composition in $\Span_{1}(\mathcal{X}; \mathcal{C})$ that a map 
  $X \to \mathcal{C}_{1}$ has an inverse when viewed as a morphism in $\Span_{1}(\mathcal{X};
  \mathcal{C})$ \IFF{} it has one when viewed as a morphism in 
  $\Map_{\mathcal{X}}(X, \mathcal{C}_{\bullet})$, which completes the proof.
\end{proof}

\begin{propn}\label{propn:span1XCcompl}
  Suppose $\mathcal{X}$ is an \itopos{} and $\mathcal{C}$ is a
  complete Segal object in $\mathcal{X}$. Then $\Span_{1}(\mathcal{X};
  \mathcal{C})$ is a complete Segal space.
\end{propn}
\begin{proof}
  By Theorem~\ref{thm:rezkcompl} it suffices to show that the
  degeneracy map \[\Span(\mathcal{X}; \mathcal{C})_{0} \to
  \Span(\mathcal{X}; \mathcal{C})_{\txt{eq}}\] is an equivalence. By
  Lemma~\ref{lem:eqinSpanXC} we may identify this with the map
  \[\coprod_{X} \Map(X, \mathcal{C}_{0}) \to \coprod_{X} \Map(X,
  \mathcal{C})_{\txt{eq}}.\] But by
  Proposition~\ref{propn:complsegtoposcond} the Segal spaces $\Map(X,
  \mathcal{C})$ are complete since $\mathcal{C}$ is complete, and by
  Theorem~\ref{thm:rezkcompl} it follows that for each $X$ the map
  $\Map(X, \mathcal{C}_{0}) \to \Map(X, \mathcal{C})_{\txt{eq}}$ is an
  equivalence.
\end{proof}

To extend this to iterated Segal spaces, we first identify the mapping
$(\infty,k-1)$-categories of $\Span_{k}(\mathcal{X}; \mathcal{C})$:
\begin{propn}\label{propn:Famkmaps}
  Suppose $\mathcal{C}$ is a $k$-fold Segal object in $\mathcal{X}$,
  and that $\xi \colon X \to \mathcal{C}_{0,\ldots,0}$ and $\eta \colon Y \to
  \mathcal{C}_{0,\ldots,0}$ are objects of $\Span_{k}(\mathcal{X};
  \mathcal{C})$. Then the $(k-1)$-fold Segal space
  $\Span_{k}(\mathcal{X}; \mathcal{C})(\xi, \eta)$ of maps from $\xi$
  to $\eta$ in $\Span_{k}(\mathcal{X}; \mathcal{C})$ is equivalent to
  $\Span_{k-1}(\mathcal{X}; \mathcal{C}_{\xi,\eta})$, where
  $\mathcal{C}_{\xi,\eta}$ is the $(k-1)$-fold Segal object given by
  the pullback square
  \csquare{\mathcal{C}_{\xi,\eta}}{\mathcal{C}_{1}}{X \times
    Y}{\mathcal{C}_{0}^{\times 2}.}{}{}{}{\xi \times \eta}
\end{propn}
To prove this, we first make the following observations:
\begin{lemma}\label{lem:prodcone}
  For any \icats{} $\mathcal{A}$ and $\mathcal{B}$, the natural map
    \[ (\mathcal{A}^{\triangleleft} \times \mathcal{B} \amalg_{\mathcal{A} \times \mathcal{B}} \mathcal{A} \times
  \mathcal{B}^{\triangleleft})^{\triangleleft} \to \mathcal{A}^{\triangleleft} \times
  \mathcal{B}^{\triangleleft}\]
  is an equivalence.
\end{lemma}
\begin{proof}
  Suppose $\mathfrak{A}$ and $\mathfrak{B}$ are quasicategories
  representing the \icats{} $\mathcal{A}$ and $\mathcal{B}$. Then the
  pushout of simplicial sets
  \[ \mathfrak{A}^{\triangleleft} \times \mathfrak{B} \amalg_{\mathfrak{A} \times \mathfrak{B}} \mathfrak{A} \times
  \mathfrak{B}^{\triangleleft}\]
  is a homotopy pushout in $\sSet^{J}$, since the maps are
  cofibrations. It therefore suffices to prove that the natural map
  \[ (\mathfrak{A}^{\triangleleft} \times \mathfrak{B}
  \amalg_{\mathfrak{A} \times \mathfrak{B}} \mathfrak{A} \times
  \mathfrak{B}^{\triangleleft})^{\triangleleft} \to
  \mathfrak{A}^{\triangleleft} \times \mathfrak{B}^{\triangleleft}\]
  is a weak equivalence in the Joyal model structure. In fact, we will
  show that this map is an isomorphism for all simplicial sets
  $\mathfrak{A}$ and $\mathfrak{B}$. Since both sides preserve
  colimits in $\mathfrak{A}$ and $\mathfrak{B}$, it suffices to
  consider the case $\mathfrak{A} = \Delta^{n}$, $\mathfrak{B} =
  \Delta^{m}$. Thus, as $(\Delta^{n})^{\triangleleft} \cong
  \Delta^{n+1}$, we must show that the map
  \[ (\Delta^{n+1} \times \Delta^{m} \amalg_{\Delta^{n} \times
    \Delta^{m}} \Delta^{n} \times \Delta^{m+1})^{\triangleleft} \to
  \Delta^{n+1} \times \Delta^{m+1} \] is an isomorphism for all $n,
  m$. Now $\Delta^{n+1} \times \Delta^{m+1}$ is the nerve of the
  category $[n+1] \times [m+1]$, which is the join (of
  categories) $[0] \star \mathbf{C}$, where $\mathbf{C}$ is the
  subcategory spanned by all objects except $(0, 0)$. The nerve
  functor takes the join of categories to the join of simplicial sets,
  so it follows that $\Delta^{n+1} \times \Delta^{m+1} \cong
  (\mathrm{N}\mathbf{C})^{\triangleleft}$. Moreover, the simplicial
  set $\mathrm{N}\mathbf{C}$ is the subcomplex of
  $\Delta^{n+1} \times \Delta^{m+1}$ containing only the simplices
  that do not have $(0,0)$ as a vertex, which we can identify with
  $\Delta^{n+1} \times \Delta^{m} \amalg_{\Delta^{n} \times
    \Delta^{m}} \Delta^{n} \times \Delta^{m+1}$.
\end{proof}

\begin{lemma}\label{lem:Famkmaplem}
  Let $\mathcal{A}$ be an \icat{}, and suppose given a functor
 $\mu \colon \mathcal{A} \times [1] \amalg_{\mathcal{A} \times \{1\}} \mathcal{A}^{\triangleleft} \to \mathcal{C}$ with
  limit $X \in \mathcal{C}$. Then there is a natural pullback diagram
  \nolabelcsquare{\mathcal{C}_{/X}}{\Fun(\mathcal{A}^{\triangleleft},
    \mathcal{C})_{/\alpha}}{\{\nu\}}{\Fun(\mathcal{A}, \mathcal{C})_{/\beta},}
  where $\alpha := \mu|_{\mathcal{A}^{\triangleleft}}$, $\beta :=
  \mu|_{\mathcal{A} \times\{1\}}$, and
  $\nu$ is the object corresponding to $\mu|_{\mathcal{A} \times [1]}$.
\end{lemma}
\begin{proof}
  Since $X$ is the limit of $\mu$, the \icat{} $\mathcal{C}_{/X}$ is
  equivalent to $\mathcal{C}_{/\mu}$. And as \[(\mathcal{A} \times
  [1] \amalg_{\mathcal{A}} \mathcal{A}^{\triangleleft})^{\triangleleft} \simeq \mathcal{A}^{\triangleleft}
  \times [1]\] by Lemma~\ref{lem:prodcone}, the \icat{} $\mathcal{C}_{/\mu}$ fits in a natural
  pullback square
  \nolabelcsquare{\mathcal{C}_{/\mu}}{\Fun(\mathcal{A}^{\triangleleft} \times
    [1], \mathcal{C})}{\{\mu\}}{\Fun(\mathcal{A} \times [1]
    \amalg_{\mathcal{A} \times \{1\}} \mathcal{A}^{\triangleleft}, \mathcal{C}).}  
  Now consider the commutative diagram
  \[
  \begin{tikzcd}
    \mathcal{C}_{/\mu} \arrow{r} \arrow{d} & \Fun(\mathcal{A}^{\triangleleft},
    \mathcal{C})_{/\alpha} \arrow{r} \arrow{d} & \Fun(\mathcal{A}^{\triangleleft} \times
    [1], \mathcal{C}) \arrow{d} \\
    \{\nu\} \arrow{r} & \Fun(\mathcal{A}, \mathcal{C})_{/\beta}
    \arrow{r} \arrow{d} & \Fun(\mathcal{A} \times [1]
    \amalg_{\mathcal{A} \times \{1\}} \mathcal{A}^{\triangleleft},
    \mathcal{C}) \arrow{r} \arrow{d} & \Fun(\mathcal{A} \times [1], \mathcal{C})\arrow{d} \\
     & \{\beta\} \arrow{r}& \Fun(\mathcal{A}^{\triangleleft}, \mathcal{C}) \arrow{r}&
     \Fun(\mathcal{A}, \mathcal{C}).
  \end{tikzcd}
  \]
  Here the bottom right square and bottom composite squares are
  pullbacks, hence so is the bottom left square. Then, as the
  composite square in the middle column is a pullback, the top right
  square is a pullback. Finally, as the top composite square is a
  pullback, this implies the top left square is a pullback, as required.
\end{proof}

\begin{proof}[Proof of Proposition~\ref{propn:Famkmaps}]
  Using Lemma~\ref{lem:Famkmaplem}, we have natural pullback diagrams
  \nolabelcsquare{\Fun(\bbS^{n_1,\ldots,n_{k-1}},
    \mathcal{X})_{/s_{\mathcal{C}_{\xi,\eta}}(n_1,\ldots,n_{k-1})}}{\Fun(\bbS^{1,n_1,\ldots,n_{k-1}},
    \mathcal{X})_{/s_{\mathcal{C}}(1,n_1,\ldots,n_{k-1})}}{\{(c_{\xi},c_{\eta})\}}{\Fun(\bbS^{0,n_1,\ldots,n_{k-1}},
    \mathcal{X})_{/s_{\mathcal{C}}(0,n_1,\ldots,n_{k-1})}^{\times 2}.}
  It now follows using Lemma~\ref{lem:cartindcond} that these restrict
  to pullback diagrams
  \nolabelcsquare{\Fun^{\Cart}(\bbS^{n_1,\ldots,n_{k-1}},
    \mathcal{X})_{/s_{\mathcal{C}_{\xi,\eta}}(n_1,\ldots,n_{k-1})}}{\Fun^{\Cart}(\bbS^{1,n_1,\ldots,n_{k-1}},
    \mathcal{X})_{/s_{\mathcal{C}}(1,n_1,\ldots,n_{k-1})}}{\{(c_{\xi},c_{\eta})\}}{\Fun^{\Cart}(\bbS^{0,n_1,\ldots,n_{k-1}},
    \mathcal{X})_{/s_{\mathcal{C}}(0,n_1,\ldots,n_{k-1})}^{\times 2}.}
  The functor $\iota$, which takes the underlying space of an \icat{},
  is a right adjoint and hence preserves limits; applying it we
  therefore get, by naturality, a pullback square of $(k-1)$-uple
  Segal spaces
  \nolabelcsquare{\SPAN_{k-1}(\mathcal{X};
    \mathcal{C}_{\xi,\eta})}{\SPAN_{k}(\mathcal{X};
    \mathcal{C})_{1}}{\{(\xi, \eta)\}}{\SPAN_{k}(\mathcal{X};
    \mathcal{C})_{0}^{\times 2}.}

  Now applying the right adjoint $U^{k-1}_{\Seg}$ we get from this a
  pullback square 
  \nolabelcsquare{\Span_{k-1}(\mathcal{X};
    \mathcal{C}_{\xi,\eta})}{U_{\Seg}^{k-1}(\SPAN_{k}(\mathcal{X};
    \mathcal{C})_{1})}{\{(\xi, \eta)\}}{U_{\Seg}^{k-1}(\SPAN_{k}(\mathcal{X};
    \mathcal{C})_{0})^{\times 2}.}
  As in the proof of Proposition~\ref{propn:spanmaps} we can now
  combine this with a pullback square we get from the construction
  of $U_{\Seg}^{k}$, to get a pullback square
  \nolabelcsquare{\Span_{k-1}(\mathcal{X};
    \mathcal{C}_{\xi,\eta})}{\Span_{k}(\mathcal{X};
    \mathcal{C})_{1}}{\{(c_{\xi}, c_{\eta})\}}{\Span_{k}(\mathcal{X};
    \mathcal{C})_{0}^{\times 2},} as required.
\end{proof}

\begin{remark}
  In particular, if $X \simeq Y \simeq *$, so that $\xi$ and $\eta$
  are determined by two objects $x$ and $y$ of $\mathcal{C}$, then
  $\Span_{k}(\mathcal{X}; \mathcal{C})(\xi, \eta) \simeq
  \Span_{k-1}(\mathcal{X}; \mathcal{C}(x,y))$.
\end{remark}

\begin{cor}
  Suppose $\mathcal{X}$ is an \itopos{} and $\mathcal{C}$ is a
  complete $k$-fold Segal object in $\mathcal{X}$. Then
  $\Span_{k}(\mathcal{X}; \mathcal{C})$ is a complete $k$-fold Segal space.
\end{cor}
\begin{proof}
  The case $k = 1$ is Proposition~\ref{propn:span1XCcompl}; we will
  prove the general case by induction on $k$. Suppose we know the
  result for $k$-fold Segal objects for all $k < n$. By
  Theorem~\ref{thm:CSScomplmapcond} to show that
  $\Span_{n}(\mathcal{X}; \mathcal{C})$ is complete it suffices to
  prove that the Segal space $\Span_{n}(\mathcal{X};
  \mathcal{C})_{\bullet,0,\ldots,0}$ is complete, and the $(n-1)$-fold
  Segal spaces $\Span_{n}(\mathcal{X}; \mathcal{C})(\xi,\eta)$ are
  complete for all $\xi$, $\eta$. But $\Span_{n}(\mathcal{X};
  \mathcal{C})_{\bullet,0,\ldots,0}$ is equivalent to
  $\Span_{1}(\mathcal{X}; \mathcal{C}_{\bullet,0,\ldots,0})$, which we
  know is complete, and by Proposition~\ref{propn:Famkmaps} we can
  identify $\Span_{n}(\mathcal{X}; \mathcal{C})(\xi,\eta)$ with
  $\Span_{n-1}(\mathcal{X}; \mathcal{C}_{\xi,\eta})$. The $(n-1)$-fold
  Segal object $\mathcal{C}_{\xi,\eta}$ is complete since
  complete $(n-1)$-fold Segal objects in $\mathcal{X}$ are closed under
  pullback, so $\Span_{n-1}(\mathcal{X}; \mathcal{C}_{\xi,\eta})$ is
  complete by the inductive hypothesis.
\end{proof}

\section{Symmetric Monoidal Structures on Iterated Segal Spaces}\label{sec:symmmon}
In this section we introduce ($n$-fold) monoids in \icats{}, with a
special case being ($n$-fold) monoidal structures on iterated Segal
spaces. In the limit, we use these to give a definition of $\infty$-fold
monoids and $\infty$-fold monoidal structures, which is the form in
which symmetric monoidal structures will show up below. We then show
that $n$-fold monoids are the same thing as $\mathbb{E}_{n}$-algebras
(i.e.\ algebras for the $\infty$-operad corresponding to the
little $n$-disc operad) and $\infty$-fold monoids are
commutative algebras, as defined in \cite{HA}; for this we assume the
reader has some familiarity with the formalism of $\infty$-operads,
but this discussion can easily be skipped as it is not needed in 
the rest of the paper.

\begin{defn}
  Suppose $\mathcal{C}$ is an \icat{} with finite products. An
  (\emph{associative}) \emph{monoid} in $\mathcal{C}$ is a simplicial
  object $A_{\bullet} \colon \simp^{\op} \to \mathcal{C}$ such that the Segal
  maps $A_{n} \to \prod_{i = 1}^{n} A_{1}$ (induced by the inert maps
  $\rho_{i} \colon \{i-1,i\} \to [n]$) are equivalences for all $n =
  0,1,\ldots$. We write $\Mon(\mathcal{C})$ for the full subcategory
  of $\Fun(\simp^{\op}, \mathcal{C})$ spanned by the monoids. If
  $\mathcal{C}$ is the \icat{} $\Seg_{k}(\mathcal{X})$ or
  $\CSS_{k}(\mathcal{X})$ we refer to monoids as \emph{monoidal}
  $k$-fold (complete) Segal objects.
\end{defn}

\begin{defn}
  We inductively define an \emph{$n$-fold monoid} in $\mathcal{C}$ to
  be an $(n-1)$-fold monoid in $\Mon(\mathcal{C})$. Unwinding this
  definition, we see that an $n$-fold monoid is a functor $A \colon
  \simp^{n,\op} \to \mathcal{C}$ such that the natural maps
  \[ A_{t_{1},\ldots,t_{n}} \to \prod_{i_{1} = 1}^{t_{1}} \cdots
  \prod_{i_{n} = 1}^{t_{n}} A_{1,\ldots,1}\] are equivalences for all
  $t_{1},\ldots,t_{n}$. We write $\Mon_{n}(\mathcal{C}) :=
  \Mon_{n-1}(\Mon(\mathcal{C}))$ for the \icat{} of $n$-fold monoids
  in $\mathcal{C}$. If $\mathcal{C}$ is the \icat{}
  $\Seg_{k}(\mathcal{X})$ or $\CSS_{k}(\mathcal{X})$ we refer to
  $n$-fold monoids as \emph{$n$-monoidal} $k$-fold (complete)
  Segal spaces.
\end{defn}

\begin{remark}
  Since the localization functor $\Seg_{k}(\mathcal{X}) \to
  \CSS_{k}(\mathcal{X})$ preserves products by
  Lemma~\ref{lem:CSSlocprod}, the completion of an $n$-monoidal
  $k$-fold Segal space is an $n$-monoidal $k$-fold complete Segal
  space.
\end{remark}

\begin{defn}
  We define the \icat{} $\Mon_{\infty}(\mathcal{C})$ of
  \emph{$\infty$-fold monoids} in $\mathcal{C}$ as the limit of the sequence
  \[ \cdots \to \Mon_{n}(\mathcal{C}) \to \Mon_{n-1}(\mathcal{C}) \to
  \cdots \Mon(\mathcal{C}) \to \mathcal{C},\] where the functors are
  given by evaluation at $[1]$ in the first factor of
  $\simp^{n,\op}$. Thus an $\infty$-fold monoid in $\mathcal{C}$ is a
  sequence $(A^{0},A^{1},\ldots)$ where $A^{n}$ is an $n$-fold monoid,
  such that $A^{n-1} \simeq A^{n}([1],\blank,\ldots,\blank)$ for each
  $n$. If $\mathcal{C}$ is the \icat{} $\Seg_{k}(\mathcal{X})$ or
  $\CSS_{k}(\mathcal{X})$ we refer to $\infty$-fold monoids as
  \emph{$\infty$-monoidal} $k$-fold (complete) Segal objects.
\end{defn}

\begin{remark}
  An associative monoid in $\mathcal{C}$ is the same thing as a
  category object $X$ in $\mathcal{C}$ such that $X_{0} \simeq
  *$. Because of this, we can extract a monoid $\Omega_{p}X$ from a
  pointed category object, i.e.\ a category object $X$ equipped with a
  map $p \colon * \to X$ (or equivalently a map
  $p \colon * \to X_{0}$)\footnote{This has been changed from the
    published version, which has an incorrect construction of
    $\Omega_{p}X$}: Let $i$ denote the inclusion
  $\{[0]\} \hookrightarrow \simp^{\op}$; then the functor
  $i^{*} \colon \Fun(\simp^{\op}, \mathcal{C}) \to \mathcal{C}$,
  given by evaluation at $[0]$, has a left adjoint $i_{!}$ and a right
  adjoint $i_{*}$, given by left and right Kan extension along
  $i$. From the formula for right Kan extensions we see 
  $i_{*}C$ is given by $(i_{*}C)_{n} \simeq C^{\times (n+1)}$ with
  face maps being projections and degeneracies diagonal maps. The map
  $p \colon * \to X_{0}$ induces a map
  $i_{*}p \colon * \simeq i_{*}* \to i_{*}X_{0}$, and so we can take
  the pullback \nolabelcsquare{\Omega_{p}X}{X}{*}{i_{*}X_0} in
  category objects. We think of the monoid $\Omega_{p}X$ as the monoid
  of endomorphisms of the object $p$ of $X$.
\end{remark}

\begin{remark}
  Suppose $\mathcal{X}$ is an $\infty$-topos. We have a composite
  functor $\Mon(\mathcal{X}) \to \Seg(\mathcal{X}) \to
  \CSS(\mathcal{X})$, where the first functor is the natural inclusion
  and the second is the localization functor. Then the induced functor
  $\Mon(\mathcal{X}) \to \CSS(\mathcal{X})_{*}$ to pointed complete
  Segal objects is fully faithful, and is left adjoint to the functor
  $\Omega \colon \CSS(\mathcal{X})_{*} \to \Mon(\mathcal{X})$ we just
  defined. This can be proved by the same argument as for
  \cite{enr}*{Theorem 6.3.2}; we do not recall this as we will not
  make any further use of this observation.
\end{remark}

\begin{remark}\label{rmk:Ekmonfromnplusk}
  Similarly, an $n$-fold monoid is the same thing as an $n$-fold Segal
  object $X$ such that
  $X_{0,\ldots,0} \simeq X_{1,0,\ldots,0} \simeq X_{1,\ldots,1,0}
  \simeq *$. We can again use this to extract an $n$-fold monoid from
  a pointed $n$-fold Segal object (or $n$-fold category object)
  $(X, p \colon * \to X)$.\footnote{This has been changed from the
    published version, which has an incorrect construction of
    $\Omega^{n}_{p}X$.} If $X' := X_{\bullet,\ldots,\bullet,0}$
  denotes the underlying $(n-1)$-fold Segal object (or $n$-fold
  category object), then we get an $n$-fold monoid $\Omega^{n}_{p}X$
  by taking the pullback
  \nolabelcsquare{\Omega^{n}_{p}X}{X}{*}{i_{*}X',} where $*$ denotes
  the terminal $n$-fold simplicial object and
  $i_{*}X'$ is obtained by taking the right Kan extension along
  $i \colon \{[0]\} \hookrightarrow \simp^{\op}$ in the last
  simplicial coordinate. We think of $\Omega^{n}_{p}X$ as the $n$-fold
  monoid of endomorphisms of the identity $(n-1)$-morphism of $p$. As
  in the case $n = 1$, if $\mathcal{X}$ is an $\infty$-topos it can be
  shown that
  $\Omega^{n} \colon \CSS_{n}(\mathcal{X})_{*} \to
  \Mon_{n}(\mathcal{X})$ has a fully faithful left adjoint $B^{n}$
  given by the composite
  $\Mon_{n}(\mathcal{X}) \to \Seg_{n}(\mathcal{X})_{*} \to
  \CSS_{n}(\mathcal{X})_{*}$ where the second functor is the
  localization.
\end{remark}

\begin{ex}
  Given an $(n+k)$-fold Segal space $X$ and an object $p \in
  X_{0,\ldots,0}$, we can extract an $n$-monoidal $k$-fold Segal
  space $\Omega^{n}_{p}X$.
\end{ex}

\begin{defn}\label{defn:symmmonfromseq}
  Given a pointed $n$-fold Segal object $(X,p)$ in $\mathcal{C}$, we
  can extract a pointed $(n-1)$-fold Segal object by taking the
  pullback \csquare{X'}{X_{1}}{*}{X_{0}^{\times 2}}{}{}{}{p} with $X'$
  pointed via the degeneracy by $s_{0} \circ p$. This defines a
  functor $\phi_{n} \colon \Seg_{n}(\mathcal{C})_{*} \to
  \Seg_{n-1}(\mathcal{C})_{*}$, which restricts to the forgetful
  functor $\Mon_{n}(\mathcal{C}) \to \Mon_{n-1}(\mathcal{C})$ we used
  above. The $(n-1)$-fold Segal object $\phi_{n}(X, p)$ is the
  mapping object $X(p,p)$. Let us define
  $\Seg_{\infty}(\mathcal{C})_{*}$ to be the limit of the sequence of
  \icats{}
  \[ \cdots \to \Seg_{n}(\mathcal{C})_{*} \to
  \Seg_{n-1}(\mathcal{C})_{*} \to \cdots \to \Seg(\mathcal{C})_{*} \to
  \mathcal{C}_{*}.\]
  The objects of $\Seg_{\infty}(\mathcal{C})_{*}$, which we will call
  \emph{infinite delooping sequences}, can then be described as sequences
  $((X^{0},p^{0}),(X^{1},p^{1}),\ldots)$ such that $(X^{n},p^{n})$ is
  a pointed $n$-fold Segal object, together with equivalences $(X^{n-1},p^{n-1}) \simeq
  \phi_{n}(X^{n},p^{n})$ for all $n$. The functors $\Omega^{n}$ we defined above
  sit in commutative diagrams
  \csquare{\Seg_{n}(\mathcal{C})_{*}}{\Mon_{n}(\mathcal{C})}{\Seg_{n-1}(\mathcal{C})_{*}}{\Mon_{n-1}(\mathcal{C}),}{\Omega^{n}}{\phi_{n}}{}{\Omega^{n-1}}
  and so taking the limit we get a functor $\Omega^{\infty} \colon
  \Seg_{\infty}(\mathcal{C})_{*} \to \Mon_{\infty}(\mathcal{C})$ that
  extracts an $\infty$-fold monoid from an infinite delooping sequence.
\end{defn}

\begin{ex}
  Given a sequence $(X^{n}, p^{n})$ where $(X^{n},p^{n})$ is a pointed
  $(k+n)$-fold Segal space, such that $(X^{n-1}, p^{n-1}) \simeq
  (X^{n}(p^{n},p^{n}), \id_{p^{n}})$, we can extract an $\infty$-monoidal $k$-fold Segal space as $\Omega^{\infty}_{p^{\bullet}}X^{\bullet}$.
\end{ex}

We now wish to compare our notions of $n$-fold monoids to
$\mathbb{E}_{n}$-algebras, where $\mathbb{E}_{n}$ is the
$\infty$-operad associated to the little $n$-disc operad. This is a
straightforward consequence of results proved in \cite{HA}:
\begin{propn} \label{propn:iopdcomp}
Let $\mathcal{C}$ be an \icat{} with finite products,
  and let $\mathcal{C}^{\times}$ denote the associated Cartesian
  symmetric monoidal \icat{} (see \cite{HA}*{\S 2.4.1}). Then:
  \begin{enumerate}[(i)]
  \item There is a natural equivalence $\Mon(\mathcal{C}) \simeq
    \Alg_{\mathbb{E}_{1}}(\mathcal{C}^{\times})$.
  \item For every integer $n$ there is a natural equivalence
    $\Mon_{n}(\mathcal{C}) \simeq
    \Alg_{\mathbb{E}_{n}}(\mathcal{C}^{\times})$; under this
    equivalence the forgetful map $\Mon_{n}(\mathcal{C}) \to
    \Mon_{n-1}(\mathcal{C})$ corresponds to the map induced by
    a map of $\infty$-operads $\mathbb{E}_{n-1} \to \mathbb{E}_{n}$.
  \item There is a natural equivalence $\Mon_{\infty}(\mathcal{C})\simeq
    \Alg_{\mathbb{E}_{\infty}}(\mathcal{C}^{\times})$, where
    $\mathbb{E}_{\infty}$ denotes the commutative $\infty$-operad.
  \end{enumerate}
\end{propn}
\begin{proof}
  (i) follows by combining \cite{HA}*{Proposition 4.1.2.10},
  \cite{HA}*{Proposition 2.4.2.5} and \cite{HA}*{Example 5.1.0.7}. Now
  we prove (ii) by induction: if it holds for $n-1$ we have an
  equivalence
  \[ \Mon_{n}(\mathcal{C}) \simeq \Mon(\Mon_{n-1}(\mathcal{C})) \simeq
  \Mon(\Alg_{\mathbb{E}_{n-1}}(\mathcal{C}^{\times})) \simeq
  \Alg_{\mathbb{E}_{1}}(\Alg_{\mathbb{E}_{n-1}}(\mathcal{C}^{\times})^{\times}).\]
  The universal property of the Boardman-Vogt tensor product
  (see \cite{HA}*{\S 2.2.5}) implies that this is naturally equivalent
  to $\Alg_{\mathbb{E}_{1} \otimes
    \mathbb{E}_{n-1}}(\mathcal{C})$. By the Additivity Theorem,
  \cite{HA}*{Theorem 5.1.2.2} the tensor product $\mathbb{E}_{1} \otimes
  \mathbb{E}_{n-1}$ is equivalent to $\mathbb{E}_{n}$, so we obtain
  a natural equivalence $\Mon_{n}(\mathcal{C}) \simeq
  \Alg_{\mathbb{E}_{n}}(\mathcal{C}^{\times})$. Moreover, under this
  equivalence the forgetful functor $\Mon_{n}(\mathcal{C}) \to
  \Mon_{n-1}(\mathcal{C})$ corresponds to the map induced by
  \[\mathbb{E}_{n-1} \simeq \txt{Triv} \otimes \mathbb{E}_{n-1}\to
  \mathbb{E}_{1} \otimes \mathbb{E}_{n-1} \simeq \mathbb{E}_{n},\]
  where $\txt{Triv} \to \mathbb{E}_{1}$ is the obvious map from the
  trivial operad corresponding to the forgetful functor from
  $\mathbb{E}_{1}$-algebras in $\mathcal{C}$ to $\mathcal{C}$.

  Taking the limit of the equivalences in (iii) as $n$ goes to
  $\infty$, we get an equivalence $\Mon_{\infty}(\mathcal{C}) \simeq
  \lim_{n \to \infty} \Alg_{\mathbb{E}_{n}}(\mathcal{C})$. Since the
  functors in the diagram come from maps of $\infty$-operads, we can identify
  the right-hand side with $\Alg_{\colim_{n \to
      \infty}\mathbb{E}_{n}}(\mathcal{C})$. But by
  \cite{HA}*{Corollary 5.1.1.5} the colimit $\colim_{n \to
      \infty}\mathbb{E}_{n}$ is the commutative $\infty$-operad
    $\mathbb{E}_{\infty}$. Putting these equivalences together now gives a natural equivalence
    $\Mon_{\infty}(\mathcal{C}) \simeq \Alg_{\mathbb{E}_{\infty}}(\mathcal{C}^{\times})$.  
\end{proof}
Given these natural equivalences, we will allow ourselves to refer to $n$-
and $\infty$-monoidal (complete) $k$-fold Segal spaces as
\emph{$\mathbb{E}_{n}$-monoidal} and \emph{symmetric monoidal}
(complete) $k$-fold Segal spaces.

\section{Adjoints and Duals in Iterated Segal Spaces}\label{sec:adj}
In this section we first review the notions of $(\infty,k)$-categories
with adjoints and (symmetric) monoidal $(\infty,k)$-categories with
duals from \cite{LurieCob}, and then extend these notions to
$(\infty,k)$-categories internal to an \itopos{}. We begin by
recalling some key facts about adjunctions in $(\infty,2)$-categories
due to Riehl and Verity:
\begin{defn}
  Let $\Adj$ denote the generic adjunction, i.e.\ the universal
  2-category containing an adjunction between two 1-morphisms. An
  explicit description of $\Adj$ can be found in
  \cite{RiehlVerityAdj}*{\S 4}. We will think of $\Adj$ as a 2-fold
  Segal space via the nerve functor from 2-categories to
  2-fold Segal spaces. An \emph{adjunction} in a (complete)
  2-fold Segal space $\mathcal{C}$ is then a map of 2-fold Segal
  spaces $\Adj \to \mathcal{C}$. If $\mathcal{C}$ is a complete 2-fold Segal
  space, we write $\Adj(\mathcal{C}) := \Map(\Adj, \mathcal{C})$ for
  the space of adjunctions in $\mathcal{C}$.
\end{defn}

\begin{thm}[\cite{RiehlVerityAdj}*{Theorem 5.3.9}]
  Every adjunction in the homotopy 2-category of an
  $(\infty,2)$-category extends to an adjunction in the
  $(\infty,2)$-category. In particular, a 1-morphism in an
  $(\infty,2)$-category has a (left or right) adjoint \IFF{} it has
  one in the homotopy 2-category.
\end{thm}

\begin{defn}
  More or less keeping the notation of \cite{RiehlVerityAdj}, among
  the data defining the  $(\infty,2)$-category $\Adj$ we have:
  \begin{itemize}
  \item two objects $+$ and $-$,
  \item 1-morphisms $\mathfrak{f} \colon - \to +$ (the left adjoint) and $\mathfrak{g}
    \colon + \to -$ (the right adjoint),
  \item 2-morphisms $\mathfrak{u} \colon \id_{+} \to \mathfrak{g}\mathfrak{f}$ (the unit) and $\mathfrak{c}
    \colon \mathfrak{f}\mathfrak{g} \to \id_{-}$ (the counit), satisfying the triangle identities.
  \end{itemize}
\end{defn}

\begin{thm}[\cite{RiehlVerityAdj}*{Theorem 5.4.22}]\label{thm:RVadjtrunc}
  Suppose $\mathcal{C}$ is a complete 2-fold Segal space. The maps
  $\mathfrak{f}^{*}$ and $\mathfrak{g}^{*} \colon \Adj(\mathcal{C}) \to \Mor_{1}(\mathcal{C})$ sending
  an adjunction in $\mathcal{C}$ to the left and right adjoint,
  respectively, are
  (-1)-connected, i.e.\ their fibres are either empty or contractible.
\end{thm}

\begin{remark}
  The results of Riehl and Verity are proved in the context of
  categories strictly enriched in simplicial sets equipped with the
  Joyal model structure. Reformulating these theorems in terms of
  complete 2-fold Segal spaces is justified because these two models
  of $(\infty,2)$-categories are equivalent by the unicity theorem of
  Barwick and Schommer-Pries~\cite{BarwickSchommerPriesUnicity}. (An
  explicit equivalence can also be obtained by combining Theorem 5.9 and
  Corollary 7.21 of \cite{enrcomp}.)
\end{remark}

Now we recall what it means for an $(\infty,n)$-category to \emph{have
  adjoints}:
\begin{defn}
  Suppose $\mathcal{C}$ is a (complete) $n$-fold Segal space with $n >
  1$. We say that $\mathcal{C}$ \emph{has adjoints for 1-morphisms} if
  every 1-morphism in the homotopy 2-category of $\mathcal{C}$ has a
  left and a right adjoint. Equivalently, $\mathcal{C}$ has adjoints
  for 1-morphims if the maps \[\mathfrak{f}^{*},\mathfrak{g}^{*} \colon \Adj(u_{(\infty,2)}\mathcal{C}) \to
  \Mor_{1}(u_{(\infty,2)}\mathcal{C})\] are both equivalences.
\end{defn}

\begin{defn}\label{defn:hasadjsegspace}
  Suppose $\mathcal{C}$ is a (complete) $n$-fold Segal space with $n >
  1$. For $1 < k < n$ we say that $\mathcal{C}$ \emph{has adjoints for
    k-morphisms} if for all objects $X,Y$ of $\mathcal{C}$ the
  $(n-1)$-fold Segal space $\mathcal{C}(X,Y)$ has adjoints for
  $(k-1)$-morphims. We say that $\mathcal{C}$ \emph{has adjoints} if
  it has adjoints for $k$-morphims for all $k =1,\ldots, n-1$.
\end{defn}

\begin{remark}
  To see that a not necessarily complete $n$-fold Segal space
  $\mathcal{C}$ has adjoints, it is not necessary to complete it:
  Whether $\mathcal{C}$ has adjoints for 1-morphisms only depends on
  the homotopy $2$-category, which is easy to describe without
  completing $\mathcal{C}$. Moreover, by \cite{nmorita}*{Lemma 5.50}
  the mapping $(n-1)$-fold Segal spaces in the completion of
  $\mathcal{C}$ are the completions of the mapping $(n-1)$-fold Segal
  spaces of $\mathcal{C}$, so by induction we do not need to complete
  to see that $\mathcal{C}$ has adjoints for $k$-morphisms also for $k
  > 1$.
\end{remark}

\begin{defn}
  We say that a monoidal $n$-fold Segal space
  $\mathcal{C}^{\otimes}$ \emph{has duals} if $\mathcal{C}$ has
  adjoints when regarded as an $(n+1)$-fold Segal space. We also say a
  symmetric monoidal or $\mathbb{E}_{k}$-monoidal $n$-fold Segal space
  \emph{has duals} if the underlying monoidal $n$-fold Segal space has
  duals.
\end{defn}

\begin{lemma}\label{lem:gplikedualspace}
  We may regard a space $X$ as an $n$-fold Segal space for any $n$ by
  taking the constant functor with value $X$. If $X$ is an associative
  monoid in $\mathcal{S}$ (or in other words an
  $A_{\infty}$-space), then $X$ has duals \IFF{} $X$ is grouplike,
  i.e.\ under the induced multiplication the monoid $\pi_{0}X$ is a group.
\end{lemma}
\begin{proof}
  It suffices to check that the homotopy 1-category of $X$, equipped with
  the induced monoidal structure, has duals. But this is just the
  fundamental 1-groupoid of $X$, and an object of a
  monoidal groupoid has a dual \IFF{} it has an inverse.
\end{proof}

We now prove a characterization of $n$-fold Segal spaces with adjoints
that will be useful later:
\begin{propn}\label{propn:hasadjointseqt}
  Suppose $\mathcal{C}$ is a complete $n$-fold Segal space. Then
  $\mathcal{C}$ has adjoints for $k$-morphisms for any $2 < k < n$
  \IFF{} for every map $\phi \colon X \to \Ob(\mathcal{C})^{\times 2}$
  in $\mathcal{S}$, the complete $(n-1)$-fold Segal space
  $\mathcal{C}_{\phi}$, defined by the pullback square
  \nolabelcsquare{\mathcal{C}_{\phi}}{\mathcal{C}_{1}}{X}{\mathcal{C}_{0}^\times
    2} in $(n-1)$-fold Segal spaces, has adjoints for
  $(k-1)$-morphisms.
\end{propn}

The proof depends on the following observation:
\begin{lemma}\label{lem:hasadjfamily}
  Suppose given a morphism of $n$-fold Segal spaces $\mathcal{C} \to
  X$, where $X$ is constant. If all the fibres $\mathcal{C}_{x}$ for
  $x \in X$ have adjoints for $k$-morphisms, then so does $\mathcal{C}$.
\end{lemma}
\begin{proof}
  To prove this, we induct on $n$. For $n = 1$, there is nothing to
  prove, so we may suppose that the statement is true for $(n-1)$-fold
  Segal spaces for all $k = 1,\ldots,n-1$. 
  
  We first consider the case $k = 1$. Since $\Adj(X) \simeq
  \Mor_{1}(X) \simeq X$, we have a commutative diagram
  \nolabelopctriangle{\Adj(\mathcal{C})}{\Mor_{1}(\mathcal{C})}{X.}
  Since the functors $\Adj(\blank)$ and $\Mor_{1}(\blank)$
  preserve limits, the induced map on fibres over $p \in X$ can be
  identified with $\Adj(\mathcal{C}_{p}) \to
  \Mor_{1}(\mathcal{C}_{p})$. By assumption this is an equivalence for
  all $p \in X$, and so $\Adj(\mathcal{C}) \to \Mor_{1}(\mathcal{C})$ is also an
  equivalence, i.e.\ $\mathcal{C}$ has adjoints for $1$-morphisms.

  For $k > 2$, we must show that $\mathcal{C}(c,d)$ has adjoints for
  $(k-1)$-morphisms for all $c,d \in \Ob(\mathcal{C})$. But by
  Lemma~\ref{lem:CtoXmaps} there is a map $\mathcal{C}(c,d) \to
  \Omega_{\pi(c),\pi(d)}X$ whose fibres are mapping $(n-1)$-fold Segal
  spaces in the fibres of $\pi$, and so have adjoints for
  $(k-1)$-morphisms. The result therefore holds by the inductive
  hypothesis.
\end{proof}

\begin{proof}[Proof of Proposition~\ref{propn:hasadjointseqt}]
  One direction is obvious: If $\mathcal{C}_{\phi}$ has adjoints for
  $(k-1)$-morphisms for every map $\phi$, then in particular this is
  true for the $(n-1)$-fold Segal spaces $\mathcal{C}(x,y)$ for all
  objects $x,y$, so $\mathcal{C}$ has adjoints for $k$-morphisms.

  For the other direction we must show that if $\mathcal{C}$ has
  adjoints for $k$-morphisms then $\mathcal{C}_{\phi}$ has adjoints
  for $(k-1)$-morphisms for all $\phi$. By
  Lemma~\ref{lem:hasadjfamily}, to see this it suffices to show that
  the fibres of the map $\mathcal{C}_{\phi} \to X$ have adjoints for
  $(k-1)$-morphisms. But the fibre of this map at $p \in X$ is
  $\mathcal{C}(a,b)$ where $\phi(p) \simeq (a,b)$, which by assumption
  has adjoints for $k$-morphisms.
\end{proof}

An advantage of the characterization of
Proposition~\ref{propn:hasadjointseqt} is that this has a
straightforward generalization to other $\infty$-topoi. We introduce
this after some preliminary discussion:
\begin{defn}
  Suppose $\mathcal{X}$ is an \itopos, and let 
  \[r^{*} : \mathcal{S} \rightleftarrows \mathcal{X} : r_{*} \]
  denote the unique geometric morphism from the \icat{} of spaces. By
  Proposition~\ref{propn:CSSfunc}, this induces an adjunction
  \[L_{n}(r^{*})_{*} : \CSS^{n}(\mathcal{S}) \rightleftarrows
  \CSS^{n}(\mathcal{X}) : (r_{*})_{*} \]
  If $\mathcal{C}$ is a complete $2$-fold Segal object in
  $\mathcal{X}$, then an \emph{adjunction} in $\mathcal{C}$ is a
  functor $(r^{*})_{*}\Adj \to \mathcal{C}$. We write
  $\Adj(\mathcal{C}) \in \mathcal{X}$ for the mapping object
  $\MAP((r^{*})_{*}\Adj, \mathcal{C})$ in $\mathcal{X}$, defined in
  Definition~\ref{defn:MAP}. Similarly, if $\mathcal{C}$ is
  a complete $k$-fold Segal object in $\mathcal{X}$, we write
  $\Ob(\mathcal{C}) := \MAP((r^{*})_{*}C_{0}, \mathcal{C}) \simeq
  \mathcal{C}_{0,\ldots,0}$ and $\Mor_{n}(\mathcal{C}) :=
  \MAP((r^{*})_{*}C_{n}, \mathcal{C})$ for $n = 1,\ldots,k$.
\end{defn}

\begin{lemma}\label{lem:adjtrunctopos}
  Let $\mathcal{C}$ be a complete 2-fold Segal object in an \itopos{}
  $\mathcal{X}$. Then the morphisms $\mathfrak{f}^{*}$ and $\mathfrak{g}^{*} \colon
  \Adj(\mathcal{C}) \to \Mor_{1}(\mathcal{C})$ are $(-1)$-truncated.
\end{lemma}
\begin{proof}
  We must show that for any $X \in \mathcal{X}$, the map
  $\Map_{\mathcal{X}}(X, \Adj(\mathcal{C})) \to \Map_{\mathcal{X}}(X,
  \Mor_{1}(\mathcal{C}))$ is $(-1)$-truncated. But there is a natural equivalence
  \[ 
  \begin{split}
\Map_{\mathcal{X}}(X, \Adj(\mathcal{C})) & \simeq
  \Map_{\CSS^{2}(\mathcal{X})}(X \times (r^{*})_{*}\Adj, \mathcal{C})
  \simeq \Map_{\CSS^{2}(\mathcal{X})}((r^{*})_{*}\Adj,
  \mathcal{C}^{X}) \\ & \simeq \Map_{\CSS^{2}(\mathcal{S})}(\Adj,
  (r_{*})_{*}\mathcal{C}^{X}) \simeq
  \Adj((r_{*})_{*}\mathcal{C}^{X}),    
  \end{split}
\]
  and similarly $\Map_{\mathcal{X}}(X, \Mor_{1}(\mathcal{C})) \simeq
  \Mor_{1}((r_{*})_{*}\mathcal{C}^{X})$. Thus this follows by applying
  Theorem~\ref{thm:RVadjtrunc} to the complete 2-fold Segal spaces
  $(r_{*})_{*}\mathcal{C}^{X}$ for all $X \in \mathcal{X}$.
\end{proof}

\begin{defn}
  Suppose $\mathcal{C}$ is a complete $n$-fold Segal object in
  $\mathcal{X}$ with $n > 1$. We say that $\mathcal{C}$ \emph{has
    adjoints for 1-morphisms} if the maps $\mathfrak{f}^{*},\mathfrak{g}^{*} \colon
  \Adj(u_{(\infty,2)}\mathcal{C}) \to
  \Mor_{1}(u_{(\infty,2)}\mathcal{C})$ are both equivalences.
\end{defn}

\begin{defn}\label{defn:hasadjitopos}
  Suppose $\mathcal{C}$ is a complete $n$-fold Segal object in
  $\mathcal{X}$ with $n > 1$. For $1 < k < n$ we say that
  $\mathcal{C}$ \emph{has adjoints for k-morphisms} if for all maps
  $\phi \colon X \to \Ob(\mathcal{C})^{\times 2}$ in $\mathcal{X}$,
  the complete $(n-1)$-fold Segal object $\mathcal{C}_{\phi}$, defined
  by the pullback square
  \nolabelcsquare{\mathcal{C}_{\phi}}{\mathcal{C}_{1}}{X}{\mathcal{C}_{0}^\times
    2} in $(k-1)$-fold Segal objects, has adjoints for $(k-1)$-morphisms.
  We say that $\mathcal{C}$ \emph{has adjoints} if
  it has adjoints for $k$-morphims for all $k =1,\ldots, n-1$.
\end{defn}

\begin{defn}
  If $\mathcal{C}$ is a (not necessarily complete) $n$-fold Segal
  object in $\mathcal{X}$, we say that $\mathcal{C}$ \emph{has adjoints} (for
  $k$-morphisms) if this is true of the completion $L\mathcal{C}$.
\end{defn}

\begin{defn}\label{defn:hasdualtopoi}
  We say that a monoidal complete $n$-fold Segal space $\mathcal{C}$
  \emph{has duals} if it has adjoints when regarded as an $(n+1)$-fold
  Segal space.  We also say a symmetric monoidal or
  $\mathbb{E}_{k}$-monoidal complete $n$-fold Segal object \emph{has
    duals} if the underlying monoidal complete $n$-fold Segal object
  has duals.
\end{defn}

\begin{propn}\label{propn:grouplike}
  We may regard an object $X \in \mathcal{X}$ as a complete $n$-fold Segal object for any $n$ by
  taking the constant functor with value $X$. If $X$ is an associative
  monoid in $\mathcal{X}$ then $X$ has duals \IFF{} $X$ is
  grouplike, i.e.\ it is a groupoid object in the sense of Definition~\ref{defn:gpdob}.
\end{propn}
\begin{proof}
  Write $\mathcal{C}$ for the associative monoid corresponding to $X$,
  regarded as an $(n+1)$-fold Segal object in $\mathcal{X}$. Then by
  Lemma~\ref{lem:adjtrunctopos}, the $(n+1)$-fold Segal object
  $\mathcal{C}$ has adjoints for 1-morphisms \IFF{} for every $Y \in
  \mathcal{X}$ the $(n+1)$-fold Segal space
  $(r_{*})_{*}\mathcal{C}^{Y}$ has adjoints for
  1-morphisms. Similarly, $\mathcal{C}$ is a groupoid object \IFF{}
  $(r_{*})_{*}\mathcal{C}^{Y}$ is a groupoid object for all $Y \in
  \mathcal{X}$. The result therefore follows from
  Lemma~\ref{lem:gplikedualspace}.
\end{proof}

\section{Full Dualizability for Iterated Spans}\label{sec:dual}
In this section we will show that $\Span_{k}(\mathcal{C})$ is
symmetric monoidal, and that all its objects are fully dualizable ---
in fact, we will show that $\Span_{k}(\mathcal{C})$ has duals.

\begin{propn}
  Suppose $\mathcal{C}$ is an \icat{} with finite limits. Then the
  $(\infty,k)$-category $\Span_{k}(\mathcal{C})$ is symmetric monoidal.
\end{propn}
\begin{proof}
  By Proposition~\ref{propn:spanmaps} we can identify
  $\Span_{k}(\mathcal{C})$ with the $(\infty,k)$-category
  $\Span_{k+1}(\mathcal{C})(*, *)$ of endomorphisms of $*$ in
  $\Span_{k+1}(\mathcal{C})$. The sequence $(\Span_{k}(\mathcal{C}),
  *)$ of pointed $k$-fold Segal
  spaces therefore defines an infinite
  delooping sequence $\Span_{\infty+k}(\mathcal{C})$ in $k$-fold
  complete Segal spaces. From this we can extract an $\infty$-fold
  monoid $\Omega^{\infty}\Span_{\infty+k}(\mathcal{C})$ in complete
  $k$-fold Segal spaces. By Proposition~\ref{propn:iopdcomp} this is
  equivalent to a symmetric monoidal structure on
  $\Span_{k}(\mathcal{C})$.
\end{proof}

\begin{remark}
  In the case $k = 1$, an explicit construction of this symmetric
  monoidal structure can also be found in \cite{BarwickMackey2}.
\end{remark}

\begin{lemma}\label{lem:adj1}
  Let $\mathcal{C}$ be an \icat{} with finite limits. For all $k \geq
  2$, the 1-morphisms in $\Span_{k}(\mathcal{C})$ have adjoints.
\end{lemma}
\begin{proof}
  It suffices to check this in the homotopy 2-category of
  $\Span_{k}(\mathcal{C})$. A 1-morphism $\phi \colon A \to B$ in
  $\Span_{k}(\mathcal{C})$ is a span \nolabelspandiag{X}{A}{B.}
  We will show
  that the reversed span $\bar{\phi}$ given by \nolabelspandiag{X}{B}{A} is a right
  adjoint to this, with unit $\eta \colon \id_{A} \to \bar{\phi}\phi$
  given by the span
  \spandiag{X}{A}{X \times_{B} X}{}{\Delta}
  over $A \times A$, and counit $\epsilon \colon \phi\bar{\phi} \to
  \id_{B}$ given by
  \spandiag{X}{X \times_{A}X}{B}{\Delta}{} over $B \times B$, where
  $\Delta$ denotes the relevant diagonal maps. To see this it suffices
  to check that the triangle equations hold up to homotopy. The
  2-morphism $\phi\eta \colon \phi \to \phi\bar{\phi}\phi$ is given by
  the span
  \spandiag{X \times_{B}X}{X}{X \times_{B} X \times_{A} X,}{\pi_{1}}{\id \times \Delta}
  and $\epsilon \phi$ is given by
  \spandiag{X \times_{A} X}{X \times_{B} X \times_{A} X}{X.}{\Delta
    \times \id}{\pi_{2}} The composite $\phi \to \phi$ of
  these two maps is therefore given by the pullback
  \[ (X \times_{B} X) \times_{(X \times_{B} X \times_{A} X)} (X
  \times_{A} X).\]
  We claim that this pullback is equivalent to the limit of the
  diagram
  
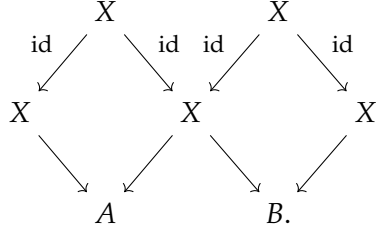
\[ %
\begin{tikzpicture} %
\matrix (m) [matrix of math nodes,row sep=2em,column sep=1.5em,text height=1.5ex,text depth=0.25ex] %
{  & X &   & X & \\
 X &   & X &   & X \\
   & A &   & B. &   \\ };
\path[->,font=\footnotesize] %
(m-1-2) edge node[above left] {$\id$} (m-2-1)
(m-1-2) edge node[above right] {$\id$} (m-2-3)
(m-1-4) edge node[above left] {$\id$} (m-2-3)
(m-1-4) edge node[above right] {$\id$} (m-2-5)
(m-2-1) edge (m-3-2)
(m-2-3) edge (m-3-2)
(m-2-3) edge (m-3-4)
(m-2-5) edge (m-3-4);
\end{tikzpicture}%
\]
To see this, take the right Kan extension of this diagram along the
inclusion into the category with shape
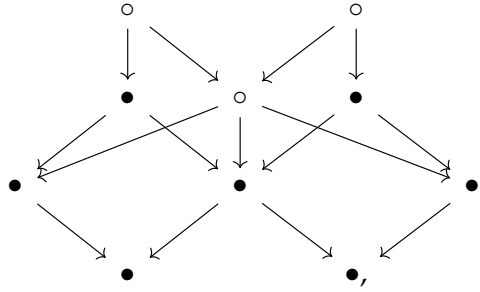
\[
\begin{tikzcd}
{} & \circ \arrow{d} \arrow{dr} & & \circ \arrow{d} \arrow{dl}& \\
 & \bullet \arrow{dl} \arrow{dr} & \circ \arrow{dll} \arrow{d} \arrow{drr} & \bullet\arrow{dl}\arrow{dr} & \\
\bullet \arrow{dr} & & \bullet \arrow{dl} \arrow{dr} & & \bullet \arrow{dl}\\
 & \bullet & & \bullet,
\end{tikzcd}
\]
where $\circ$ denotes the new objects. This produces a diagram
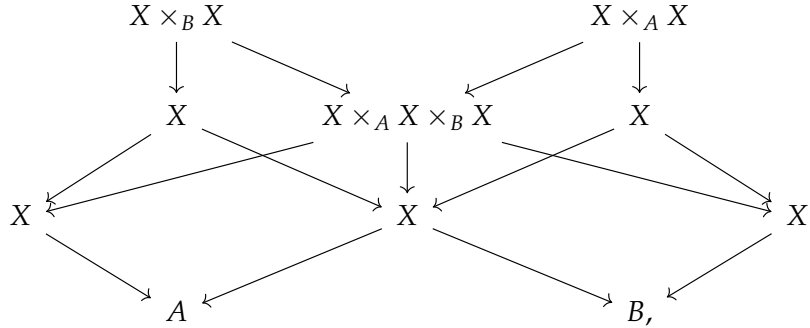
\[
\begin{tikzcd}
{} & X \times_{B} X \arrow{d} \arrow{dr} & & X \times_{A} X \arrow{d} \arrow{dl}& \\
 & X \arrow{dl} \arrow{dr} & X \times_{A} X \times_{B} X \arrow{dll} \arrow{d} \arrow{drr} &
 X \arrow{dl}\arrow{dr} & \\
X \arrow{dr} & & X \arrow{dl} \arrow{dr} & & X \arrow{dl}\\
 & A & & B,
\end{tikzcd}
\]
whose limit can be identified with $(X \times_{B} X) \times_{(X
  \times_{B} X \times_{A} X)} (X \times_{A} X)$ by a simple cofinality
argument. On the other hand, since right Kan extensions are
transitive, this must agree with the limit of the first diagram, which
can be identified with $X$ (again by an easy cofinality
argument). Thus $(\epsilon \phi) \circ (\phi \eta) \simeq \id_{\phi}$,
and the other triangle equivalence, $(\bar{\phi} \epsilon) \circ (\eta
\bar{\phi})\simeq \id_{\bar{\phi}}$, is proved similarly.
\end{proof}

\begin{thm}
  The $(\infty,k)$-category $\Span_{k}(\mathcal{C})$ has adjoints.
\end{thm}
\begin{proof}
  We prove this by induction on $k$. For $k = 1$, there is nothing to
  prove. Suppose we have shown that for all $\mathcal{C}$ the
  $(\infty,k-1)$-category $\Span_{k-1}(\mathcal{C})$ has adjoints.  We
  saw in Lemma~\ref{lem:adj1} that $\Span_{k}(\mathcal{C})$ has
  adjoints for 1-morphisms, and for every pair $X, Y$ of objects in
  $\mathcal{C}$ the $(\infty,k-1)$-category
  $\Span_{k}(\mathcal{C})(X,Y)$ can be identified with
  $\Span_{k-1}(\mathcal{C}_{/X\times Y})$ by
  Proposition~\ref{propn:spanmaps}, and so has adjoints by the
  inductive hypothesis. Thus $\Span_{k}(\mathcal{C})$ also has
  adjoints.
\end{proof}

\begin{cor}
  The (symmetric) monoidal $(\infty,k)$-category
  $\Span_{k}(\mathcal{C})$ has duals.
\end{cor}
\begin{proof}
  As a monoidal $(\infty,k)$-category, we may
  identify $\Span_{k}(\mathcal{C})$ with the endomorphism
  $(\infty,k)$-category $\Span_{k+1}(\mathcal{C})(*,*)$. Since $\Span_{k+1}(\mathcal{C})$ has adjoints, it follows that $\Span_{k}(\mathcal{C})$ has duals.
\end{proof}

Invoking the cobordism hypothesis in the form \cite{LurieCob}*{Theorem
  2.4.6}, we get:
\begin{cor}
  Suppose $\mathcal{C}$ is an \icat{} with finite limits. Then every object
  $C$ of $\mathcal{C}$ defines a framed $k$-dimensional TQFT
  $\mathcal{Z}^{k}_{C} \colon \txt{Bord}_{k}^{\txt{fr}} \to
  \Span_{k}(\mathcal{C})$, where $\txt{Bord}_{k}^{\txt{fr}}$ denotes
  the $(\infty,k)$-category of framed cobordisms.
\end{cor}

\begin{remark}\label{rmk:spanTQFT}
  For $\mathcal{D}$ an \icat{} with finite colimits, we write
  $\txt{Cospan}_{k}(\mathcal{D})$ for the $(\infty,k)$-category
  $\Span_{k}(\mathcal{D}^{\op})$. If $\txt{Bord}_{k}^{\txt{un}}$
  denotes the unoriented cobordism $(\infty,k)$-category, it is
  reasonable to expect that there is a symmetric monoidal ``forgetful
  functor'' $\txt{Bord}_{k}^{\txt{un}}\to
  \txt{Cospan}_{k}(\mathcal{S}^{\txt{fin}})$, where
  $\mathcal{S}^{\txt{fin}}$ is the \icat{} of finite
  CW-complexes. This would send a cobordism between manifolds with
  corners to the iterated cospan given by the inclusions of the
  underlying homotopy types of the incoming and outgoing boundaries and corners. If we assume this, we
  can give an explicit construction of the framed field theory
  $\mathcal{Z}_{C}^{k}$ valued in $\Span_{k}(\mathcal{C})$ associated
  to an object $C \in \mathcal{C}$:
  \begin{enumerate}[(1)]
  \item If $\mathcal{C}$ is an \icat{} with finite limits, then
    $\mathcal{C}$ is cotensored over $\mathcal{S}^{\txt{fin}}$. Thus given $C \in
    \mathcal{C}$ there is a functor $C^{(\blank)} \colon
    (\mathcal{S}^{\txt{fin}})^{\op} \to \mathcal{C}$. Since $\Span_{k}(\blank)$ is
    natural in limit-preserving functors, this induces a functor
    $C^{(\blank)} \colon \txt{Cospan}_{k}(\mathcal{S}^{\txt{fin}}) \to
    \Span_{k}(\mathcal{C})$ for all $k$.
  \item Identifying $\Span_{k}(\mathcal{C})$ as an endomorphism
    $(\infty,k)$-category in $\Span_{k+n}(\mathcal{C})$ for all $n$, we
    conclude that the functor $C^{(\blank)}$ is
    $\mathbb{E}_{n}$-monoidal for all $n$, hence symmetric monoidal.
  \item Composing, we get a symmetric monoidal functor
    $\widehat{\mathcal{Z}}^{k}_{C}\colon \txt{Bord}_{k}^{\txt{un}} \to
    \Span_{k}(\mathcal{C})$ that sends a cobordism $X$ to $C^{X}$ and
    the iterated span coming from the iterated boundary of $X$.
  \item By the cobordism hypothesis, the composite of
    $\widehat{\mathcal{Z}}^{k}_{C}$ with the forgetful functor
    $\txt{Bord}_{k}^{\txt{fr}} \to \txt{Bord}_{k}^{\txt{un}}$ is the unique
    symmetric monoidal functor $\txt{Bord}_{k}^{\txt{fr}} \to
    \Span_{k}(\mathcal{C})$ that sends the point to $C$, hence it must
    be equivalent to $\mathcal{Z}^{k}_{C}$.
  \end{enumerate}
  This construction would also allow us to understand the $O(k)$-action on
  the space $\iota \mathcal{C}$ of objects of
  $\Span_{k}(\mathcal{C})$: This action is given on the space of TQFTs
  by acting by $O(k)$ on the framings in
  $\txt{Bord}_{k}^{\txt{fr}}$. Since we know all the framed TQFTs
  factor through the forgetful functor $\txt{Bord}_{k}^{\txt{fr}} \to
  \txt{Cospan}_{k}(\mathcal{S})$, which is $O(k)$-equivariant with
  respect to the trivial action on the target, we see that the action
  on $\iota \mathcal{C}$ is trivial. As a consequence, we can also
  classify other kinds of TQFTs in $\Span_{k}(\mathcal{C})$, since by
  \cite{LurieCob}*{Theorem 2.4.18} these are determined by
  $O(k)$-equivariant maps to $\iota \mathcal{C}$. For example, the
  space of unoriented field theories is equivalent to $\Map(BO(n),
  \mathcal{C})$ --- when $\mathcal{C}$ is the \icat{} $\mathcal{S}$
  of spaces, this is precisely the $\infty$-groupoid of spaces
  equipped with an $n$-dimensional vector bundle.

  Actually constructing such a forgetful functor from cobordisms to
  cospans would, however, depend on the details of a construction of
  $\txt{Bord}_{k}^{\txt{un}}$, such as that of Calaque and
  Scheimbauer~\cite{CalaqueScheimbauerCob}, and we will not attempt to
  do so here.
\end{remark}

\section{Full Dualizability for Iterated Spans with Local
  Systems}\label{sec:duallocsys}
In this section we consider dualizability for iterated spans with
local systems. We'll prove that if $\mathcal{X}$ is an \itopos{} and
$\mathcal{C}$ is a complete $k$-fold Segal object in $\mathcal{X}$,
then a symmetric monoidal structure on $\mathcal{C}$ induces one on
$\Span_{k}(\mathcal{X}; \mathcal{C})$. Moreover, we will show that if
$\mathcal{C}$ has duals then the same is true of $\Span_{k}(\mathcal{X}; \mathcal{C})$.
\begin{propn}
  Suppose $\mathcal{C}$ is a symmetric monoidal
  complete $k$-fold Segal object in an \itopos{} $\mathcal{X}$. Then
  the $(\infty,k)$-category
  $\Span_{k}(\mathcal{X}; \mathcal{C})$ is symmetric monoidal.
\end{propn}
\begin{proof}
  Since $\mathcal{C}$ is symmetric monoidal, we can choose a sequence of
  ``deloopings'' $(\mathcal{C}_{i}, c_{i})$ such that $\mathcal{C}_{0} = \mathcal{C}$
  and $\mathcal{C}_{i} \simeq \mathcal{C}_{i+1}(c_{i+1},c_{i+1})$,
  i.e.\ an infinite delooping sequence in complete
  $k$-fold Segal objects. By
  Proposition~\ref{propn:Famkmaps}, we can then identify
  $\Span_{k+i}(\mathcal{X}; \mathcal{C}_{i})$ with the mapping
  $(\infty,k+i)$-category \[\Span_{k+i+1}(\mathcal{X};
  \mathcal{C}_{i+1})(x_{i+1},x_{i+1})\] in $\Span_{k+i+1}(\mathcal{X};
  \mathcal{C}_{i+1})$, where the object $x_{i+1}$ is the map $* \to
  \Ob(\mathcal{C}_{i+1})$ corresponding to the object $c_{i+1}$. Thus
  we get an infinite delooping sequence
  $\Span_{\infty+k}(\mathcal{X}; \mathcal{C})$, from which we can extract an
  $\infty$-fold monoid $\Omega^{\infty}\Span_{\infty+k}(\mathcal{X}; \mathcal{C})$
  in complete $k$-fold Segal spaces. By
  Proposition~\ref{propn:iopdcomp} this is equivalent to a symmetric
  monoidal structure on $\Span_{k}(\mathcal{X}; \mathcal{C})$.
\end{proof}

\begin{propn}\label{propn:locsysoneadj}
  Suppose $\mathcal{C}$ is a complete $k$-fold Segal object in
  $\mathcal{X}$ that has adjoints for 1-morphisms. Then
  $\Span_{k}(\mathcal{X}; \mathcal{C})$ has adjoints for 1-morphisms.
\end{propn}
\begin{proof}
  Suppose given a 1-morphism in $\Span_{k}(\mathcal{X}; \mathcal{C})$, i.e.\ a span
  $A \from X \to B$ in $\mathcal{X}$ equipped with a map to the span
  $\Ob(\mathcal{C}) \from \Mor_{1}(\mathcal{C}) \to
  \Ob(\mathcal{C})$. We will show that a right adjoint to
  this morphism is given by $B \from X \to A$, now with $X$ equipped
  with the map \[X \to \Mor_{1}(\mathcal{C}) \xto{(\mathfrak{f}^{*})^{-1}}
  \Adj(\mathcal{C}) \xto{\mathfrak{g}^{*}} \Mor_{1}(\mathcal{C}),\] which
  interchanges the source and target of 1-morphisms in $\mathcal{C}$.

  The unit for the adjunction is given by the span $A \from X \to X
  \times_{B} X$ over $A \times A$, where the map $X \to
  \Mor_{2}(\mathcal{C})$ is the composite
  \[ X \to \Mor_{1}(\mathcal{C}) \xto{(\mathfrak{f}^{*})^{-1}} \Adj(\mathcal{C})
  \xto{\mathfrak{u}^{*}} \Mor_{2}(\mathcal{C}) \]
  and the counit by $B \from X \to X \times_{A} X$, where $X$ is now
  equipped with 
  \[ X \to \Mor_{1}(\mathcal{C}) \xto{(\mathfrak{f}^{*})^{-1}} \Adj(\mathcal{C})
  \xto{\mathfrak{c}^{*}} \Mor_{2}(\mathcal{C}). \]
  The triangle identities for the adjunction then follow by combining
  the proof of Lemma~\ref{lem:adj1} with the homotopies coming from
  the triangle identities for the generic adjunction. Thus all
  1-morphisms in $\Span_{k}(\mathcal{X}; \mathcal{C})$ have right adjoints. To see
  that they also have left adjoints, we simply interchange the roles
  of the morphisms $\mathfrak{f}^{*}$ and $\mathfrak{g}^{*}$ above.
\end{proof}

\begin{thm}\label{thm:locsyshasadj}
  Suppose $\mathcal{C}$ is a complete $k$-fold Segal object in
  $\mathcal{X}$ that has adjoints. Then $\Span_{k}(\mathcal{X};
  \mathcal{C})$ has adjoints.
\end{thm}
\begin{proof}
  We will show that if $\mathcal{C}$ has adjoints for $i$-morphisms
  then $\Span_{k}(\mathcal{X}; \mathcal{C})$ also has adjoints for
  $i$-morphisms. The case $i = 1$ was proved in
  Proposition~\ref{propn:locsysoneadj}. Suppose $i > 1$, then we must
  show that $\Span_{k}(\mathcal{X}; \mathcal{C})(\xi, \eta)$ has
  adjoints for $(i-1)$-morphisms for all $\xi, \eta \in
  \Span_{k}(\mathcal{X}; \mathcal{C})$. By
  Proposition~\ref{propn:Famkmaps}, this $(k-1)$-fold Segal space can
  be identified with $\Span_{k-1}(\mathcal{X};
  \mathcal{C}_{\xi,\eta})$, and by definition (or by
  Proposition~\ref{propn:hasadjointseqt} in the case of spaces)
  $\mathcal{C}_{\xi,\eta}$ has adjoints for $(i-1)$-morphisms if
  $\mathcal{C}$ has adjoints for $i$-morphisms. Thus by induction we
  see that $\Span_{k}(\mathcal{X}; \mathcal{C})$ has adjoints for
  $i$-morphisms.
\end{proof}

\begin{cor}
  Suppose $\mathcal{C}$ is a (symmetric) monoidal complete $k$-fold
  Segal object in $\mathcal{X}$ that has duals. Then the (symmetric)
  monoidal $(\infty,k)$-category $\Span_{k}(\mathcal{X}; \mathcal{C})$
  also has duals.
\end{cor}
\begin{proof}
  Since $\mathcal{C}$ is monoidal, by
  Definition~\ref{defn:hasdualtopoi} there is a pointed complete $(n+1)$-fold
  Segal object $(\mathcal{C}^{\otimes}, *)$ with adjoints such that
  $\mathcal{C}$ is the endomorphism $n$-fold Segal object
  $\mathcal{C}^{\otimes}(*, *)$. Then $\Span_{k}(\mathcal{X};
  \mathcal{C})$ is the endomorphism $(\infty,n)$-category
  $\Span_{k}(\mathcal{X}; \mathcal{C}^{\otimes})(x,x)$, where $x$ is
  the object $* \to \Ob(\mathcal{C}^{\otimes})$ corresponding to the
  base point. By Theorem~\ref{thm:locsyshasadj} the
  $(\infty,n+1)$-category $\Span_{k}(\mathcal{X};
  \mathcal{C}^{\otimes})$ has adjoints, hence the monoidal
  $(\infty,n)$-category $\Span_{k}(\mathcal{X}; \mathcal{C})$ has duals.
\end{proof}

Invoking the cobordism hypothesis, we get:
\begin{cor}
  Suppose $\mathcal{C}$ is a symmetric monoidal complete $k$-fold
  Segal object in $\mathcal{X}$ that has duals. Every morphism $\phi
  \colon X \to \Ob(\mathcal{C})$ in $\mathcal{X}$ defines a framed
  $k$-dimensional TQFT
  $\mathcal{Z}^{k}_{\phi} \colon \txt{Bord}_{k}^{\txt{fr}} \to
  \Span_{k}(\mathcal{X}; \mathcal{C})$, where $\txt{Bord}_{k}^{\txt{fr}}$ denotes
  the $(\infty,k)$-category of framed bordisms.
\end{cor}

\begin{ex}\label{ex:gplike}
  Suppose $A$ is a grouplike $E_{\infty}$-algebra in an \itopos{}
  $\mathcal{X}$. Then by Proposition~\ref{propn:grouplike} we may
  regard $A$ as a symmetric monoidal $n$-fold complete Segal object
  with duals in $\mathcal{X}$ for any $n$, and so we get for every $n$
  a symmetric monoidal $(\infty,n)$-category $\Span_{n}(\mathcal{X};
  A)$. The underlying $(\infty,n)$-category of this is just
  $\Span_{n}(\mathcal{X}_{/A})$, but the symmetric monoidal structure
  is not that coming from the Cartesian product in $\mathcal{X}_{/A}$
  (i.e.\ the fibre product over $A$); instead, the tensor product of
  two maps $X,Y \to A$ is the product $X \times Y$ equipped with the
  composite map $X \times Y \to A \times A \to A$ where the second map
  is the multiplication in $A$. Similarly, the unit for the symmetric
  monoidal structure is the unit map $* \to A$, and the dual of an
  object $X \to A$ is the composite $X \to A \to A$ where the second
  map is the inverse mapping for $A$.
\end{ex}

\begin{remark}
  According to \cite{LurieCob}*{Proposition 3.2.7}, the \icat{}
  $\Span_{k}(\mathcal{S}; \mathcal{C})$, where $\mathcal{C}$ is a
  complete $k$-fold Segal space, should have a universal
  property. Namely, if $\mathcal{B}$ is a symmetric monoidal
  $(\infty,k)$-category equipped with a symmetric monoidal functor to
  $\Span_{k}(\mathcal{S})$, then the space of symmetric monoidal
  functors $\mathcal{B} \to \Span_{k}(\mathcal{S}; \mathcal{C})$ over
  $\Span_{k}(\mathcal{S})$ should be naturally equivalent to the space
  of symmetric monoidal functors from the pullback $\mathcal{B}
  \times_{\Span_{k}(\mathcal{S})} \Span_{k}(\mathcal{S}_{*})$ to
  $\mathcal{C}$. Moreover, \cite{LurieCob}*{Proposition 3.2.6} gives a
  description of the general cobordism $(\infty,k)$-category
  $\txt{Bord}_{k}^{(X,\zeta)}$ (where $\zeta$ is a map of spaces $X
  \to BO(k)$) as a pullback of this form, namely
  \nolabelcsquare{\txt{Bord}_{k}^{(X,\zeta)}}{\Span_{k}(\mathcal{S}_{*})}{\txt{Bord}_{k}^{\txt{un}}}{\Span_{k}(\mathcal{S})}
  where the bottom map is the unoriented TQFT determined by
  $\zeta$. If we assume this, as well as the consequences of a
  hypothetical forgetful functor from cobordisms to cospans discussed
  in Remark~\ref{rmk:spanTQFT}, we would get much more information
  about field theories valued in $\Span_{k}(\mathcal{S};
  \mathcal{C})$. For example, we would be able to describe the
  $O(k)$-action on the space of objects: the forgetful functor
  $\Span_{k}(\mathcal{S}; \mathcal{C}) \to \Span_{k}(\mathcal{S})$
  induces an $O(k)$-equivariant map from framed field theories valued
  in $\Span_{k}(\mathcal{S}; \mathcal{C})$ to those valued in
  $\Span_{k}(\mathcal{S})$. Since the $O(k)$-action on the target is
  trivial, this would mean that in the source $O(k)$ can only act on the
  fibres of this map. The fibre at $X \in \mathcal{S}$ can be
  identified with the space of $(X,\zeta)$-field theories valued in
  $\mathcal{C}$, where $\zeta$ is the map $X \to * \to BO(k)$;
  invoking the cobordism hypothesis it would follow that the fibre is
  $\Map(X, \Ob(\mathcal{C}))$, with the obvious $O(k)$-action induced
  from that on $\Ob(\mathcal{C})$.
\end{remark}

\section{The $(\infty,1)$-Category of Lagrangian Correspondences}\label{sec:lag}
In this section we will use the theory of symplectic derived stacks
and Lagrangian morphisms developed by Pantev, To\"{e}n, Vaquié, and
Vezzosi~\cite{PTVVShiftedSympl} to construct an \icat{}
$\Lag_{(\infty,1)}^{n}$ of $n$-symplectic derived stacks, with
morphisms given by Lagrangian correspondences between them. This
construction is based on ideas of Calaque, who describes the
underlying homotopy category in \cite{CalaqueTFT}. We begin by briefly
recalling the setup for derived stacks and (closed) $p$-forms; for
full details we refer to \cite{HAG2} and \cite{PTVVShiftedSympl}.

\begin{defn}
  Let $k$ be a field of characteristic 0. We write
  $\txt{cdga}_{k}^{\leq 0}$ for the category of non-positively graded
  commutative algebras in cochains over $k$, equipped with the
  usual model structure, and $\dAff^{\op}$ for the associated
  \icat{}. We may equip this \icat{} with an \emph{étale topology}, as
  described in \cite{HAG2}, and we write $\dSt :=
  \txt{Sh}_{\txt{ét}}(\dAff)$ for the associated (very large) \itopos{} of
  \'{e}tale sheaves of (large) spaces.
\end{defn}

\begin{remark}
  It is not necessary to take $k$ to be a field of characteristic
  zero. However, for more general rings commutative differential
  graded algebras are often not the most appropriate notion of
  ``derived rings'', the more useful notions being simplicial
  commutative algebras and (connective) $E_{\infty}$-algebras. For a
  field of characteristic zero, however, all three notions coincide.
\end{remark}

\begin{remark}
  For many purposes we do not want to consider arbitrary objects of
  $\dSt$, but only some subclass of ``geometric'' objects. The notion
  of \emph{derived Artin stacks} provides a good definition of such
  geometric stacks; roughly speaking they are derived stacks obtained
  as iterated realizations of smooth groupoids --- see \cite{HAG2} for
  details. In particular, derived Artin stacks always have cotangent
  complexes. As in \cite{PTVVShiftedSympl} we will implicitly add the
  assumption that all derived Artin stacks considered are locally of
  finite presentation, so that their cotangent complexes are
  dualizable quasicoherent sheaves. We write $\dStA$ for the full
  subcategory of $\dSt$ spanned by the derived Artin stacks locally of
  finite presentation --- by \cite{HAG2}*{Corollary 1.3.3.5} this is
  closed under finite limits in $\dSt$.
\end{remark}

\begin{defn}
  In \cite{PTVVShiftedSympl}, functors $\Omega^{p}$ and
  $\Omega^{p}_{\txt{cl}}$ from $\dAff^{\op}$ to the \icat{} of cochain
  complexes that take a derived affine scheme to its complex of
  $p$-forms and closed $p$-forms, respectively, are constructed and
  shown to be étale sheaves. We let $\mathcal{A}^{p}[n]$ and
  $\mathcal{A}^{p}_{\txt{cl}}[n]$ be the derived stacks (i.e.\ sheaves
  of spaces on $\dSt$) obtained from the shifts $\Omega^{p}[n]$ and
  $\Omega^{p}_{\txt{cl}}[n]$ via the Dold-Kan construction. If $X$ is
  a derived stack, we refer to $\mathcal{A}^{p}[n](X)$ as the space of
  \emph{$n$-shifted $p$-forms} on $X$. There is a ``forgetful'' map
  $\Omega^{p}_{\txt{cl}} \to \Omega^{p}$, which induces natural
  transformations $\mathcal{A}^{p}_{\txt{cl}}[n] \to
  \mathcal{A}^{p}[n]$, but the components are not in general
  monomorphisms (i.e.\ inclusions of a subset of the connected
  components).
\end{defn}

\begin{remark}
  Since $\mathcal{A}^{p}_{\txt{cl}}[n]$ comes from a sheaf of cochain
  complexes on $\dSt$, and hence a sheaf of spectra, we may regard it
  as a sheaf of grouplike $E_{\infty}$-spaces, i.e.\ a grouplike
  $E_{\infty}$-monoid in $\dSt$. Thus by Example~\ref{ex:gplike} there
  is a symmetric monoidal $(\infty,k)$-category $\Span_{k}(\dSt;
  \mathcal{A}^{p}_{\txt{cl}}[n])$ with duals for all $p, n$.
\end{remark}

\begin{defn}
  If $X$ is a derived Artin stack, an $n$-shifted 2-form $\omega \in
  \mathcal{A}^{2}[n](X)$ corresponds to a morphism
  $\Lambda^{2}\mathbb{T}_{X} \to \mathcal{O}_{X}[n]$ of quasi-coherent
  sheaves on $X$, where the tangent complex $\mathbb{T}_{X}$ is the
  dual of the cotangent complex $\mathbb{L}_{X}$. We say that $\omega$
  is \emph{non-degenerate} if the induced morphism $\mathbb{T}_{X} \to
  \mathbb{L}_{X}[n]$ is an equivalence, and write
  $\mathcal{A}^{2}_{\txt{nd}}[n](X)$ for the collection of components
  of $\mathcal{A}^{2}[n](X)$ corresponding to the non-degenerate
  $n$-shifted 2-forms.
\end{defn}

\begin{defn}
  An \emph{$n$-shifted symplectic form} on a derived Artin stack $X$ is a
  non-degenerate closed 2-form, i.e.\ an element of the pullback
  \[ \Sympl_{n}(X) := \mathcal{A}^{2}_{\txt{cl}}[n](X)
  \times_{\mathcal{A}^{2}[n](X)} \mathcal{A}^{2}_{\txt{nd}}[n](X),\]
  which is a subset of the connected components of
  $\mathcal{A}^{2}_{\txt{cl}}[n](X)$. An \emph{$n$-symplectic derived
   Artin stack} $(X, \omega)$ is a derived Artin stack $X$ equipped with an
  $n$-shifted symplectic form $\omega$.
\end{defn}

\begin{defn}
  Suppose $X$ is a derived stack, and $\omega$ is an $n$-shifted closed
  2-form on $X$. If $f \colon L \to X$ is a morphism of derived
  stacks, then an \emph{$\omega$-isotropic structure} on $f$ is a
  commutative square
  \csquare{L}{X}{*}{\mathcal{A}^{2}_{\txt{cl}}[n].}{f}{}{\omega}{0}
  Equivalently, it is a path from $0$ to the composite closed
  $n$-shifted 2-form $f^{*}\omega$ in
  $\mathcal{A}^{2}_{\txt{cl}}[n](L)$.
\end{defn}

\begin{defn}
  Suppose $(X, \omega)$ is an $n$-symplectic derived Artin stack, and
  $f \colon L \to X$ is a morphism of derived Artin stacks. Then an
  isotropic structure on $f$ induces a commutative square
  (see \cite{PTVVShiftedSympl}*{\S 2.2} for the details)
  \nolabelcsquare{\mathbb{T}_{L}}{f^{*}\mathbb{T}_{X}}{0}{\mathbb{L}_{L}[n]}
  of quasi-coherent sheaves on $L$. We say that the isotropic structure
  is \emph{Lagrangian} if this square is Cartesian.
\end{defn}

\begin{defn}
  Suppose $(X, \omega_{X})$ and $(Y, \omega_{Y})$ are $n$-shifted
  symplectic derived Artin stacks. A span \spandiag{L}{X}{Y}{f}{g} in
  $(\dStA)_{/\mathcal{A}^{2}_{\txt{cl}}[n]}$ induces a commutative
  square
  \nolabelcsquare{\mathbb{T}_{L}}{f^{*}\mathbb{T}_{X}}{g^{*}\mathbb{T}_{Y}}{\mathbb{L}_{L}[n]}
  of quasi-coherent sheaves on $L$. We say the span is a
  \emph{Lagrangian correspondence} if this square is Cartesian.
\end{defn}

\begin{remark}
  If we write $\overline{Y}$ for the derived stack $Y$ equipped with
  the \emph{negative} of the symplectic form of $Y$, which is also a
  symplectic form, we may identify spans of the form $X \xfrom{f} L
  \xto{g} Y$ over $\mathcal{A}^{2}_{\txt{cl}}[n]$ with isotropic
  morphisms $L \to X \times \overline{Y}$: Making the maps $f$ and $g$
  into a span over $\mathcal{A}^{2}_{\txt{cl}}[n]$ precisely
  corresponds to choosing a path from $f^{*}\omega_{X}$ to
  $g^{*}\omega_{Y}$ in the space of closed 2-forms on $L$, which is
  the same as giving a path from the zero form to $f^{*}\omega_{X} -
  g^{*}\omega_{Y}$. Under this equivalence, Lagrangian correspondences
  correspond to Lagrangian morphisms to $X \times \overline{Y}$: since
  quasicoherent sheaves on $L$ form a stable \icat{}, producing a
  pullback square
  \nolabelcsquare{\mathbb{T}_{L}}{f^{*}\mathbb{T}_{X}}{g^{*}\mathbb{T}_{Y}}{\mathbb{L}_{L}[n]}
  is the same thing as producing a fibre sequence
  \nolabelcsquare{\mathbb{T}_{L}}{f^*\mathbb{T}_{X} \oplus
    g^{*}\mathbb{T}_{Y}}{0}{\mathbb{L}_{L}[n],}
  where the right vertical map is the \emph{difference} of the two
  maps to $\mathbb{L}_{L}[n]$ in the first square.
\end{remark}

\begin{propn}[\cite{CalaqueTFT}*{Theorem 4.4}]\label{propn:lagcomp}
  Suppose $X$, $Y$, and $Z$ are $n$-symplectic derived Artin stacks,
  and $X \xfrom{f} K \xto{g} Y$ and $Y \xfrom{h} L \xto{k} Z$ are
  Lagrangian correspondences. Then the composite span \nolabelspandiag{K\times_{Y}L}{X}{Z}
  is also a Lagrangian correspondence.
\end{propn}
\begin{proof}
  A Cartesian square
  \csquare{N}{K}{L}{Y}{\phi}{\psi}{g}{h}
  induces a commutative diagram
  \[ %
\begin{tikzpicture} %
\matrix (m) [matrix of math nodes,row sep=3em,column sep=2.5em,text height=1.5ex,text depth=0.25ex] %
{ \mathbb{T}_{N} & \phi^{*}\mathbb{T}_{K} &
  \phi^{*}f^{*}\mathbb{T}_{X}  \\
  \psi^{*}\mathbb{T}_{L} & \phi^{*}g^{*}\mathbb{T}_{Y} &
  \phi^{*}\mathbb{L}_{K}[n] \\
  \psi^{*}k^{*}\mathbb{T}_{Z} & \psi^{*}\mathbb{L}_{L}[n] &
  \mathbb{L}_{N}[n]. \\};
\path[->,font=\footnotesize] %
(m-1-1) edge (m-1-2)
(m-1-2) edge (m-1-3)
(m-2-1) edge (m-2-2)
(m-2-2) edge (m-2-3)
(m-3-1) edge (m-3-2)
(m-3-2) edge (m-3-3)
(m-1-1) edge (m-2-1)
(m-2-1) edge (m-3-1)
(m-1-2) edge (m-2-2)
(m-2-2) edge (m-3-2)
(m-1-3) edge (m-2-3)
(m-2-3) edge (m-3-3)
;
\end{tikzpicture}%
\]%
Here the upper right square is Cartesian since $X \from K \to Y$ is a
Lagrangian correspondence, and the bottom left square is Cartesian
since $Y \from L \to Z$ is a Lagrangian correspondence. The top left
square is Cartesian since $N$ is the fibre product of $K$ and $L$ over
$Y$. Finally, since $Y$ is symplectic we have an equivalence
$\phi^{*}g^{*}\mathbb{T}_{Y} \simeq \phi^{*}g^{*}\mathbb{L}_{Y}[n]$,
and so we can identify the bottom right square with a shift of the
dual of the top left square. The bottom right square is therefore
coCartesian, but we're in a stable \icat{} so coCartesian and
Cartesian squares coincide. Thus the boundary square in the diagram is
also Cartesian, which by definition means that $X \from N \to Z$ is a
Lagrangian correspondence.
\end{proof}

\begin{defn}
  By Proposition~\ref{propn:lagcomp}, Lagrangian correspondences are
  closed under composition in the homotopy category of
  $\Span_{1}(\dSt; \mathcal{A}^{2}_{\txt{cl}}[n])$. We can therefore
  define the \icat{} $\Lag_{(\infty,1)}^{n}$ to be the subcategory of
  $\Span_{1}(\dSt; \mathcal{A}^{2}_{\txt{cl}}[n])$ whose objects are
  the $n$-symplectic derived Artin stacks and whose 1-morphisms are
  the Lagrangian correspondences between these.
\end{defn}

\begin{remark}
  The idea of considering symplectic derived stacks and Lagrangian
  correspondences as a subcategory of $\Span_{1}(\dSt;
  \mathcal{A}^{p}_{\txt{cl}}[n])$ is taken from
  \cite{SchreiberCFTviaCohHT}.
\end{remark}

\begin{lemma}
  The symmetric monoidal structure on $\Span_{1}(\dSt;
  \mathcal{A}^{2}_{\txt{cl}}[n])$ induces a symmetric monoidal
  structure on $\Lag_{(\infty,1)}^{n}$.
\end{lemma}
\begin{proof}
  To show that $\Lag_{(\infty,1)}^{n}$ inherits a symmetric monoidal
  structure, it suffices to prove that it contains the unit of the
  symmetric monoidal structure on $\Span_{1}(\dSt;
  \mathcal{A}^{2}_{\txt{cl}}[n])$, and that its objects and morphisms
  are closed under this. The unit is the map $* \to
  \mathcal{A}^{2}_{\txt{cl}}[n]$ corresponding to 0, which is
  obviously symplectic. If $X$ and $Y$ are $n$-symplectic derived
  Artin stacks, then their tensor product is the Cartesian product $X
  \times Y$ equipped with the sum of the symplectic forms on $X$ and
  $Y$, which is again symplectic. Finally, the tensor product of two
  Lagrangian correspondences is again their Cartesian product, which
  is Lagrangian with respect to the sum symplectic structures (since
  we just get the direct sum of the two Cartesian squares of
  quasi-coherent sheaves, which is again Cartesian).
\end{proof}

\begin{propn}
  With respect to the induced symmetric monoidal structure, all
  $n$-symplectic derived Artin stacks are dualizable in $\Lag_{(\infty,1)}^{n}$.
\end{propn}
\begin{proof}
  Let $(X, \omega)$ be an $n$-symplectic derived Artin stack. We must
  show that the dual of $X$ is also an $n$-symplectic derived Artin
  stack, and the evaluation and coevaluation maps, as described in the
  proof of Proposition~\ref{propn:locsysoneadj}, are Lagrangian
  correspondences. By Example~\ref{ex:gplike} the dual of $X$ is
  $\bar{X}$, meaning $X$ equipped with the negative $-\omega$ of its
  symplectic form. This is again symplectic, as the morphism
  $\mathbb{T}_{X} \to \mathbb{L}_{X}[n]$ induced by $-\omega$ is
  simply the negative of that induced by $\omega$, and so is also an
  equivalence. The coevaluation map is given by the span $* \from X
  \xto{\Delta} \bar{X} \times X$, where $\Delta$ is the diagonal and
  $\bar{X} \times X$ is equipped with the sum symplectic structure
  $(-\omega, \omega)$. The induced diagram of quasi-coherent sheaves
  on $X$ is \nolabelcsquare{\mathbb{T}_{X}}{\mathbb{T}_{X} \oplus
    \mathbb{T}_X}{0}{\mathbb{L}_{X}[n],} where the top horizontal map
  is $(-\id, \id)$. This is Cartesian \IFF{} the square
  \nolabelcsquare{\mathbb{T}_{X}}{\mathbb{T}_{X}}{\mathbb{T}_{X}}{\mathbb{L}_{X}[n]}
  is Cartesian, where the top horizontal and left vertical maps are
  both the identity, but this is true since $X$ is symplectic as this
  means the other two maps, which are also identical, are
  equivalences. Thus this span is a Lagrangian
  correspondence. The evaluation map $\bar{X} \times X
  \xfrom{\Delta} X \to *$ is likewise a Lagrangian correspondence by a
  similar argument.
\end{proof}

\begin{cor}
  Every $n$-sympectic derived Artin stack $X$ determines a framed
  1-dimensional TQFT \[\mathcal{Z}_{X} \colon \txt{Bord}_{1}^{\txt{fr}}
  \to \Lag_{(\infty,1)}^{n}.\]
\end{cor}

\end{document}